\documentclass[12pt,leqno,draft]{amsart}
\usepackage{amsfonts,amsthm,amssymb,amsmath,amscd}
\usepackage[all]{xy}

\def\N{{\mathbb N}}
\def\Z{{\mathbb Z}}
\def\Qq{{\mathbb Q}}
\def\R{{\mathbb R}}
\def\C{{\mathbb C}}
\def\H{{\mathfrak H}}

\def\sq{\hbox{\rlap{$\sqcap$}$\sqcup$}}
\def\qed{\ifmmode\sq\else{\unskip\nobreak\hfil
         \penalty50\hskip1em\null\nobreak\hfil\sq
         \parfillskip=0pt\finalhyphendemerits=0\endgraf}\fi}

\def\smat#1#2#3#4{\left(\begin{smallmatrix}#1&#2\\#3&#4\end{smallmatrix}\right)}

\newtheorem{theorem}{Theorem}
\newtheorem{lemma}[theorem]{Lemma}
\newtheorem{prop}[theorem]{Proposition}
\newtheorem{cor}[theorem]{Corollary}

\newtheorem{df}{Definition}

\numberwithin{theorem}{section}
\numberwithin{equation}{section}

\allowdisplaybreaks[1]

\title[
Lifting to Siegel modular forms of half-integral weight
]
{
Lifting from pairs of two elliptic modular forms
to Siegel modular forms of half-integral weight of
even degree
}
\author[S.~Hayashida]{Shuichi Hayashida
}
\date{\today}
\keywords{Siegel modular forms, Jacobi forms, Maass relation}
\subjclass[2010]{11F46 (primary), 11F37, 11F50 (secondary)}

\begin{document}

\begin{abstract}
The aim of this paper is to show lifts from pairs of two elliptic modular forms
to Siegel modular forms of half-integral weight of even degree under the assumption
that the constructed Siegel modular form is not identically zero.
The key of the proof is to introduce a certain generalization of the Maass relation
for Siegel modular forms of half-integral weight of general degree.
\end{abstract}

\maketitle

\section{Introduction}\label{s:introduction}

\subsection{}

Lifts from pairs of two elliptic modular forms to Siegel modular form
of half-integral weight of degree two have been conjectured
by~Ibukiyama and the author\cite{HI}.
In the present article we will give a partial answer for the conjecture in~\cite{HI}
and shall generalize these lifts as lifts from pairs of two elliptic modular forms
to Siegel modular forms of half-integral weight of
\textit{any even degree} (Theorem~\ref{th:main}).

The construction of the lift can be regarded as a half-integral weight version
of the Miyawaki-Ikeda lift. The Miyawaki-Ikeda lift has been shown by Ikeda~\cite{Ik2}.
In the present article we will give a proof to the fact that
certain constructed Siegel modular forms of half-integral weight
are eigenforms, if it does not identically vanish.
Moreover, we will compute the $L$-function of the constructed Siegel modular forms
of half-integral weight.
The key of the proof of the lift in the present article
is to introduce a generalized Maass relation
for Siegel modular forms of half-integral weight
(Theorem~\ref{thm:maass_e_half},~\ref{th:maass_relation_half_cusp}).
Generalized Maass relations are certain relations among Fourier-Jacobi coefficients
of Siegel modular forms
and are regarded as relations among Fourier coefficients.
Theorem~\ref{thm:maass_e_half} is a generalization of the Maass relation
for generalized Cohen-Eisenstein series, which is a Siegel modular form
of half-integral weight of general degree.
And Theorem~\ref{th:maass_relation_half_cusp} is a generalization
of the Maass relation for Siegel cusp forms of half-integral weight of odd degree.

\subsection{}
We explain our results more precisely.

We denote by $M_{k-\frac12}^+(\Gamma_0^{(n)}(4))$ the generalized plus-space
of weight $k-\frac12$ of degree $n$, which is a certain subspace
of Siegel modular forms of half-integral weight
and is a generalization of the Kohnen plus-space
(see \cite{Ib} or
\S\ref{ss:fj_expansion_half} for the definition of generalized plus-space).
Let $F \in M_{k-\frac12}^+(\Gamma_0^{(n)}(4))$
be an eigenform for any Hecke operators.
We put
\begin{eqnarray*}
  Q_{F,p}(z) &=& \prod_{i=0}^n (1-\mu_{i,p} z) (1 - \mu_{i,p}^{-1} z),
\end{eqnarray*}
where complex numbers
$\{\mu_{i,p}^{\pm}\}$ are $p$-parameters of $F$ introduced in~\cite{Zhu:euler}
if $p$ is an odd prime. If $p=2$, then we define $\{\mu_{i,2}^{\pm}\}$ by using
the isomorphism between generalized plus-space and the space of
Jacobi forms of index $1$.
We denote the modified Zhuravlev $L$-function by
\begin{eqnarray*}
  L(s,F) &:=& \prod_{p} Q_{F,p}(p^{-s+k-\frac32}).
\end{eqnarray*}
The Zhuravlev $L$-function is originally introduced in~\cite{Zhu:euler}
without the Euler $2$-factor,
which is a generalization of the $L$-function of elliptic modular forms
of half-integral weight introduced in~\cite{Shi}.

We denote by $S_{k-\frac12}^{+}(\Gamma_0^{(n)}(4))$
the space of Siegel cusp forms in $M_{k-\frac12}^{+}(\Gamma_0^{(n)}(4))$.
The following theorem is the main result of this article.
\newtheorem*{th:main}{Theorem \ref{th:main}}
\begin{th:main}
  Let $k$ be an even integer and $n$ be a non-negative integer.
  Let $h \in S_{k-n+\frac12}^+(\Gamma_0^{(1)}(4))$
  and $g \in S_{k-\frac12}^+(\Gamma_0^{(1)}(4))$
  be eigenforms for all Hecke operators.
  Then there exists a $\mathcal{F}_{h,g} \in S_{k-\frac12}^+(\Gamma_0^{(2n-2)})$.
  Under the assumption that $\mathcal{F}_{h,g}$ is not identically zero, then $\mathcal{F}_{h,g}$
  is an eigenform with the $L$-function which satisfies
  \begin{eqnarray*}
    L(s,\mathcal{F}_{h,g})
    &=&
    L(s,g) \prod_{i=1}^{2n-3} L(s-i,h).
  \end{eqnarray*}
\end{th:main}

By numerical computations of Fourier coefficients of $\mathcal{F}_{h,g}$
we checked that
$\mathcal{F}_{h,g}$ does not identically vanish for some $(n,k)$.
(See \S\ref{s:examples_non_vanishing} for the detail).

Remark that the above theorem was first conjectured by Ibukiyama and the author~\cite{HI}  in the case of $n=2$ not only for even integer $k$, but also for odd integer $k$.

The construction of $\mathcal{F}_{h,g}$ was suggested by T. Ikeda to the author,
which is given by a composition of three maps
and an inner product.
These three maps are
a Ikeda lift (Duke-Imamoglu-Ibukiyama-Ikeda lift), a map of the Fourier-Jacobi expansion and an isomorphism between
Jacobi forms of index $1$ and Siegel modular forms of half-integral weight.
In \S\ref{s:maass_relation_siegel_cusp} we will explain the detail of
the construction of $\mathcal{F}_{h,g}$.

To prove Theorem~\ref{th:main} we use a generalized
Maass relation for generalized Cohen-Eisenstein series
(Theorem~\ref{thm:maass_e_half}).
Once we obtain Theorem~\ref{thm:maass_e_half},
it is not so hard to show Theorem~\ref{th:main}.
The most part of this article is devoted to show Theorem~\ref{thm:maass_e_half}.
We now explain the generalized Maass relation
for generalized Cohen-Eisenstein series (Theorem~\ref{thm:maass_e_half}).

Let $k$ be an \textit{even} integer and $\mathcal{H}_{k-\frac12}^{(n+1)}$ be
the generalized Cohen-Eisenstein series of degree $n+1$ of weight $k-\frac12$
(see \S\ref{ss:CE_J*} for the definition of 
generalized Cohen-Eisenstein series).
The form $\mathcal{H}_{k-\frac12}^{(n+1)}$ is a certain Siegel modular form
of weight $k-\frac12$ of degree $n+1$.

For integer $m$,
we denote by $e_{k-\frac12,m}^{(n)}$ the $m$-th Fourier-Jacobi coefficient
of $\mathcal{H}_{k-\frac12}^{(n+1)}$:
\begin{eqnarray}\label{eq:df_fourier_jacobi_cohen_eisenstein}
 \mathcal{H}_{k-\frac12}^{(n+1)}\!\!\left(\begin{pmatrix} \tau & z \\ ^t z & \omega \end{pmatrix}\right)
&=&
 \sum_{\begin{smallmatrix} m \geq 0 \\ m \equiv 0,3 \!\!\! \mod 4  \end{smallmatrix}}
 e_{k-\frac12,m}^{(n)}(\tau,z)\, e^{2\pi \sqrt{-1} m \omega},
\end{eqnarray}
where $\tau \in \H_n$ and $\omega \in \H_1$, and where $\H_n$ denotes
the Siegel upper half space of degree $n$.
We denote by $J_{k-\frac12,m}^{(n)}$ the space of Jacobi forms of degree $n$
of weight $k-\frac12$ of index $m$ (cf. \S\ref{ss:def_jacobi_half_weight})
and denote by $J_{k-\frac12,m}^{(n)*}$ (cf. \S\ref{ss:CE_J*}) a certain subspace of
$J_{k-\frac12,m}^{(n)}$.
Then, the above form $e_{k-\frac12,m}^{(n)}$ belongs to $J_{k-\frac12,m}^{(n)*}$.
Because $\mathcal{H}_{k-\frac12}^{(n+1)}$ belongs to the generalized plus-space
$M_{k-\frac12}^{+}(\Gamma_0^{(n+1)}(4))$,
we can show that
the form $e_{k-\frac12,m}^{(n)}$ is identically zero unless $m \equiv 0$, $3 \!\! \mod 4$.

We denote by $M_k(\Gamma_{n+2})$ the space of Siegel modular forms of
weight $k$ of degree $n+2$ and denote by $J_{k,1}^{(n+1)}$ the space
of Jacobi forms of weight $k$ of index $1$ of degree $n+1$.
We denote by $E_k^{(n)} \in M_k(\Gamma_n)$ the Siegel-Eisenstein series
of weight $k$ of degree $n$
(cf. (\ref{df:siegel_eisenstein}) in \S\ref{s:fourier_matrix})
and by $E_{k,1}^{(n)} \in J_{k,1}^{(n)}$ 
the Jacobi-Eisenstein series of weight $k$ of index $1$ of degree $n$
(cf. (\ref{df:jacobi_eisenstein}) in \S\ref{s:fourier_matrix}).
The form $\mathcal{H}_{k-\frac12}^{(n+1)}$ is constructed from $E_{k,1}^{(n+1)}$.
The diagram of the above correspondence is

\[
 \xymatrix{
	  E_k^{(n+2)} \in M_k(\Gamma_{n+2})  \ar[d]^{\mbox{}}
     &
   \\
	  E_{k,1}^{(n+1)} \in J_{k,1}^{(n+1)}  
                                        \ar[r]_{\mbox{}}
     &  \mathcal{H}_{k-\frac12}^{(n+1)} \ar[d]^{\mbox{}}  \in M_{k-\frac12}^{+}(\Gamma_0^{(n+1)}(4))
   \\
     &
         \displaystyle{
            \left\{e_{k-\frac12,m}^{(n)}\right\}_{m
            }
                           }
        \in
           \!\!\!\!\!
    	\displaystyle{ \bigotimes_{m \equiv 0,3\!\!\!\!\! \mod 4}
               \!\!\!\!\!
             J_{k-\frac12,m}^{(n)*}}.
	                   }
\]

In \S\ref{ss:hecke_operators} (for any odd prime $p$) and in \S\ref{ss:hecke_p2} (for $p=2$)
we will introduce index-shift maps
$\tilde{V}_{\alpha,n-\alpha}(p^2)$ $(\alpha = 0, ..., n)$,
which are certain linear maps
from $J_{k-\frac12,m}^{(n)*}$ to the space of holomorphic functions on $\H_n \times \C^{(n,1)}$. 
If $p$ is odd then
$\tilde{V}_{\alpha,n-\alpha}(p^2)$ is a linear map from
$J_{k-\frac12,m}^{(n)*}$ to $J_{k-\frac12,mp^2}^{(n)}$.
These maps $\tilde{V}_{\alpha,n-\alpha}(p^2)$
are certain generalizations of the $V_l$-map in \cite[p.43]{EZ}
for half-integral weight of general degrees.
For any $\phi \in J_{k-\frac12,m}^{(n)}$ and for any integer $a$ we define
$(\phi|U_a)(\tau,z) := \phi(\tau,az)$.

The following theorem is a generalization of the Maass relation
for the generalized Cohen-Eisenstein series,
where we use the symbol
\begin{equation*}
\begin{split}
  &e^{(n)}_{k-\frac12,m}|(\tilde{V}_{0,n}(p^2), \tilde{V}_{1,n-1}(p^2),...,\tilde{V}_{n,0}(p^2))\\
  &:=
  (e^{(n)}_{k-\frac12,m}|\tilde{V}_{0,n}(p^2), e^{(n)}_{k-\frac12,m}|\tilde{V}_{1,n-1}(p^2),...,
  e^{(n)}_{k-\frac12,m}|\tilde{V}_{n,0}(p^2)).
\end{split}
\end{equation*}
\newtheorem*{thm:maass_e_half}{Theorem \ref{thm:maass_e_half}}
\begin{thm:maass_e_half}
 Let $e_{k-\frac12,m}^{(n)}$ be the $m$-th Fourier-Jacobi coefficient of generalized Cohen-Eisenstein series
 $H_{k-\frac12}^{(n+1)}$. $($See~$($\ref{eq:df_fourier_jacobi_cohen_eisenstein}$)$$)$.
 Then we obtain
  \begin{eqnarray*}
  &&  \!\!\!\!\!\!\!\!\!\!
  e^{(n)}_{k-\frac12,m}|(\tilde{V}_{0,n}(p^2), \tilde{V}_{1,n-1}(p^2),...,\tilde{V}_{n,0}(p^2))\\
  &=&
  p^{k(n-1)-\frac12(n^2+5n-5)}
  \left(e^{(n)}_{k-\frac12,\frac{m}{p^2}}|U_{p^2},\, e^{(n)}_{k-\frac12,m}|U_p,\, e^{(n)}_{k-\frac12,mp^2} \right)
  \begin{pmatrix}
          0 & p^{2k-3} \\
          p^{k-2} & p^{k-2} \left(\frac{-m}{p}\right) \\
          0 & 1 \end{pmatrix}\\
  &&
  \times
  A_{2,n+1}^{p}\!\left( p^{k-\frac{n+2}{2}-\frac12} \right)
  \mbox{diag}(1,p^{1/2},\cdots, p^{n/2}).
\end{eqnarray*}
Here $A_{2,n+1}^p\!\left( p^{k-\frac{n+2}{2}-\frac12} \right)$
is a $2 \times (n+1)$ matrix which is introduced
in the beginning of \S\ref{s:maass_relation_cohen_eisen}
and the both sides of the above identity are vectors of forms.
For any prime $p$ we regard $ e_{k-\frac12,\frac{m}{p^2}}^{(n)}$ as zero,
if $\frac{m}{p^2}$ is not an integer or $\frac{m}{p^2} \not \equiv 0$, $3 \mod 4$.
The symbol $\left(\frac{ * }{p}\right)$ denotes the Legendre symbol for odd prime $p$,
and $\left(\frac{a}{2}\right)$ $:=$ $0,1,-1$ accordingly as $a$ is even,
$a \equiv \pm 1$ mod $8$ or $a \equiv \pm 3$ mod $8$.
\end{thm:maass_e_half}

Theorem~\ref{thm:maass_e_half} gives also a relation among Fourier coefficients
of Siegel-Eisenstein series of integral weight.
The Fourier coefficients of Ikeda lifts satisfy similar relations
to the ones of the Fourier coefficients of Siegel-Eisenstein series (see Theorem~\ref{th:maass_relation_half_cusp}
for the detail).
We call these relations of Fourier coefficients of Ikeda lifts also the generalized Maass relations.
The generalized Maass relation among Fourier coefficients of the Ikeda lift $I_{2n}(h)$
of $h$ gives a fact that $\mathcal{F}_{h,g}$ in  Theorem~\ref{th:main}
is an eigenform for all Hecke operators, since the form $\mathcal{F}_{h,g}$ is constructed
from $I_{2n}(h)$ (and $g$).
Moreover, the eigenvalues of $\mathcal{F}_{h,g}$ are calculated from
the generalized Maass relations of Fourier coefficients of $I_{2n}(h)$.
This is the reason why we need Theorem~\ref{thm:maass_e_half}
to show Theorem~\ref{th:main}.
For the detail of the proof of Theorem~\ref{th:main}
see \S\ref{s:maass_relation_siegel_cusp}.

\subsection{About generalized Cohen-Eisenstein series}

We remark that
the generalized Cohen-Eisenstein series
has been introduced by Arakawa~\cite{Ar3}.
These series are certain Siegel modular forms of half-integral weight.
The Cohen-Eisenstein series were originally introduced by Cohen~\cite{Co}
as one variable functions.
In the case of degree one, it is known that the Cohen-Eisenstein series correspond
to the Eisenstein series
with respect to $\mbox{SL}(2,\Z)$ by the Shimura correspondence.
The generalized Cohen-Eisenstein series is defined from the Jacobi-Eisenstein series
of index $1$ through the isomorphism between Jacobi forms of index $1$ and
Siegel modular forms of half-integral weight.

\subsection{About generalized Maass relations}

As for generalizations of the Maass relation,
Yamazaki \cite{Ya,Ya2} obtained certain relations among Fourier-Jacobi coefficients of
Siegel-Eisenstein series of arbitrary degree of integral weight of \textit{integer indices}.
For our purpose we generalize some results in \cite{Ya,Ya2} on Fourier-Jacobi coefficients
of Siegel-Eisenstein series of integer indices to indices of half-integral symmetric matrix of size $2$.
Here the right-lower part or the left-upper part of these matrices of the index is $1$.
We need to introduce certain index-shift maps on Jacobi forms of
indices of such matrix (cf. \S\ref{ss:hecke_operators}).
To calculate the action of index-shift maps on Fourier-Jacobi coefficients of Siegel-Eisenstein series,
we use a certain relation between Fourier-Jacobi 
coefficients of Siegel-Eisenstein series and Jacobi-Eisenstein series
(cf. Proposition~\ref{prop:fourier_jacobi}).
This relation has been shown by Boecherer \cite[Satz7]{Bo}.
We also need to show a certain identity relation between Jacobi forms of integral
weight of $2\times 2$ matrix index
and Jacobi forms of half-integral weight of integer index (Lemma~\ref{lemma:iota}).
Moreover, we need to show a compatibility between this identity relation and
index-shift maps (cf. Proposition~\ref{prop:iota_U},~\ref{prop:iota_hecke}).

Through these relations
we can show that the generalized Maass relation of
generalized Cohen-Eisenstein series (Theorem~\ref{thm:maass_e_half}) are equivalent to certain relations among Jacobi-Eisenstein series
of integral weight of indices of matrix of size $2$ (Proposition~\ref{prop:Eisenstein_Voperator}). 
Finally,
to obtain the generalized Maass relation in Theorem~\ref{thm:maass_e_half},
we need to calculate the action of index-shift maps on Jacobi-Eisenstein series
of integral weight of indices of matrix of size $2$ (cf. \S\ref{s:action_of_ism}).

\vspace{0.3cm}
\noindent
\textit{Remark 1.1}

In his paper~\cite{Ko2} Kohnen gives a generalization of the Maass relation
for Siegel modular forms of even degree $2n$.
His result is different from our generalization,
since his result is concerned with the Fourier-Jacobi coefficients
with $(2n-1) \times (2n-1)$ matrix index.
We remark that some characterizations of the Ikeda lift by using generalized
Maass relation in~\cite{Ko2} are obtained by Kohnen-Kojima~\cite{KK} and by Yamana~\cite{Yamana}.
The characterization of the Ikeda lift by using the generalized Maass relation in Theorem~\ref{th:maass_relation_half_cusp} is open problem.

\vspace{0.3cm}
\noindent
\textit{Remark 1.2}

In his paper~\cite[\S5]{Tani} Tanigawa has obtained the same identity in
Theorem~\ref{thm:maass_e_half}
for \textit{Siegel-Eisenstein series of half-integral weight of degree two}
with arbitrary level $N$ which satisfies $4|N$.
He showed the identity by using the formula of local densities under the assumption $p {\not|} N$.
In our case we treat the \textit{generalized Cohen-Eisenstein series} of arbitrary degree,
which has essentially level $1$. Hence our result contains the relation also for $p=2$.
Moreover, our result is valid for any general degree.

\vspace{0.3cm}
\noindent
\textit{Remark 1.3}

To show the generalized Maass relations in Theorem~\ref{thm:maass_e_half},~\ref{th:maass_relation_half_cusp},
we treat the following four things:
\begin{enumerate}
\item
Fourier-Jacobi expansion of Jacobi forms (cf. \S\ref{ss:fj_expansion}),
\item
Fourier-Jacobi expansion of Siegel modular forms of half-integral weight (cf.~\S
\ref{ss:fourier_jacobi_expansion_half}),
\item
An isomorphism between Jacobi forms of matrix index of integral weight
and Jacobi forms of integer index of half-integral weight (cf. \S\ref{ss:iota})
\item
Exchange relations between the Siegel $\Phi$-operator for Jacobi forms and the index-shift map
for Jacobi forms of matrix index or of half-integral weight (cf.~\S\ref{s:siegelop}).
This is an analogue of the result shown by Krieg~\cite{Kr} in the case of
Siegel modular forms of integral weight.
\end{enumerate}

\ \\

\subsection{}
This paper is organized as follows:
in Sect.~2, the necessary notation and definitions are reviewed.
In Sect.~3, the relation among Fourier-Jacobi coefficients of
the Siegel-Eisenstein series and the Jacobi-Eisenstein series
is derived, which is a modification of the result given by Boecherer~\cite{Bo}
for certain special cases.
In Sect.~4, a certain map from a subspace of Jacobi forms of integral weight
of matrix index to a subspace of Jacobi forms of half-integral weight of
integer index is defined.
Moreover, the compatibility of this map with certain index-shift maps is studied.
In Sect.~5, we calculate the action of index-shift maps on the Jacobi-Eisenstein series.
We express these functions as summations of certain exponential functions with generalized Gauss sums.
In Sect.~6, a certain commutativity between index-shift maps on Jacobi forms
and Siegel $\Phi$-operators is derived.
In Sect.~7,  a generalized Maass relation for generalized Cohen-Eisenstein
series (Theorem~\ref{thm:maass_e_half})
will be proved,
while we will give a generalized Maass relation
for Siegel cusp forms of half-integral weight
and the proof of the main result (Theorem~\ref{th:main})
in Sect.~8.
We shall explain some numerical examples of the non-vanishing of
the lift in Sect.~9.

\ \\

\noindent
Acknowledgement.
The construction of the lift in this article has been suggested by Professor Tamotsu Ikeda
to the author at the Hakuba Autumn Workshop 2001.
The author wishes to express his hearty gratitude to Professor Ikeda for the suggestion.
The author also would like to express his sincere gratitude to
Professor Tomoyoshi Ibukiyama for continuous encouragement.
This work was supported by JSPS KAKENHI Grant Number 23740018 and 80597766.

\section{Notation and definitions}

$\R^+$ : the set of all positive real numbers
 
$R^{(n,m)}$ : the set of $n\times m$ matrices with entries in a commutative ring $R$

$\mbox{Sym}_n^*$ : the set of all half-integral symmetric matrices of size $n$

$\mbox{Sym}_n^+$ : all positive definite matrices in $\mbox{Sym}_n^*$

${^t B}$ : the transpose of a matrix $B$

$A[B]$ $:=$ ${^t B} A B $ for two matrices $A \in R^{(n,n)}$ and $B \in R^{(n,m)}$

$1_n$ (resp. $0_n$) : identity matrix (resp. zero matrix) of size $n$

$\mbox{tr}(S)$ : the trace of a square matrix $S$

$e(S) := e^{2 \pi \sqrt{-1}\, \mbox{tr}(S)}$ for a square matrix $S$

$\mbox{rank}_p(x)$ : the rank of matrix $x \in \Z^{(n,m)}$ over the finite field $\Z/p\Z$

$\mbox{diag}(a_1,...,a_n)$ : the diagonal matrix $\left(\begin{smallmatrix} a_1 & & \\ &\ddots & \\ & & a_n\end{smallmatrix}\right)$
for square matrices $a_1$, ..., $a_n$

$\left(\frac{*}{p}\right)$ : the Legendre symbol for odd prime $p$

$\left(\frac{*}{2}\right)$ $:=$ $0,1,-1$ accordingly as $a$ is even,
                             $a \equiv \pm 1$ mod $8$ or $a \equiv \pm 3$ mod $8$

$M_{k-\frac12}(\Gamma_0^{(n)}(4))$ : the space of Siegel modular forms of weight $k-\frac12$ of degree $n$

$M_{k-\frac12}^{+}(\Gamma_0^{(n)}(4))$ : the plus-space of $M_{k-\frac12}(\Gamma_0^{(n)}(4))$ (cf. \cite{Ib}).

$\mathfrak{H}_n$ : the Siegel upper half space of degree $n$

$\delta(\mathcal{S}) := 1 $ or $0$ accordingly as the statement $\mathcal{S}$ is true or false. 

For any function $F$ and operators $T_1$, $T_2$, ... , $T_n$
we put
\begin{eqnarray*}
  F|(T_1,T_2,...,T_n)
  &:=&
  (F|T_1, F|T_2,...,F|T_n).
\end{eqnarray*}

\subsection{Jacobi group}\label{ss:Jacobi_group}
For a positive integer $n$ we define the following groups:
\begin{eqnarray*}
 \mbox{GSp}_n^+(\R)
 &:=&
 \left\{
  g \in \R^{(2n,2n)} \, 
  | \, 
  g \left(\begin{smallmatrix}
     0_n & -1_n \\ 1_n & 0_n 
   \end{smallmatrix}\right)
  {^t g} 
  = n(g) 
  \left(\begin{smallmatrix}
     0_n & -1_n \\ 1_n & 0_n 
   \end{smallmatrix}\right)
  \mbox{ for some } n(g) \in \R^+
 \right\} ,
\\
 \mbox{Sp}_n(\R)
 &:=&
 \left\{
  g \in \mbox{GSp}_n^+(\R) \, | \, 
  n(g) = 1
 \right\} ,
\\
 \Gamma_n
 &:=&
 \mbox{Sp}_n(\R) \cap \Z^{(2n,2n)},
\\
 \Gamma_{\infty}^{(n)}
 &:=&
 \left.
 \left\{
  \begin{pmatrix} A & B \\ C & D \end{pmatrix} \in \Gamma_n
 \, \right| \,
  C = 0_n 
 \right\} ,\\
 \Gamma_0^{(n)}(4)
 &:=&
 \left\{
 \left.
  \begin{pmatrix}
   A & B \\ C & D
  \end{pmatrix}
  \in \Gamma_n
  \, \right| \, 
  C \equiv 0 \mod 4
 \right\} .
\end{eqnarray*}
For a matrix $g \in \mbox{GSp}_n^+(\R)$, the number $n(g)$ in the above definition of $\mbox{GSp}_n^+(\R)$ is called the \textit{similitude} of the matrix $g$.

For positive integers $n$ and $r$,
we define the subgroup $G_{n,r}^J \subset \mbox{GSp}_{n+r}^+(\R)$ by
\begin{eqnarray*}
 G_{n,r}^J 
 &:=&
 \left\{
  \begin{pmatrix}
   A &   & B &  \\
     & U &   &  \\
   C &   & D &  \\
     &   &   & V
\end{pmatrix}\begin{pmatrix}
   1_n &   &   & \mu  \\
   ^t \lambda  & 1_r & ^t \mu  & {^t \lambda} \mu + \kappa  \\
     & & 1_n & - \lambda  \\
     &   &   & 1_r
\end{pmatrix} 
\in \mbox{GSp}_{n+r}^+(\R)
\right\},
\end{eqnarray*}
where
the matrices runs over
$\begin{pmatrix}
  A&B\\C&D
 \end{pmatrix}
 \in \mbox{GSp}_n^+(\R)$,
$\begin{pmatrix}
  U&0\\0&V
 \end{pmatrix}
 \in \mbox{GSp}_r^+(\R)$,
$\lambda, \mu \in \R^{(n,r)}$
and
$\kappa = {^t \kappa}\in \R^{(r,r)}$.

We will abbreviate such an element
 $\left(\begin{smallmatrix}
    A &   & B &  \\
      & U &   &  \\
    C &   & D &  \\
      &   &   & V
 \end{smallmatrix}\right)
 \left(\begin{smallmatrix}
   1_n &   &   & \mu  \\
   ^t \lambda  & 1_r & ^t \mu  & {^t \lambda}\mu + \kappa  \\
     & & 1_n & - \lambda  \\
     &   &   & 1_r
 \end{smallmatrix}\right)$
as
\[
\left(\begin{pmatrix}A&B\\C&D\end{pmatrix}
  \times
  \begin{pmatrix}U&0\\0&V\end{pmatrix},
  [
   (\lambda,\mu),\kappa
  ]
  \right) .
\]
We remark that two matrices $\left(\begin{smallmatrix}A&B\\C&D\end{smallmatrix}\right)$ and
$\left(\begin{smallmatrix}U&0\\0&V\end{smallmatrix}\right)$
in the above notation
have the same similitude.
We will often write
\begin{eqnarray*}
\left(\left(\begin{matrix}A&B\\C&D\end{matrix}\right),
  [
   (\lambda,\mu),\kappa
  ]
  \right)
\end{eqnarray*}
instead of writing
 $\left(\left(\begin{smallmatrix}A&B\\C&D\end{smallmatrix}\right)
  \times
  1_{2r},
  [
   (\lambda,\mu),\kappa
  ]
  \right)$
for simplicity.
We remark that the element
 $\left(\left(\begin{smallmatrix}A&B\\C&D\end{smallmatrix}\right),
  [
   (\lambda,\mu),\kappa
  ]
  \right)$
belongs to $\mbox{Sp}_{n+r}(\R) $.
Similarly, an element
\begin{eqnarray*}
 \left(\begin{smallmatrix}
   1_n &   &   & \mu  \\
   ^t \lambda  & 1_r & ^t \mu  & {^t \lambda}\mu + \kappa  \\
     & & 1_n & - \lambda  \\
     &   &   & 1_r
 \end{smallmatrix}\right)
 \left(\begin{smallmatrix}
    A &   & B &  \\
      & U &   &  \\
    C &   & D &  \\
      &   &   & V
 \end{smallmatrix}\right)
\end{eqnarray*}
will be abbreviated as
\begin{eqnarray*}
 \left(
  [
   (\lambda,\mu),\kappa
  ],
  \left(\begin{matrix}A&B\\C&D\end{matrix}\right)
  \times
  \left(\begin{matrix}U&0\\0&V\end{matrix}\right)
 \right),
\end{eqnarray*}
and we will abbreviate it as
 $
 \left(
  [
   (\lambda,\mu),\kappa
  ],
  \left(\begin{smallmatrix}A&B\\C&D\end{smallmatrix}\right)
 \right)
 $
for the case $U = V = 1_r$ .

We set a subgroup $\Gamma_{n,r}^J$ of $G_{n,r}^J$ by
\begin{eqnarray*}
 \Gamma_{n,r}^J 
  &:=&
 \left\{
  \left(M,
  [
   (\lambda,\mu),\kappa
  ]
  \right) 
   \in G_{n,r}^J
  \, \left| \,
  M \in \Gamma_n ,
  \lambda, \mu \in \Z^{(n,r)}, \kappa \in \Z^{(r,r)}
 \right\} \right. .
\end{eqnarray*}

\subsection{Groups $\widetilde{\mbox{GSp}_n^+(\R)}$
and $\widetilde{G_{n,1}^J}$}\label{ss:double_covering_groups}

We denote by $\widetilde{\mbox{GSp}_n^+(\R)}$
a group
which consists of pairs $(M,\varphi(\tau))$,
where $M$ is a matrix $M = \left(\begin{smallmatrix} A&B\\C&D \end{smallmatrix}\right)\in \mbox{GSp}_n^+(\R)$,
and where $\varphi$ is any holomorphic function on $\mathfrak{H}_n$
such that $|\varphi(\tau)|^2 = \det(M)^{-\frac12} |\det(C\tau + D)|$.
The group operation on $\widetilde{\mbox{GSp}_n^+(\R)}$ is given by
$(M,\varphi(\tau))(M',\varphi'(\tau)) := (M M', \varphi(M'\tau)\varphi'(\tau))$.

We embed $\Gamma_0^{(n)}(4)$ into the group $\widetilde{\mbox{GSp}_n^+(\R)}$
via $M \rightarrow (M,\theta^{(n)}(M\tau)\, \theta^{(n)}(\tau)^{-1})$, where
 $ \theta^{(n)}(\tau) 
  :=
   \displaystyle{\sum_{p \in \Z^{(n,1)}} e(\tau[p])}$
is the theta constant.
We denote by $\Gamma_0^{(n)}(4)^*$ the image of $\Gamma_0^{(n)}(4)$
in $\widetilde{\mbox{GSp}_n^+(\R)}$ by this embedding.

We define the Heisenberg group
\begin{eqnarray*}
 H_{n,1}(\R)
 &:=&
 \left\{(1_{2n},[(\lambda,\mu),\kappa]) \in \mbox{Sp}_{n+1}(\R)
 \, | \,
 \lambda,\mu \in \R^{(n,1)}, \kappa \in \R \right\} .
\end{eqnarray*}
If there is no confusion, we will write $[(\lambda,\mu),\kappa]$
for the element $(1_{2n},[(\lambda,\mu),\kappa])$ for simplicity.

We define a group
\begin{eqnarray*}
 \widetilde{G_{n,1}^J}
 &:=&
 \widetilde{\mbox{GSp}_n^+(\R)} \ltimes H_{n,1}(\R) \\
 &=&
 \left. \left\{(\tilde{M},[(\lambda,\mu),\kappa]) \, \right| \,
   \tilde{M} \in \widetilde{\mbox{GSp}_n^+(\R)},
   [(\lambda,\mu),\kappa] \in H_{n,1}(\R) \right\} .
\end{eqnarray*}
Here the group operation on $\widetilde{G_{n,1}^J}$ is given by
\begin{eqnarray*}
 (\tilde{M_1},[(\lambda_1,\mu_1),\kappa_1])\cdot (\tilde{M_2},[(\lambda_2,\mu_2),\kappa_2])
 &:=&
 (\tilde{M_1}\tilde{M_2},[(\lambda',\mu'),\kappa'])
\end{eqnarray*}
for $(\tilde{M_i},[(\lambda_i,\mu_i),\kappa_i]) \in \widetilde{G_{n,1}^J}$ $(i=1,2)$,
and where $[(\lambda',\mu'),\kappa'] \in H_{n,1}(\R)$ is the matrix determined through the identity
\begin{eqnarray*}
\begin{aligned}
 &
 (M_1\times\left(\begin{smallmatrix}n(M_1)&0\\0&1\end{smallmatrix}\right),[(\lambda_1,\mu_1),\kappa_1])
 (M_2\times\left(\begin{smallmatrix}n(M_2)&0\\0&1\end{smallmatrix}\right),[(\lambda_2,\mu_2),\kappa_2])
\\
 &=
 (M_1M_2 \times \left(\begin{smallmatrix}n(M_1)n(M_2)&0\\0&1\end{smallmatrix}\right),[(\lambda',\mu'),\kappa'])
\end{aligned}
\end{eqnarray*}
in $G_{n,1}^J$.
Here $n(M_i)$ is the similitude of $M_i$.

\subsection{Action of the Jacobi group}
The group $G_{n,r}^J$ acts on $\mathfrak{H}_n \times \C^{(n,r)}$ by
\begin{eqnarray*}
 \gamma \cdot (\tau,z) 
 &:=&
 \left(
  \begin{pmatrix}A&B\\C&D\end{pmatrix} \cdot \tau
 \, ,\,
  ^t(C\tau+D)^{-1}(z + \tau\lambda + \mu)^t U
 \right)
\end{eqnarray*}
for any $\gamma = \left(\left(\begin{smallmatrix}A&B\\C&D\end{smallmatrix}\right)
  \times
  \left(\begin{smallmatrix}U&0\\0&V\end{smallmatrix}\right),
  [
   (\lambda,\mu),\kappa
  ]
  \right) \in G_{n,r}^J$ 
and for any $(\tau,z) \in \mathfrak{H}_n \times \C^{(n,r)}$. 
Here $\begin{pmatrix}A&B\\C&D\end{pmatrix} \cdot \tau := (A\tau+B)(C\tau+D)^{-1}$ 
is the usual transformation.

The group $\widetilde{G_{n,1}^J}$ acts on $\mathfrak{H}_n\times \C^{(n,1)}$
through the projection $\widetilde{G_{n,1}^J} \rightarrow G_{n,1}^J$.
It means
\begin{eqnarray*}
 \tilde{\gamma}\cdot(\tau,z) 
  &:=& 
  (M\times\left(\begin{smallmatrix}n(M)&0\\0&1\end{smallmatrix}\right),[(\lambda,\mu),\kappa])\cdot(\tau,z)
\end{eqnarray*}
for $\tilde{\gamma} = ((M,\varphi),[(\lambda,\mu),\kappa]) \in \widetilde{G_{n,1}^J}$
and for $(\tau,z) \in \mathfrak{H}_n\times \C^{(n,1)}$.
Here $n(M)$ is the similitude of $M \in \mbox{GSp}_n^+(\R)$.

\subsection{Factors of automorphy}\label{ss:factors_automorphy}
Let $k$ be an integer and let $\mathcal{M} \in \mbox{Sym}_r^+$.
For
 $\gamma = \left(\left(\begin{smallmatrix}A&B\\C&D\end{smallmatrix}\right)
  \times
  \left(\begin{smallmatrix}U&0\\0&V\end{smallmatrix}\right),
  [
   (\lambda,\mu),\kappa
  ]
  \right) \in G_{n,r}^J$
we define a factor of automorphy
\begin{eqnarray*}
 J_{k,\mathcal{M}}\left( 
   \gamma, (\tau,z) \right)
 &:=& 
  \det(V)^k \det(C\tau+D)^k\, e(V^{-1}\mathcal{M}U (((C\tau+D)^{-1}C)[z + \tau \lambda + \mu]) ) \\
 && \times
  e(- V^{-1} \mathcal{M} U ({^t \lambda} \tau \lambda
     + {^t z} \lambda + {^t \lambda} z + {^t \mu} \lambda + {^t \lambda}\mu + \kappa)).
\end{eqnarray*}
We define a slash operator $|_{k,\mathcal{M}}$ by
\begin{eqnarray*}
 (\phi|_{k,\mathcal{M}}\gamma)(\tau,z) 
 &:=& 
 J_{k,\mathcal{M}}(\gamma,(\tau,z))^{-1} \phi(\gamma\cdot(\tau,z))
\end{eqnarray*}
for any function $\phi$ on $\mathfrak{H}_n \times \C^{(n,r)}$ 
and for any $\gamma \in G_{n,r}^J$.
We remark that
\begin{eqnarray*}
 J_{k,\mathcal{M}}(\gamma_1 \gamma_2, (\tau,z))
 &=&
 J_{k,\mathcal{M}}(\gamma_1, \gamma_2 \cdot (\tau,z))
 J_{k,V_1^{-1}\mathcal{M}U_1}(\gamma_2, (\tau,z)),
\\
 \phi|_{k,\mathcal{M}}\gamma_1 \gamma_2 
 &=&
 (\phi|_{k,\mathcal{M}}\gamma_1) |_{k,V_1^{-1}\mathcal{M}U_1}\gamma_2 .
\end{eqnarray*}
for any $\gamma_i = \left(M_i
  \times
  \left(\begin{smallmatrix}U_i&0\\0&V_i\end{smallmatrix}\right),
  [
   (\lambda_i,\mu_i),\kappa_i
  ]
  \right) \in G_{n,r}^J$ $(i =1,2)$.

Let $k$ and $m$ be integers.
We define a slash operator $|_{k-\frac12,m}$ for any function $\phi$ on $\mathfrak{H}_n\times \C^{(n,1)}$ by
\begin{eqnarray*}
 \phi|_{k-\frac12,m}\tilde{\gamma}
 &:=&
 J_{k-\frac12,m}(\tilde{\gamma},(\tau,z))^{-1}
 \phi(\tilde{\gamma}\cdot(\tau,z))
\end{eqnarray*}
for any $\tilde{\gamma} = ((M,\varphi),[(\lambda,\mu),\kappa]) \in \widetilde{G_{n,1}^J}$.
Here we define a factor of automorphy
\begin{eqnarray*} 
 J_{k-\frac12,m}(\tilde{\gamma},(\tau,z))
 &:=&
 \varphi(\tau)^{2k-1} e(n(M) m (((C\tau+D)^{-1}C)[z + \tau \lambda + \mu]) ) \\
 && \times
  e(-  n(M) m ({^t \lambda} \tau \lambda + {^t z} \lambda + {^t \lambda} z + {^t \mu} \lambda + {^t \lambda}\mu + \kappa)),
\end{eqnarray*}
where $n(M)$ is the similitude of $M$.
We remark that
\begin{eqnarray*}
 J_{k-\frac12,m}(\tilde{\gamma_1}\tilde{\gamma_2},(\tau,z))
 &=&
 J_{k-\frac12,m}(\tilde{\gamma_1},\tilde{\gamma_2}\cdot(\tau,z))
 J_{k-\frac12,n(M_1) m}(\tilde{\gamma_2},(\tau,z))
\\
 \phi|_{k-\frac12,m}\tilde{\gamma_1}\tilde{\gamma_2}
 &=&
 (\phi|_{k-\frac12,m}\tilde{\gamma_1})|_{k-\frac12,n(M_1)m}\tilde{\gamma_2}
\end{eqnarray*}
for any $\tilde{\gamma_i} = ((M_i,\varphi_i),[(\lambda_i,\mu_i),\kappa_i]) \in \widetilde{G_{n,1}^J}$
$(i=1,2)$.

\subsection{Jacobi forms of matrix index}\label{ss:jacobi_forms_of_matrix_index}
We quote the definition of Jacobi form of matrix index from \cite{Zi}.
\begin{df}
For an integer $k$ and for an matrix $\mathcal{M} \in \mbox{Sym}_r^+$,
a $\C$-valued holomorphic function $\phi$ on $\mathfrak{H}_n \times \C^{(n,r)}$ is called
\textit{a Jacobi form of weight $k$ of
index $\mathcal{M}$ of degree $n$}, if $\phi$ satisfies the following two conditions:
\begin{enumerate}
\item
the transformation formula
$\phi|_{k,\mathcal{M}} \gamma = \phi$ for any $\gamma \in \Gamma_{n,r}^J$,
\item
$\phi$ has the Fourier expansion:
$ \phi(\tau,z)
 = \!\!\!\!\!
 \displaystyle{
   \sum_{\begin{smallmatrix}N \in Sym_n^*,R \in Z^{(n,r)} \\ 4N - R \mathcal{M}^{-1} {^t R} \geq 0 
          \end{smallmatrix}} \!\!\!\!\! c(N,R) e(N\tau) e({^t R} z)
              }$.
\end{enumerate}
\end{df}
We remark that the second condition follows from the Koecher principle (cf.~\cite[Lemma~1.6]{Zi})
if $n > 1$.
In the condition (2),
if $\phi$ satisfies $c(N,R) = 0$ unless $4N - R \mathcal{M}^{-1} {^t R} > 0$,
then $\phi$ is called a \textit{Jacobi cusp form}.

We denote by $J_{k,\mathcal{M}}^{(n)}$ the $\C$-vector space of Jacobi forms of weight $k$ of index $\mathcal{M}$
of degree $n$.

\subsection{Jacobi forms of half-integral weight}\label{ss:def_jacobi_half_weight}

We set the subgroup $\Gamma_{n,1}^{J*}$ of $\widetilde{G_{n,1}^J}$ by
\begin{eqnarray*}
 \Gamma_{n,1}^{J*}
 &:=&
 \left\{
  (M^*,[(\lambda,\mu),\kappa]) \in \widetilde{G_{n,1}^J}
  \, | \,
  M^* \in \Gamma_0^{(n)}(4)^*, \,
  \lambda,\mu \in \Z^{(n,1)}, \kappa \in \Z
 \right\} \\
 &\cong&
  \Gamma_0^{(n)}(4)^* \ltimes H_{n,1}(\Z),
\end{eqnarray*}
where we put $H_{n,1}(\Z) := H_{n,1}(\R) \cap \Z^{(2n+2,2n+2)}$.
Here the group $\Gamma_0^{(n)}(4)^*$ was defined in \S\ref{ss:double_covering_groups}.
\begin{df}
For an integer $k$ and for an integer $m$,
a holomorphic function $\phi$ on $\mathfrak{H}_n \times \C^{(n,1)}$
is called a \textit{Jacobi form of weight $k-\frac12$ of index $m$},
if $\phi$ satisfies the following two conditions:
\begin{enumerate}
\item
the transformation formula
$\phi|_{k-\frac12,m} \gamma^* = \phi$ for any $\gamma^* \in \Gamma_{n,1}^{J*}$,
\item
$\phi^2|_{2k-1,2m}\gamma$ has the Fourier expansion for any $\gamma \in \Gamma_{n,1}^J$:
\begin{eqnarray*}
 \left(\phi^2|_{2k-1,2m}\gamma\right) (\tau,z)
 &=&
 \sum_{\begin{smallmatrix}
        N \in Sym_n^*,R \in \Z^{(n,1)} \\
        4Nm - h R {^t R} \geq 0
       \end{smallmatrix}}
      C(N,R)\, e\!\left(\frac{1}{h}N\tau\right) e\!\left({^t R} z\right).
\end{eqnarray*}
with a certain integer $h > 0$, and
where the slash operator $|_{k-\frac12, m}$ was defined in $\S\ref{ss:factors_automorphy}$.
\end{enumerate}
\end{df}

In the condition (2),
for any $\gamma$ if $\phi$ satisfies $C(N,R) = 0$ unless $4Nm - h R{^t R} > 0$,
then $\phi$ is called a \textit{Jacobi cusp form}.

We denote by $J_{k-\frac12,m}^{(n)}$
the $\C$-vector space of Jacobi forms
of weight $k-\frac12$ of index $m$ of degree $n$.

\subsection{Index-shift maps of Jacobi forms}\label{ss:hecke_operators}
In this subsection we introduce two kinds of maps.
The both maps shift the index of Jacobi forms and these are generalizations of
the $V_l$-map in the sense of Eichler-Zagier \cite{EZ}.
 
We define two groups $\mbox{GSp}_n^+(\Z) := \mbox{GSp}_n^+(\R) \cap \Z^{(2n,2n)}$ and
\begin{eqnarray*}
 \widetilde{\mbox{GSp}_n^+(\Z)}
 &:=&
 \left.
 \left\{
  (M,\varphi) \in \widetilde{\mbox{GSp}_n^+(\R)} \, \right| \, M \in \mbox{GSp}_n^+(\Z)
 \right\}.
\end{eqnarray*}

First we define index-shift maps for Jacobi forms of \textit{integral weight of matrix index}.
Let $\mathcal{M} = \smat{*}{*}{ * }{1} \in \mbox{Sym}_2^+$.
Let $X \in \mbox{GSp}_n^+(\Z)$ be a matrix such that the similitude of $X$ is $n(X)=p^2$ with a prime $p$.
For any $\phi \in J_{k,\mathcal{M}}^{(n)}$ we define the map
\begin{eqnarray*}
 \phi|V(X)
 &:=&
 \sum_{u,v \in (\Z/p\Z)^{(n,1)}}
 \sum_{M \in \Gamma_n \backslash \Gamma_n X \Gamma_n} \!\!\!\!\!
  \phi|_{k,\mathcal{M}} 
   \left(M\times
     \left(\begin{smallmatrix}p^2&0&0&0\\0&p&0&0\\0&0&1&0\\0&0&0&p\end{smallmatrix}\right),
     [((0,u),(0,v)),0_2]\right),
\end{eqnarray*}
where $(0,u),(0,v) \in (\Z/p\Z)^{(n,2)}$ and where $0_2$
is the zero matrix of size $2$.
See \S\ref{ss:Jacobi_group}
for the symbol of  the matrix
$
  \left(M\times
     \left(\begin{smallmatrix}p^2&0&0&0\\0&p&0&0\\0&0&1&0\\0&0&0&p\end{smallmatrix}\right),
     [((0,u),(0,v)),0_2]\right).
$
The above summations are finite sums and do not depend on
the choice of the representatives $u$, $v$ and $M$.
A straightforward calculation shows that
$\phi|V(X)$ belongs to $J_{k,\mathcal{M}[\left(\begin{smallmatrix}p&0\\0&1\end{smallmatrix}\right)]}^{(n)}$.
Namely $V(X)$ is a map:
\begin{eqnarray*}
 V(X) \ : \ J_{k,\mathcal{M}}^{(n)} \rightarrow J_{k,\mathcal{M}[\left(\begin{smallmatrix}p&0\\0&1\end{smallmatrix}\right)]}^{(n)}.
\end{eqnarray*}

For the sake of simplicity we set
\begin{eqnarray*}
 V_{\alpha,n-\alpha}(p^2) 
 &:=&
 V(\mbox{diag}(1_{\alpha},p 1_{n-\alpha}, p^2 1_{\alpha}, p 1_{n-\alpha}))
\end{eqnarray*}
for any prime $p$ and for any $\alpha$ $(0\leq \alpha \leq n)$.

Next we shall define index-shift maps for Jacobi forms of \textit{half-integral weight of integer index}.
We assume that $p$ is \textit{an odd prime}.
Let $m$ be a positive integer.
Let $Y = (X,\varphi) \in \widetilde{\mbox{GSp}_n^+(\Z)}$ with $n(X) = p^{2t}$, where $t$ is a positive integer.
For $\psi \in J_{k-\frac12,m}^{(n)}$ we define
\begin{eqnarray*}
 \psi|\widetilde{V}(Y)
 &:=&
 n(X)^{\frac{n(2k-1)}{4} - \frac{n(n+1)}{2}}\sum_{\tilde{M} \in \Gamma_0^{(n)}(4)^* \backslash \Gamma_0^{(n)}(4)^* Y \Gamma_0^{(n)}(4)^*}
  \psi|_{k-\frac12,m} (\tilde{M},[(0,0),0]) ,
\end{eqnarray*}
where the above summation is a finite sum and
does not depend on the choice of the representatives $\tilde{M}$.
A direct computation shows
that $\psi|\widetilde{V}(Y)$ belongs to $J_{k-\frac12,mp^{2t}}^{(n)}$.

For the sake of simplicity we set
\begin{eqnarray*}
 \tilde{V}_{\alpha,n-\alpha}(p^2)
 &:=&
 \tilde{V}((\mbox{diag}(1_{\alpha},p 1_{n-\alpha}, p^2 1_{\alpha}, p 1_{n-\alpha}),p^{\alpha/2}))
\end{eqnarray*}
for any odd prime $p$ and for any $\alpha$ $(0\leq \alpha \leq n)$.
As for $p=2$,
we will introduce  index-shift maps $\tilde{V}_{\alpha,n-\alpha}(4)$ in \S\ref{ss:hecke_p2},
which are maps from a certain subspace $J_{k-\frac12,m}^{(n)*}$ of
$J_{k-\frac12,m}^{(n)}$ to $J_{k-\frac12,4m}^{(n)}$.

\subsection{Hecke operators for Siegel modular forms of half-integral weight}\label{ss:hecke_op_siegel_half}
The Hecke theory for Siegel modular forms was first introduced by Shimura~\cite{Shi} for degree $n=1$
and by Zhuravlev~\cite{Zhu:hecke,Zhu:euler} for degree $n>1$.
We quote the definition of Hecke operator from~\cite{Zhu:hecke,Zhu:euler}.
Let $Y = (X,\varphi) \in \widetilde{\mbox{GSp}_n^+(\Z)}$.
Let $\phi \in M_{k-\frac12}(\Gamma_0^{(n)}(4))$. We define
\begin{eqnarray*}
   \phi| \tilde{T}(Y)
 &:=&
    n(X)^{\frac{n(2k-1)}{4} - \frac{n(n+1)}{2}}\sum_{\tilde{M} \in \Gamma_0^{(n)}(4)^* \backslash \Gamma_0^{(n)}(4)^* Y \Gamma_0^{(n)}(4)^*}
  \phi|_{k-\frac12} \tilde{M} ,
\end{eqnarray*}
where
$(\phi|_{k-\frac12} \tilde{M})(\tau) := \varphi(\tau)^{-2k+1} \phi(M\cdot \tau)$ for $\tilde{M} = (M,\varphi)$
and $n(X)$ is the similitude of $X$.
For the sake of simplicity we set
\begin{eqnarray*}
 \tilde{T}_{\alpha,n-\alpha}(p^2)
 &:=&
 \tilde{T}((\mbox{diag}(1_{\alpha},p 1_{n-\alpha}, p^2 1_{\alpha}, p 1_{n-\alpha}),p^{\alpha/2}))
\end{eqnarray*}
for any odd prime $p$ and for any $\alpha$ $(0\leq \alpha \leq n)$.

\subsection{$L$-function of Siegel modular forms of half-integral weight}
\label{ss:L_function_siegel_half}
In this subsection we review the Hecke theory for Siegel modular forms
of half-integral weight which has been introduced by Zhuravlev~\cite{Zhu:hecke,Zhu:euler}
and quote the definition of $L$-function of a Siegel modular form of half-integral weight.

Let $\tilde{\mathcal{H}}_{p^2}^{(m)}$ be the local Hecke ring generated by
double cosets
\[
K_{\alpha}^{(m)} := \Gamma_0^{(m)}(4)^*(\mbox{diag}(1_{\alpha},p 1_{m-\alpha}, p^2 1_{\alpha}, p 1_{m-\alpha}),p^{\alpha/2})\Gamma_0^{(m)}(4)^* \quad (0 \leq \alpha \leq m)
\]
and
${K_0^{(m)}}^{-1}$
over $\C$.
If $p$ is an odd prime, then it is shown in~\cite{Zhu:hecke,Zhu:euler}
that the local Hecke ring $\tilde{\mathcal{H}}_{p^2}^{(m)}$ is commutative and
there exists the isomorphism map
  \begin{eqnarray*}
    \Psi_m \ : \ \tilde{\mathcal{H}}_{p^2}^{(m)} &\rightarrow& R_m,
  \end{eqnarray*}
where $R_m$ denotes
 $R_m := \C^{W_2}\!\!\left[z_0^{\pm},z_1^{\pm},...,z_{m}^{\pm}\right]$,
and where the subring $\C^{W_2}\!\!\left[z_0^{\pm},z_1^{\pm},...,z_{m}^{\pm}\right]$
of $\C\left[z_0^{\pm},z_1^{\pm},...,z_{m}^{\pm}\right]$ consists of all $W_2$-invariant
polynomials.
Here $W_2$ is the Weyl group of a symplectic group
and the action of $W_2$ on $\C\!\left[z_0^{\pm},...,z_m^{\pm}\right]$ is generated
by all permutations of $\{z_1,...,z_m\}$ and by the maps
\begin{eqnarray*}
  \sigma_i \ : \
  z_0 \rightarrow z_0 z_i,\ z_i \rightarrow z_i^{-1},\ z_j \rightarrow z_j\ (j \neq i)
\end{eqnarray*}
for $i = 1,...,m$.
The isomorphism $\Psi_m$ is defined as follows:
Let
\begin{eqnarray*}
  T &=& \sum_i a_i \Gamma_0^{(m)}(4)^* (X_i, \varphi_i)
\end{eqnarray*}
be a decomposition of $T \in \tilde{\mathcal{H}}_{p^2}^{(m)}$,
where $a_i \in \C$ and $(X_i, \varphi_i) \in \widetilde{\mbox{GSp}_n^+(\Z)}$.
We can assume that $X_i$ is an upper-triangular matrix
$X_i = \smat{p^{\delta_i} {^t D_i}^{-1}}{B_i}{0}{D_i}$ with
\begin{eqnarray*}
  D_i = \begin{pmatrix} d_{i1} & * & * \\ 0 & \ddots & * \\ 0 & 0 & d_{im}\end{pmatrix}
\end{eqnarray*}
and $\varphi_i$ is a constant function.
It is known that $|\varphi_i|^{-1} \varphi_i$ is a forth root of unity.
Then $\Psi_m(T)$ is given by
\begin{eqnarray*}
  \Psi_m(T) &:=&
  \sum_i a_i \left( \frac{\varphi_i}{|\varphi_i|} \right)^{-2k+1} z_0^{\delta_i}
  \prod_{j=1}^m (p^{-j} z_j)^{d_{ij}}
\end{eqnarray*}
with a fixed integer $k$.
For the explicit decomposition of generators $K_{\alpha}^{(m)}$ by left
$\Gamma_0^{(m)}(4)^*$-cosets, see~\cite[Prop.7.1]{Zhu:hecke}.

  We define $\gamma_j \in \C[z_1^{\pm},...,z_m^{\pm}]$ $(j=0,...,2m)$ through the identity
    \begin{eqnarray*}
     \sum_{j=0}^{2m} \gamma_j X^j
     &=&
     \prod_{i=1}^{m} \left\{\left( 1 - z_i X \right) (1 - z_i^{-1} X)\right\}.
  \end{eqnarray*}
  Here $\gamma_j$ $(j=0,...,2m)$ is a $W_2$-invariant.
  There exists $\tilde{\gamma}_{i,p} \in \tilde{\mathcal{H}}_{p^2}^{(m)}$ $(i=0,...,2m)$
  which satisfies $\Psi_m(\tilde{\gamma}_{i,p}) = \gamma_i \in R_{m}$.
  We remark that
  $\tilde{\gamma}_{i,p} = \tilde{\gamma}_{2m-i,p}$ and
  $\tilde{\gamma}_{0,p} = K_0^{(m)}$.
  
  For $p=2$ we will introduce in \S\ref{ss:fj_expansion_half}
  the Hecke operators $\tilde{T}_{\alpha,m-\alpha}(4)$ $(\alpha=0,...,m)$
  through the isomorphism between Siegel modular forms
  of half-integral weight and Jacobi forms of index $1$
  (see (\ref{eq:half_hecke_2}) in \S\ref{ss:fj_expansion_half}).
  We remark that the Hecke operators $\tilde{T}_{\alpha,m-\alpha}(4)$ $(\alpha=0,...,m)$
  are defined for the generalized plus-space,
  which is a certain subspace of $M_{k-\frac12}(\Gamma_0^{(m)}(4))$.
  Through the definition of $\tilde{\gamma}_{i,p}$ for odd prime $p$,
  we define $\tilde{\gamma}_{i,2}$ in the same formula by using 
  $\tilde{T}_{\alpha,m-\alpha}(4)$ $(\alpha=0,...,m)$
  as in the case of odd primes.
  by replacing $p$ by $2$.
  
  Let $F \in M_{k-\frac12}^{+}(\Gamma_0^{(m)}(4))$ be an eigenform for any Hecke operator
  $\tilde{T}_{\alpha,m-\alpha}(p^2)$ $(0 \leq \alpha \leq m)$ and for any prime $p$.
  Here $M_{k-\frac12}^{+}(\Gamma_0^{(m)}(4))$ denotes the generalized plus-space
  which is a certain subspace of $M_{k-\frac12}(\Gamma_0^{(m)}(4))$
  (see~\cite{Ib} or \S\ref{ss:fj_expansion_half} for the definition of
  $M_{k-\frac12}^{+}(\Gamma_0^{(m)}(4))$).
  We define the Euler $p$-factor of $F$ by
  \begin{eqnarray*}
    Q_{F,p}(z) :&=&
    \sum_{j=0}^{2m} \lambda_F(\tilde{\gamma}_{j,p})
     z^j,
  \end{eqnarray*}
  where $\lambda_F(\tilde{\gamma}_{j,p})$ is the eigenvalue of $F$
  with respect to $\tilde{\gamma}_{j,p}$.
  There exists a set of complex numbers
  $\{\mu_{0,p}^2,\mu_{1,p}^{\pm},...\mu_{m,p}^{\pm}\}$
  which satisfies
  \begin{eqnarray*}
    Q_{F,p}(z) &=&
    \prod_{i=1}^m
    \left\{ \left( 1 - \mu_{i,p} z \right)
    \left( 1 - \mu_{i,p}^{-1} z \right)
    \right\}
  \end{eqnarray*}
  and
  \begin{eqnarray*}
  \mu_{0,p}^2\mu_{1,p} \cdots \mu_{m,p}
  &=& p^{m(2k-1)/2-m(m+1)/2},
  \end{eqnarray*}
  since $\gamma_{2m-j} = \gamma_j$ $(j=0,...,m-1)$,
  $Q_{F,p}(z^{-1}) = z^{-2m} Q_{F,p}(z)$ and $Q_{F,p}(0) = 1 \neq 0$.
  Following Zhuravlev~\cite{Zhu:euler} we call the set
  $\{\mu_{0,p}^2,\mu_{1,p}^{\pm},...,\mu_{m,p}^{\pm}\}$
  \textit{the $p$-parameters} of $F$.
  The $L$-function of $F$ is defined by
  \begin{eqnarray*}
    L(s,F) &:=&
        \prod_p Q_{F,p}(p^{-s+k-3/2})^{-1}.
  \end{eqnarray*}

\section{Fourier-Jacobi expansion of  Siegel-Eisenstein series with matrix index}\label{s:fourier_matrix}

In this section we assume that $k$ is an even integer.

Let $r$ be a non-negative integer.
For $\mathcal{M} \in \mbox{Sym}_r^+$ and for an even integer $k$ 
we define the Jacobi-Eisenstein series of weight $k$ of index $\mathcal{M}$ of degree $n$ by
\begin{eqnarray}\label{df:jacobi_eisenstein}
 E_{k,\mathcal{M}}^{(n)}
 &:=& 
  \sum_{M \in \Gamma_{\infty}^{(n)}\backslash \Gamma_n}\sum_{\lambda \in \Z^{(n,r)}}
  1|_{k,\mathcal{M}}([(\lambda,0),0_r], M) .
\end{eqnarray}
The above sum converges for $k > n+r+1$ (cf.~\cite{Zi}).
The Siegel-Eisenstein series $E_k^{(n)}$ of weight $k$ of degree $n$ is defined by
\begin{eqnarray}\label{df:siegel_eisenstein}
 E_k^{(n)}(Z)
 &:=&
 \sum_{\smat{ * }{ * }{C}{D} \in \Gamma_{\infty}^{(n)}\backslash \Gamma_n} \det(C Z + D)^{-k},
\end{eqnarray}
where $Z \in \H_n$.
We denote by $e_{k,\mathcal{M}}^{(n-r)}$ the $\mathcal{M}$-th Fourier-Jacobi coefficient of $E_k^{(n)}$,
it means that
\begin{eqnarray}\label{eq:fourier_jacobi_eisenstein}
 E_k^{(n)}(\left(\begin{smallmatrix}\tau&z\\ {^t z}& \omega \end{smallmatrix}\right)) 
 &=& 
\displaystyle{\sum_{\mathcal{M} \in Sym_r^*}e_{k,\mathcal{M}}^{(n-r)}(\tau,z)\, e(\mathcal{M}\omega)}
\end{eqnarray}
is a Fourier-Jacobi expansion of the Siegel-Eisenstein series $E_k^{(n)}$ of weight $k$ of degree $n$,
where $\tau \in \H_{n-r}$, $\omega \in \H_r$ and $z \in \C^{(n-r,r)}$.
The explicit formula for the Fourier-Jacobi expansion of Siegel-Eisenstein series
is given in \cite[Satz~7]{Bo} for arbitrary degree.

The purpose of this section is to express the Fourier-Jacobi coefficient $e_{k,\mathcal{M}}^{(n-2)}$ for
$\mathcal{M} = 
  \left(\begin{smallmatrix} *&* \\ * & 1 \end{smallmatrix}\right) \in \mbox{Sym}_2^+$
as a summation of Jacobi-Eisenstein series of matrix index (Proposition~\ref{prop:fourier_jacobi}).

We first obtain the following lemma.
\begin{lemma}\label{lemma:eisen_A}
For any $\mathcal{M} \in \mbox{Sym}_r^+$ and for any $A \in \mbox{GL}_r(\Z)$ we have
 \begin{eqnarray*}
  E_{k,\mathcal{M}}^{(n)}(\tau,z)
  &=&
  E_{k,\mathcal{M}[A^{-1}]}^{(n)}(\tau,z {^t A})
\end{eqnarray*}
and
\begin{eqnarray*}
  e_{k,\mathcal{M}}^{(n)}(\tau,z)
  &=&
  e_{k,\mathcal{M}[A^{-1}]}^{(n)}(\tau,z {^t A}) .
 \end{eqnarray*}
\end{lemma}
\begin{proof}
The first identity follows directly from the definition.
The transformation formula
 $E_k^{(n+r)}\begin{pmatrix}\begin{pmatrix}1_n & \\ & A \end{pmatrix} 
                         \begin{pmatrix}\tau&z\\ {^t z}& \omega \end{pmatrix}
                         \begin{pmatrix}1_n & \\ & {^t A} \end{pmatrix} 
           \end{pmatrix}
  =
  E_k^{(n+r)}\begin{pmatrix}
            \begin{pmatrix}\tau&z\\ {^t z}& \omega \end{pmatrix}
           \end{pmatrix}
 $
gives the second identity.
\end{proof}

Let $m$ be a positive integer. We denote by $D_0$ the discriminant of $\mathbb{Q}(\sqrt{-m})$,
and we put $f := \sqrt{\frac{m}{|D_0|}}$.
We note that $f$ is a positive integer if $-m \equiv 0, 1 \!\!\mod 4$.

We denote by $h_{k-\frac12}(m)$ the $m$-th Fourier coefficient
of the Cohen-Eisenstein series of weight $k-\frac12$ (cf.~Cohen~\cite{Co}).
The following formula is known (cf.~\cite{Co},~\cite{EZ}):
\begin{eqnarray*}
\begin{aligned}
 &
 h_{k-\frac12}(m) \\
&=
 \begin{cases}
  h_{k-\frac12}(|D_0|)\, m^{k-\frac32} \sum_{d|f}  \mu(d) \left(\frac{D_0}{d} \right) d^{1-k} \sigma_{3-2k}\left(\frac{f}{d} \right)
  & \mbox{if } -m \equiv 0, 1 \mod 4,
 \\
  0
  & \mbox{otherwise},
 \end{cases}
 \end{aligned}
\end{eqnarray*}
where
we define
$\sigma_{a}(b) := \displaystyle{\sum_{d|b}d^a}$
and $\mu$ is the M\"obis function.

We assume $-m \equiv 0, 1 \!\! \mod 4$. Let $D_0$ and $f$ be as above.
For the sake of simplicity we define
\[
g_k(m) := \sum_{d|f} \mu(d)\, h_{k-\frac12}\!\left(\frac{m}{d^2}\right).
\]
We will use the following lemma for the proof of Proposition~\ref{prop:e_k_M_V}.
\begin{lemma}\label{lemma:gk}
Let $m$ be a natural number such that $-m \equiv 0$, $1 \mod 4$. 
Then for any prime $p$ we have
\begin{eqnarray*}
 g_k(p^2m) &=& 
  \left(p^{2k-3} - \left(\frac{-m}{p} \right) p^{k-2} \right) g_k(m) . 
\end{eqnarray*}
\end{lemma}
\begin{proof}
Let $D_0$, $f$ be as above.
By using the formula of $h_{k-\frac12}(m)$ we obtain
\begin{eqnarray*}
 h_{k-\frac12}(m)
 &=&
 h_{k-\frac12}(|D_0|) |D_0|^{k-\frac32}
 \prod_{q|f}
  \left\{
   \sigma_{2k-3}(q^{l_q})
   - \left(\frac{D_0}{q}\right)q^{k-2}\sigma_{2k-3}(q^{l_q-1})
  \right\} ,
\end{eqnarray*}
where $q$ runs over all primes which divide $f$,
and where we put $l_q := \mbox{ord}_q(f)$.
In particular, the function $h_{k-\frac12}(m)(h_{k-\frac12}(|D_0|) |D_0|^{k-\frac32})^{-1}$
is multiplicative as function of $f$.
Hence, for any prime $q$, we have
\begin{eqnarray*}
 \begin{aligned}
 & h_{k-\frac12}(|D_0|q^{2l_q}) - h_{k-\frac12}(|D_0|q^{2l_q-2})
\\
 &=
 h_{k-\frac12}(|D_0|) |D_0|^{k-\frac32}
 \left(
   q^{(2k-3)l_q}
   - \left(\frac{D_0}{q}\right)q^{k-2 + (2k-3)(l_q-1)}
 \right),
 \end{aligned}
\end{eqnarray*}
Thus
\begin{eqnarray*}
 g_k(m)
 &=&
 h_{k-\frac12}(|D_0|) |D_0|^{k-\frac32}
 \sum_{d|f} \mu(d) \frac{h_{k-\frac12}\left(\frac{m}{d^2}\right)}{h_{k-\frac12}(|D_0|) |D_0|^{k-\frac32}} \\
 &=&
 h_{k-\frac12}(|D_0|) |D_0|^{k-\frac32}
 \prod_{q|f}\frac{h_{k-\frac12}(|D_0| q^{2l_q}) - h_{k-\frac12}(|D_0| q^{2l_q-2})}{h_{k-\frac12}(|D_0|) |D_0|^{k-\frac32}}\\
 &=&
 h_{k-\frac12}(|D_0|) |D_0|^{k-\frac32}
 \prod_{q|f}\left(q^{(2k-3)l_q}
   - \left(\frac{D_0}{q}\right)q^{k-2 + (2k-3)(l_q-1)}\right).
\end{eqnarray*}
The lemma follows from this identity, since
 $\left(\frac{-m}{p}\right) =  0$ if $p|f$; 
$\left(\frac{-m}{p}\right) =\left(\frac{D_0}{p}\right)$ if $p {\not|} f$.
\end{proof}

By using the function $g_k(m)$, we obtain the following proposition.

\begin{prop}\label{prop:fourier_jacobi}
For $\mathcal{M} = \begin{pmatrix} * & * \\ * & 1 \end{pmatrix} \in \mbox{Sym}_2^+$ 
we put $m = \det(2\mathcal{M})$.
Let $D_0$, $f$ be as above, which depend on the integer $m$.
If $k > n + 1$, then
\begin{eqnarray*}
 e_{k,\mathcal{M}}^{(n-2)}(\tau,z)
 &=&
 \sum_{d|f} g_k\!\left(\frac{m}{d^2}\right) E_{k,\mathcal{M}[{W_d}^{-1}]}^{(n-2)}(\tau,z {^t W_d}) ,
\end{eqnarray*}
where we chose a matrix $W_d \in \mbox{GL}_2(\Qq)\cap \Z^{(2,2)}$ for each $d$ which satisfies the conditions
$\det(W_d) = d$,
${^t W_d}^{-1}\mathcal{M} {W_d}^{-1} \in  \mbox{Sym}_2^+$
and ${^t W_d}^{-1}\mathcal{M} { W_d}^{-1} = \begin{pmatrix} * & * \\ * & 1 \end{pmatrix}$.
Remark that the matrix $W_d$ is not uniquely determined, but the above summation does not depend on the choice of $W_d$.
\end{prop}
\begin{proof}
We use the terminology and the Satz 7 in~\cite{Bo} for this proof.
For $\mathcal{M}' \in \mbox{Sym}_n^+$ we denote by $a_2^k(\mathcal{M}')$ the $\mathcal{M}'$-th
Fourier coefficient of Siegel-Eisenstein series of weight $k$ of degree 2.
We put 
\begin{eqnarray*} 
 \mbox{M}_2^n(\Z)^*
 &:=&
 \left\{ N \in \Z^{(2,2)}
  \, | \, 
  \det(N)\neq 0 
  \mbox{ and there exists }
   V = \left(\begin{smallmatrix}N&*\\ *&*\end{smallmatrix}\right) \in \mbox{GL}_n(\Z)
 \right\} .
\end{eqnarray*}
A matrix $N \in \Z^{(n,2)}$ is called \textit{primitive}
if there exists a matrix $V \in \mbox{GL}_n(\Z)$
such that $V = (N\ *)$.
From~\cite[Satz~7]{Bo} we have 
\begin{eqnarray*}
 e_{k,\mathcal{M}}^{(n-2)}(\tau,z)
 &=& 
 \sum_{\begin{smallmatrix} N_1 \in M_2^n(\Z)^*/ GL_2(\Z) \\ 
                           N_1^{-1} \mathcal{M} {^t N_1}^{-1} \in Sym_2^+\end{smallmatrix}}
 a_2^k( \mathcal{M} [{^tN_1}^{-1}] )
 \sum_{\begin{smallmatrix} N_3 \in \Z^{(n-2,2)} \\ \left(\begin{smallmatrix} N_1\\N_3 \end{smallmatrix} \right) : primitive \end{smallmatrix}}
 f(\mathcal{M},N_1,N_3;\tau,z) ,
\end{eqnarray*}
where
\begin{eqnarray*}
 \begin{aligned}
 &
 f(\mathcal{M},N_1,N_3;\tau,z)
\\
 &:=
 \sum_{\left(\begin{smallmatrix}A&B\\C&D \end{smallmatrix}\right) \in \Gamma_{\infty}^{(n-2)} \backslash \Gamma_{n-2}}
 \!\!\!\!\!\!\!\!
 \det(C\tau+D)^{-k} \\
 & \qquad
  \times \, e\!\left( \mathcal{M} 
  \left\{ 
   -{^t z}(C \tau + D)^{-1}C  z + {^t z}(C \tau + D)^{-1} N_3 N_1^{-1} \right. \right. \\
& \qquad
  \left. \left.
   + {^t N_1}^{-1} {^t N_3} {^t(C \tau + D)^{-1}} z
   + {^t N_1}^{-1} {^t N_3} {(A \tau + B) (C \tau + D)^{-1} N_3 N_1^{-1}}
  \right\}\right) .
 \end{aligned}
\end{eqnarray*}

For any positive integer $l$ such that $l^2|m$,
we chose a matrix $W_l \in \Z^{(2,2)}$ which satisfies three conditions
$\det(W_l) = l$,
${^t W_l}^{-1}\mathcal{M} { W_l}^{-1} \in \mbox{Sym}_2^+$
and ${^t W_l}^{-1}\mathcal{M} { W_l}^{-1} = \begin{pmatrix} * & * \\ * & 1 \end{pmatrix}$.
By virtue of these conditions,
$W_l$ has the form
 $W_l = \begin{pmatrix}l&0\\x&1\end{pmatrix}$
with some $x \in \Z$.
The set ${^t W_l}\, \mbox{GL}_2(\Z)$ is uniquely determined for each positive integer $l$
such that $l^2|m$.
If $N_1 = {^t W_l} = \begin{pmatrix}l&x\\0&1\end{pmatrix}$, then
\begin{eqnarray*}
  \sum_{\begin{smallmatrix} N_3 \in \Z^{(n-2,2)} \\ \left(\begin{smallmatrix} N_1\\N_3 \end{smallmatrix} \right) : primitive \end{smallmatrix}}
 f(\mathcal{M},N_1,N_3;\tau,z)
&=&
 \sum_{a|l} \mu(a) \sum_{N_3 \in \Z^{(n-2,2)}}
 f(\mathcal{M},N_1,N_3 \smat{a}{0}{0}{1} ;\tau,z).
\end{eqnarray*}
Thus
\begin{eqnarray*}
 \begin{aligned}
 &
 e_{k,\mathcal{M}}^{(n-2)}(\tau,z)
\\
 &=
 \sum_{\begin{smallmatrix} l \\ l^2 | m \end{smallmatrix}} 
  a_2^k(\mathcal{M} [{W_l}^{-1}])
 \sum_{a | l} \mu(a) \!\!\!
 \sum_{N_3 \in \Z^{(n-2,2)}}
   f(\mathcal{M},{^t W_l},N_3 \left(\begin{smallmatrix}a&0\\0&1\end{smallmatrix} \right);\tau,z)
\\
 &=
 \sum_{\begin{smallmatrix} l \\ l^2 | m \end{smallmatrix}} 
  a_2^k(\mathcal{M} [{W_l}^{-1}])
 \sum_{a | l} \mu(a) \!\!\!
 \sum_{N_3 \in \Z^{(n-2,2)}}
   f(\mathcal{M} [{ W_l}^{-1} \left(\begin{smallmatrix}a&0\\0&1\end{smallmatrix} \right) ],
     1_2,N_3;
     \tau,z {^t W_l} \left(\begin{smallmatrix}a&0\\0&1\end{smallmatrix} \right)^{-1}).
  \end{aligned}
\end{eqnarray*}
Therefore
\begin{eqnarray*}
 \begin{aligned}
 &
 e_{k,\mathcal{M}}^{(n-2)}(\tau,z) \\
 &=
 \sum_{\begin{smallmatrix} l \\ l^2 | m \end{smallmatrix}} 
  a_2^k(\mathcal{M} [{W_l}^{-1}])
 \sum_{a | l} \mu(a)
 \, E_{k,\mathcal{M}[{ W_l^{-1}}\left(\begin{smallmatrix} a& \\ & 1 \end{smallmatrix}\right)]}^{(n-2)}
    (\tau,z {^t W_l} \left(\begin{smallmatrix} a^{-1}& \\ & 1 \end{smallmatrix}\right))
\\
 &=
 \sum_{\begin{smallmatrix} d \\ d^2 | m \end{smallmatrix}}
 E_{k,\mathcal{M}[{ W_d^{-1}}]}^{(n-2)}
    (\tau,z {^t W_d})
 \sum_{\begin{smallmatrix}a \\ a^2 | \frac{m}{d^2}\end{smallmatrix}}
  \mu(a)\, a_2^k(\mathcal{M} [{W_d}^{-1}\left(\begin{smallmatrix} a^{-1}& \\ & 1
    \end{smallmatrix}\right)]).
  \end{aligned}
\end{eqnarray*}
Here we have
 $a_2^k(\mathcal{M}') = h_{k-\frac12}(\det(2\mathcal{M}'))$
for any $\mathcal{M}' = \left(\begin{smallmatrix}*&*\\ *&1\end{smallmatrix}\right)\in \mbox{Sym}_2^+$.
Moreover, if
 $m \not \equiv 0,3 \mod 4$,
then
 $h_{k-\frac12}(m) = 0$.
Hence
\begin{eqnarray*}
 e_{k,\mathcal{M}}^{(n-2)}(\tau,z)
 &=&
 \sum_{\begin{smallmatrix} d \\ d | f \end{smallmatrix}}
 E_{k,\mathcal{M}[{ W_d^{-1}}]}^{(n-2)}
    (\tau,z {^t W_d})
 \sum_{\begin{smallmatrix}a \\ a | \frac{f}{d}\end{smallmatrix}}
  \mu(a)\, h_{k-\frac12}\!\!\left(\frac{m}{a^2d^2}\right).
\end{eqnarray*}
Therefore this proposition follows.
\end{proof}

\section{Relation between Jacobi forms of half-integral weight of integer index and Jacobi forms of integral weight of matrix index}
In this section we fix a positive definite half-integral symmetric matrix
$\mathcal{M} \in \mbox{Sym}_2^+$,
and we assume that $\mathcal{M}$ has the form
$\mathcal{M} = \begin{pmatrix} l & \frac12 r \\ \frac12 r & 1 \end{pmatrix}$
with integers $l$ and $r$.

The purpose of this section is to give a map $\iota_{\mathcal{M}}$
which is a map from certain holomorphic functions on $\mathfrak{H}_n \times \C^{(n,2)}$
to holomorphic functions on $\mathfrak{H}_n \times \C^{(n,1)}$.
A restriction of $\iota_{\mathcal{M}}$ gives a map from a certain subspace
$J_{k,\, \mathcal{M}}^{(n)*}$ of $J_{k,\, \mathcal{M}}^{(n)}$
to a certain subspace
$J_{k-\frac12,\, \det(2\mathcal{M})}^{(n)*}$ of $J_{k-\frac12,\, \det(2\mathcal{M})}^{(n)}$
(cf. Lemma~\ref{lemma:iota}).
Moreover, we shall show a compatibility between the map $\iota_{\mathcal{M}}$
and index-shift maps 
(cf. Proposition~\ref{prop:iota_U} and Proposition~\ref{prop:iota_hecke}).
Furthermore, we define index-shift maps $\tilde{V}_{\alpha,n-\alpha}(p^2)$ for $J_{k-\frac12,\, \det(2\mathcal{M})}^{(n)*}$ at $p=2$
through the map $\iota_{\mathcal{M}}$ (cf. \S\ref{ss:hecke_p2}).

By virtue of the map $\iota_{\mathcal{M}}$
and by the results in this section,
we can translate some relations among
Jacobi forms of \textit{half-integral weight of integer index}
to relations among Jacobi forms of \textit{integral weight of matrix index}.

\subsection{An expansion of Jacobi forms of integer index}\label{ss:fj_expansion}

In this subsection
we consider an expansion of Jacobi forms of integer index
and shall introduce a certain subspace
$J_{k,\mathcal{M}}^{(n)*} \subset J_{k,\mathcal{M}}^{(n)}$.

The symbol $J_{k,1}^{(n+1)}$ denotes
the space of Jacobi forms of weight $k$ of index $1$ of degree $n+1$
(cf. \S\ref{ss:jacobi_forms_of_matrix_index}).

Let $\phi_1(\tau,z) \in J_{k,1}^{(n+1)}$ be a Jacobi form.
We regard $\phi_1(\tau,z)\, e(\omega)$ as a holomorphic function on $\H_{n+2}$,
where $\tau \in \H_{n+1}$, $z \in \C^{(n+1,1)}$ and $\omega \in \H_1$
such that $\smat{\tau}{z}{^t z}{\omega} \in \H_{n+2}$.
We have an expansion
\begin{eqnarray*}
 \phi_1(\tau,z) e(w)
 &=&
 \sum_{\begin{smallmatrix}
        S \in Sym_2^+ \\
        S = \smat{ *}{ *}{ *}{ 1}
       \end{smallmatrix}}
      \phi_{\mathcal{S}}(\tau',z') e(\mathcal{S} \omega'),
\end{eqnarray*}
where $\tau' \in \H_n$, $z' \in \C^{(n,2)}$ and $\omega' \in \H_2$ which satisfy
$\smat{\tau}{z}{^t z}{\omega} = \smat{\tau'}{z'}{^t z'}{\omega'} \in \H_{n+2}$.
Because
the group $\Gamma_{n,2}^J$ (cf. \S~\ref{ss:Jacobi_group}) is a subgroup of $\Gamma_{n+1,1}^J$,
the form $\phi_{\mathcal{S}}$ belongs to $J_{k,\mathcal{S}}^{(n)}$.
We denote this map by $\mbox{FJ}_{1,\mathcal{S}}$, it means that we have a map
\begin{eqnarray*}
 \mbox{FJ}_{1,\mathcal{S}}: J_{k,1}^{(n+1)} \rightarrow J_{k,\mathcal{S}}^{(n)}.
\end{eqnarray*}
By an abuse of language, we call the map $\mbox{FJ}_{1,\mathcal{S}}$
\textit{the Fourier-Jacobi expansion with respect to $S$}.

The $\C$-vector subspace $J_{k,\mathcal{M}}^{(n)*}$ of $J_{k,\mathcal{M}}^{(n)}$
denotes the image of $J_{k,1}^{(n+1)}$
by $\mbox{FJ}_{1,\mathcal{M}}$,
where $\mathcal{M}$ is a half-integral symmetric matrix of size $2$.

\subsection{Fourier-Jacobi expansion of Siegel modular forms of half-integral weight}
\label{ss:fourier_jacobi_expansion_half}
The purpose of this subsection is to show the following lemma.
\begin{lemma} \label{lemma:fj_half}
Let $F\left(\smat{\tau}{z}{^t z}{\omega} \right)= \sum_{m \in \Z}\phi_m(\tau,z) e(m\omega)$
be a Fourier-Jacobi expansion of $F \in M_{k-\frac12}(\Gamma_0^{(n+1)}(4))$,
where $\tau \in \H_n$, $\omega \in \H_1$ and $z \in \C^{(n,1)}$.
Then $\phi_m \in J_{k-\frac12,m}^{(n)}$ for any natural number $m$.
\end{lemma}
\begin{proof}
Due to the definition of $J_{k-\frac12,m}^{(n)}$,
we only need to show the identity
\begin{eqnarray*}
 \theta^{(n+1)}(\gamma \cdot \smat{\tau}{z}{^t z}{\omega})\, \theta^{(n+1)}(\smat{\tau}{z}{^t z}{\omega})^{-1}
 &=&
 \theta^{(n)}(\smat{A}{B}{C}{D} \cdot \tau)\, \theta^{(n)}(\tau)^{-1}
\end{eqnarray*}
for any $\gamma = (\smat{A}{B}{C}{D}, [(\lambda,\mu),\kappa]) \in \Gamma_{n,1}^J$
and for any $\smat{\tau}{z}{^t z}{\omega} \in \H_{n+1}$
such that $\tau \in \H_n$, $\omega \in \H_1$.
Here $\theta^{(n+1)}$ and $\theta^{(n)}$ are the theta constants (cf. \S\ref{ss:double_covering_groups}).

For any $M = \smat{A'}{B'}{C'}{D'} \in \Gamma_0^{(n+1)}(4)$
it is known that 
\begin{eqnarray*}
 \left(\theta^{(n+1)}(M \cdot Z)\, \theta^{(n+1)}(Z)^{-1}\right)^2
 &=&
 \det(C' Z + D') \left(\frac{-4}{\det D'}\right),
\end{eqnarray*}
where $Z \in \H_{n+1}$.
Here $\left(\frac{-4}{\det D'}\right)$ is the quadratic symbol
and it is known the identity $\left(\frac{-4}{\det D'}\right) = (-1)^{\frac{\det D' - 1}{2}}$.
Hence, for any $\gamma = (\smat{A}{B}{C}{D}, [(\lambda,\mu),\kappa]) \in \Gamma_{n,1}^J$,
we obtain
\begin{eqnarray*}
 \left(\theta^{(n+1)}(\gamma \cdot Z)\, \theta^{(n+1)}(Z)^{-1}\right)^2
 &=&
 \det(C \tau + D) \left(\frac{-4}{\det D}\right),
\end{eqnarray*}
where $Z = \smat{\tau}{z}{^t z}{\omega} \in \H_{n+1}$ with $\tau \in \H_n$.
In particular, the holomorphic function $\frac{\theta^{(n+1)}(\gamma\cdot Z)}{ \theta^{(n+1)}(Z)}$
does not depend on the choice of
$z \in \C^{(n,1)}$ and of $\omega \in \H_1$.
We substitute $z = 0$ into $\frac{\theta^{(n+1)}(\gamma\cdot Z)}{ \theta^{(n+1)}(Z)}$ and a straightforward calculation gives
\begin{eqnarray*}
 \frac{\theta^{(n+1)}(\gamma \cdot \smat{\tau}{0}{0}{\omega})}{\theta^{(n+1)}(\smat{\tau}{0}{0}{\omega})}
 &=&
 \frac{\theta^{(n)}(\smat{A}{B}{C}{D}\cdot \tau)}{\theta^{(n)}(\tau)}.
\end{eqnarray*}
Hence we conclude this lemma.
\end{proof}

\subsection{The map $\sigma$ and
the Hecke operator $\tilde{T}_{\alpha,n-\alpha}(p^2)$}\label{ss:fj_expansion_half}
In this subsection we review the isomorphism between the space of
Jacobi forms of index $1$ and a certain subspace of Siegel modular forms
of half-integral weight, which has been shown
by Eichler-Zagier\cite{EZ} for degree one and by Ibukiyama\cite{Ib} for general degree.

Let $M_{k-\frac12}^+(\Gamma_0^{(n)}(4))$ be the generalized plus-space
introduced in~\cite[page 112]{Ib},
which is a generalization of the Kohnen plus-space for higher degrees:
\begin{eqnarray*}
 M_{k-\frac12}^+(\Gamma_0^{(n)}(4)) 
 &:=& \left\{ F \in M_{k-\frac12}(\Gamma_0^{(n)}(4))  \left| 
 \begin{matrix} \mbox{the coefficients }
 A(N) = 0 \mbox{ unless } \\
 N+ (-1)^k R {^t R} \in 4\, \mbox{Sym}_{n}^* \\
 \mbox{ for some } R \in \Z^{(n,1)}
 \end{matrix}
 \right\}
 \right.
 .
\end{eqnarray*}

A form $F \in M_{k-\frac12}(\Gamma_0^{(n)}(4))$ is called a Siegel cusp form
if $F^2$ is a Siegel cusp form of weight $2k-1$.
We denote by $S_{k-\frac12}^+(\Gamma_0^{(n)}(4))$ the space of
all Siegel cusp forms in $M_{k-\frac12}^+(\Gamma_0^{(n)}(4))$.

For any even integer $k$,
the isomorphism between $J_{k,1}^{(n)}$
(the space of Jacobi forms of weight $k$ of index $1$ of degree $n$)
and $M_{k-\frac12}^+(\Gamma_0^{(n)}(4))$ is shown in~\cite[Theorem 5.4]{EZ} 
for $n=1$ and in~\cite[Theorem 1]{Ib} for $n > 1$.
We call this isomorphism
\textit{the Eichler-Zagier-Ibukiyama correspondence} and denote
this linear map by $\sigma$ which is a bijection
from $J_{k,1}^{(n)}$ to $M_{k-\frac12}^+(\Gamma_0^{(n)}(4))$
as modules over the ring of Hecke operators.
By the map $\sigma$ the space $S_{k-\frac12}^+(\Gamma_0^{(n)}(4))$
is isomorphic to the space of Jacobi cusp forms $J_{k,1}^{(n)\, cusp}$.
The map
\begin{eqnarray*}
 \sigma : J_{k,1}^{(n)} \rightarrow M_{k-\frac12}^+(\Gamma_0^{(n)}(4))
\end{eqnarray*}
is given as follows:
if
\begin{eqnarray*}
\phi(\tau,z) =   \sum_{\begin{smallmatrix} 
        N \in Sym_{n}^*, \, R \in \Z^{(n,1)}\\
        4N - R{^t R} \geq 0
       \end{smallmatrix}}
       C(N,R)\, e(N\tau + R{^t z})
\end{eqnarray*}
is a Jacobi form which belongs to $J_{k,1}^{(n)}$,
then $\sigma(\phi) \in M_{k-\frac12}^+(\Gamma_0^{(n)}(4))$ is defined by
\begin{eqnarray*}
 \sigma(\phi)(\tau) &:=&
 \sum_{\begin{smallmatrix}
         R\!\!\! \mod  (2\Z)^{(n,1)} \\
         R \in \Z^{(n,1)}
       \end{smallmatrix}}
 \sum_{\begin{smallmatrix}
         N \in Sym_{n}^* \\
         4N - R{^t R} \geq 0
       \end{smallmatrix}}
    C(N,R)\, e( (4N - R {^t R}) \tau).
\end{eqnarray*}

For the double coset
$\Gamma_{n} \mbox{diag}(1_{\alpha}, p 1_{n-\alpha},p^2 1_{\alpha}, p 1_{n-\alpha})
\Gamma_{n}$ and for $\phi \in J_{k,1}^{(n)}$, the Hecke
operator $T_{\alpha,n-\alpha}^J(p^2)$ is defined by
\begin{eqnarray*}
  \phi|T_{\alpha,n-\alpha}^J(p^2)
  &:=&
  \sum_{\lambda,\mu \in (\Z/p\Z)^{n}}
  \sum_{M}
  \phi|_{k,1} \left(M\times\smat{p}{0}{0}{p} ,[(\lambda,\mu),0]\right).
\end{eqnarray*}
Here, in the second summation of the RHS, the matrix $M$ runs over all representatives
of
 $\Gamma_{n}\backslash
   \Gamma_{n}\, 
   \mbox{diag}(1_{\alpha}, p 1_{n-\alpha},p^2 1_{\alpha}, p 1_{n-\alpha})
   \Gamma_{n}$.
Let $\tilde{T}_{\alpha,n-\alpha}(p^2)$ be the Hecke operator introduced
in \S\ref{ss:hecke_op_siegel_half} for odd prime $p$,
which acts on the space $M_{k-\frac12}(\Gamma_0^{(n)}(4))$.
For any odd prime $p$ the identity
\begin{eqnarray} \label{eq:sigma_isom_hecke}
  \sigma(\phi)|\tilde{T}_{\alpha,n-\alpha}(p^2)
  &=&
    p^{\alpha/2 + k(2n+1) - (2n+7)n/2} \sigma(\phi|T_{\alpha,n-\alpha}^J(p^2)).
\end{eqnarray}
has been obtained in~\cite{Ib}.

Through the identity (\ref{eq:sigma_isom_hecke}) we define the Hecke operator
$\tilde{T}_{\alpha,n-\alpha}(4)$ for $M_{k-\frac12}^{+}(\Gamma_0^{(n)}(4))$,
it means
\begin{eqnarray}\label{eq:half_hecke_2}
 \sigma(\phi)|\tilde{T}_{\alpha,n-\alpha}(4)
  &:=&
    2^{\alpha/2 + k(2n+1) - (2n+7)n/2} \sigma(\phi|T_{\alpha,n-\alpha}^J(4)).
\end{eqnarray}

\subsection{A generalization of Cohen-Eisenstein series and
the subspace $J_{k-1/2}^{(n)*}$}\label{ss:CE_J*}
In this subsection
we will introduce a subspace
$J_{k-\frac12,m}^{(n)*} \subset J_{k-\frac12,m}^{(n)}$
for any integer $n$.
Moreover, we will introduce a generalized Cohen-Eisenstein series
$\mathcal{H}_{k-\frac12}^{(n+1)}$
and will consider the Fourier-Jacobi expansion of $\mathcal{H}_{k-\frac12}^{(n+1)}$
for any integer $n$.

Let $e_{k,1}^{(n+1)}$ be the first Fourier-Jacobi coefficient of Siegel-Eisenstein series $E_k^{(n+2)}$
(see (\ref{eq:fourier_jacobi_eisenstein}) in \S\ref{s:fourier_matrix} for the definition of $e_{k,1}^{(n+1)}$).
It is known that $e_{k,1}^{(n+1)}$ coincides with the Jacobi-Eisenstein series
$E_{k,1}^{(n+1)}$ of weight $k$ of index $1$ of degree $n+1$
(cf. \cite[Satz 7]{Bo}. See (\ref{df:jacobi_eisenstein}) in \S\ref{s:fourier_matrix}
for the definition of $E_{k,1}^{(n+1)}$).

We define the \textit{generalized Cohen-Eisenstein series} $\mathcal{H}_{k-\frac12}^{(n+1)}$ of weight $k-\frac12$ of degree $n+1$ by
\begin{eqnarray*}
 \mathcal{H}_{k-\frac12}^{(n+1)}
 &:=&
 \sigma(E_{k,1}^{(n+1)}).
\end{eqnarray*}
Because $E_{k,1}^{(n+1)} \in J_{k,1}^{(n+1)}$,
we have $\mathcal{H}_{k-\frac12}^{(n+1)} \in M_{k-\frac12}^+(\Gamma_0^{(n+1)}(4))$
for any integer $n$.

For any integer $m$ we denote by $\widetilde{\mbox{FJ}}_m$ the linear map from 
$M_{k-\frac12}(\Gamma_0^{(n+1)}(4))$ to $J_{k-\frac12,m}^{(n)}$
obtained by the Fourier-Jacobi expansion with respect to the index $m$.
It means that if $G \in M_{k-\frac12}(\Gamma_0^{(n+1)}(4))$, then $G$ has the expansion
\begin{eqnarray*}
 G\!\left(\begin{pmatrix} \tau & z \\ ^t z & \omega \end{pmatrix}\right)
 &=&
 \sum_{m \in \Z} \phi_m(\tau,z) e(m\omega)
\end{eqnarray*}
and we define $\widetilde{\mbox{FJ}}_m(G) := \phi_m$.
We remark $\phi_m \in J_{k-\frac12,m}^{(n)}$ due to Lemma~\ref{lemma:fj_half}.

We denote by $J_{k-\frac12,m}^{(n)*}$ the image of
$M_{k-\frac12}^+(\Gamma_0^{(n+1)}(4))$
by the map $\widetilde{\mbox{FJ}}_m$.

We denote by $e_{k-\frac12,m}^{(n)}$ the $m$-th Fourier-Jacobi coefficient of
the generalized Cohen-Eisenstein series $\mathcal{H}_{k-\frac12}^{(n+1)}$
(see (\ref{eq:df_fourier_jacobi_cohen_eisenstein}) in \S\ref{s:introduction}
for the definition of $e_{k-\frac12,m}^{(n)}$).
We remark $e_{k-\frac12,m}^{(n)} \in J_{k-\frac12,m}^{(n)*}$
for any integer $n$.

\subsection{The map $\iota_{\mathcal{M}}$ }\label{ss:iota}
We recall
$\mathcal{M} = \left(\begin{smallmatrix} l & r/2 \\ r/2 & 1 \end{smallmatrix}\right) \in \mbox{Sym}_2^+$.
In this subsection we shall introduce a map
\begin{eqnarray*}
\iota_{\mathcal{M}} : H_{\mathcal{M}}^{(n)} \rightarrow \mbox{Hol}(\H_n\times \C^{(n,1)} \rightarrow \C),
\end{eqnarray*}
where $H_{\mathcal{M}}^{(n)}$ is a certain subspace of holomorphic functions on $\mathfrak{H}_n\times \C^{(n,2)}$,
which will be defined below,
and where $\mbox{Hol}(\H_n\times \C^{(n,1)} \rightarrow \C)$ denotes the space of all holomorphic functions
on $\H_n\times \C^{(n,1)}$.
We will show that
the restriction of $\iota_{\mathcal{M}}$ gives a linear isomorphism
between $J_{k,\mathcal{M}}^{(n)*}$ and  $J_{k-\frac12,m}^{(n)*}$ (cf. Lemma~\ref{lemma:iota}).

Let $\psi$ be a holomorphic function on $\H_n \times \C^{(n,2)}$.
We assume that $\psi$ has a Fourier expansion
\begin{eqnarray*}
 \psi(\tau,z) 
 &=& 
 \sum_{\begin{smallmatrix} N \in Sym_n^* , R \in \Z^{(n,1)}\\
                          4 N - R \mathcal{M}^{-1} {^tR} \geq 0 \end{smallmatrix}}A(N,R)\, e(N\tau + ^t R z)
\end{eqnarray*}
for $(\tau,z) \in \H_n \times \C^{(n,2)}$,
and assume that 
$\psi$ satisfies the following condition on the Fourier coefficients:
if 
\begin{eqnarray*}
 \begin{pmatrix}
  N & \frac12 R \\
  \frac12 ^t R & \mathcal{M}
 \end{pmatrix}
 &=&
 \begin{pmatrix}
  N' & \frac12 R' \\
  \frac12 ^t R' & \mathcal{M}
 \end{pmatrix}
 \left[\begin{pmatrix}
  1_n &  \\
  ^tT   & 1_2
 \end{pmatrix}\right]
\end{eqnarray*}
with some $T = \begin{pmatrix} 0 , \lambda \end{pmatrix} \in \Z^{(n,2)}$
and some $\lambda \in \Z^{(n,1)}$,
then $A(N,R) = A(N',R')$.

The symbol $H_{\mathcal{M}}^{(n)}$ denotes the $\C$-vector space consists of
all holomorphic functions which satisfy the above condition.

We remark $J_{k,\mathcal{M}}^{(n)*} \subset J_{k,\mathcal{M}}^{(n)} \subset H_{\mathcal{M}}^{(n)}$ for any even integer $k$.

Now we shall define a map $\iota_{\mathcal{M}}$.
For $\psi(\tau',z') = \sum A(N,R) e(N\tau' + R {^t z'}) \in H_{\mathcal{M}}^{(n)}$
we define a holomorphic function $\iota_{\mathcal{M}}(\psi)$
on $\H_n \times \C^{(n,1)}$ by
\begin{eqnarray*}
 \iota_{\mathcal{M}}(\psi)(\tau,z)
 &:=&
 \sum_{\begin{smallmatrix}
        M \in Sym_n^*,\ S \in \Z^{(n,1)} \\
        4 M m - S {^t S} \geq 0 
       \end{smallmatrix}}
       C(M,S) e(M\tau + S {^t z})
\end{eqnarray*}
for $(\tau,z)\in \H_n \times \C^{(n,1)}$,
where we define $C(M,S) := A(N,R)$ if there exist matrices $N \in \mbox{Sym}_2^*$ and $R = (R_1,R_2) \in \Z^{(n,2)}$
$(R_1,R_2 \in \Z^{(n,1)})$
which satisfy 
\begin{eqnarray*}
 \begin{pmatrix}
  M & \frac12 S \\ \frac12 ^t S & \det(2\mathcal{M})
 \end{pmatrix}
 &=&
 4 \begin{pmatrix}
    N & \frac12 R_1 \\ \frac12 ^t R_1 & l
   \end{pmatrix}
 - \begin{pmatrix}
    R_2 \\ r 
   \end{pmatrix}
   \begin{pmatrix}
     ^t R_2 , r
    \end{pmatrix} ,
\end{eqnarray*}
$C(M,S) := 0 $ otherwise.
We remark that the identity
\begin{eqnarray*}
   4 \begin{pmatrix}
    N & \frac12 R_1 \\ \frac12 ^t R_1 & l
   \end{pmatrix}
 - \begin{pmatrix}
    R_2 \\ r 
   \end{pmatrix}
   \begin{pmatrix}
     ^t R_2 , r
    \end{pmatrix}
    &=&
    4
    \begin{pmatrix}
     N & \frac12 R \\
     \frac12 {^t R} & \mathcal{M}
    \end{pmatrix}
    \left[
      \begin{pmatrix}
       1_{n} & \begin{matrix} 0 \\ \vdots \\ 0 \end{matrix} \\
        0 \cdots 0  & 1 \\
        - \frac12 {^t R_2} & -\frac12 r
      \end{pmatrix}
    \right]
\end{eqnarray*}
holds and remark that the coefficient $C(M,S)$ does not depend on the choice
of the matrices $N$ and $R$.
The proof of these facts are as follows.
The first  fact of the identity follows from a straightforward calculation.
As for the second fact,
if
\begin{eqnarray*}
 4 \begin{pmatrix}
    N & \frac12 R_1 \\ \frac12 ^t R_1 & l
   \end{pmatrix}
 - \begin{pmatrix}
    R_2 \\ r 
   \end{pmatrix}
   \begin{pmatrix}
     ^t R_2 , r
    \end{pmatrix} 
&=&
 4 \begin{pmatrix}
    N' & \frac12 R'_1 \\ \frac12 ^t R'_1 & l
   \end{pmatrix}
 - \begin{pmatrix}
    R'_2 \\ r 
   \end{pmatrix}
   \begin{pmatrix}
     ^t R'_2 , r
    \end{pmatrix} ,
\end{eqnarray*}
then $4N - R_2 {^t R_2} = 4N' - {R'_2} ^t {R'_2}$. Hence $R_2 {^t R_2} \equiv  {R'_2} ^t {R'_2} \!\! \mod 4$.
Thus there exists a matrix $\lambda \in \Z^{(n,1)}$ such that ${R'_2} = R_2 + 2 \lambda$.
Therefore, by straightforward calculation we have
\begin{eqnarray*}
 \begin{pmatrix}
  N & \frac12 R \\
  \frac12 ^t R & \mathcal{M}
 \end{pmatrix}
 &=&
 \begin{pmatrix}
  N' & \frac12 R' \\
  \frac12 ^t R' & \mathcal{M}
 \end{pmatrix}
 \left[
 \begin{pmatrix}
 1_n & 0 \\
  ^t T & 1_2 
 \end{pmatrix}
 \right]
\end{eqnarray*}
with $T = (0, \lambda)$, $R = (R_1, R_2)$ and $R' = ({R'}_1, {R'}_2)$.
Because $\psi$ belongs to $H_{\mathcal{M}}^{(n)}$,
we have $A(N,R) = A(N',R')$.
Hence the above definition of $C(M,S)$ is well-defined.

\begin{lemma}\label{lemma:iota}
Let $k$ be an even integer. We put $m = \det(2\mathcal{M})$.
Then we have the commutative diagram:
$$
\begin{CD}
 J_{k,1}^{(n+1)} @>\sigma >> M_{k-\frac12}^+(\Gamma_0^{(n+1)}(4)) \\
 @V{\mbox{FJ}_{1,\mathcal{M}}}VV
 @VV{\widetilde{\mbox{FJ}}_m
 }V \\
 J_{k,\mathcal{M}}^{(n)*} @>\iota_{\mathcal{M}}
 >> J_{k-\frac12,m}^{(n)*} ,
\end{CD}
$$
where two maps $\mbox{FJ}_{1,\mathcal{M}}$ and $\widetilde{\mbox{FJ}}_m$ have been introduced
in \S\ref{ss:fj_expansion} and \S\ref{ss:CE_J*}.
Moreover, the restriction of the linear map $\iota_{\mathcal{M}}$ on $J_{k,\mathcal{M}}^{(n)*}$
gives the bijection between $J_{k,\mathcal{M}}^{(n)*}$ and $J_{k-\frac12,m}^{(n)*}$.
\end{lemma}
\begin{proof}
Let $\psi \in J_{k,1}^{(n+1)}$ be a Jacobi form.
Due to the definition of $\sigma$ (cf. \S\ref{ss:fj_expansion_half}) and
$\iota_{\mathcal{M}}$,
it is not difficult to check the identity
$\iota_{M}(FJ_{1,\mathcal{M}}(\psi)) = \widetilde{\mbox{FJ}}_m (\sigma(\psi))$.
Namely, we have the above commutative diagram.

Since the restriction of the map
$\widetilde{\mbox{FJ}}_m$ on ${M_{k-\frac12}^+(\Gamma_0^{(n+1)}(4))}$
is surjective,
and since $\sigma$ is an isomorphism and since
$\iota_{M}(FJ_{1,\mathcal{M}}(\psi)) = \widetilde{\mbox{FJ}}_m (\sigma(\psi))$,
the restricted map
$\iota_{\mathcal{M}}|_{J_{k,\mathcal{M}}^{(n)*}} :
  J_{k,\mathcal{M}}^{(n)*} \rightarrow J_{k-\frac12,m}^{(n)*}$ is surjective.
The injectivity of the restricted map
$\iota_{\mathcal{M}}|_{J_{k,\mathcal{M}}^{(n)*}}$
follows directly from the definition of the map $\iota_{\mathcal{M}}$.
\end{proof}

\subsection{Compatibility between index-shift maps and $\iota_{\mathcal{M}}$}\label{ss:compati}
In this subsection we shall show a compatibility between the map $\iota_{\mathcal{M}}$ and some index-shift maps.

For function $\psi$ on $\H_n \times \C^{(n,2)}$ and for $L \in \Z^{(2,2)}$
we define the function $\psi|U_L$ on $\H_n\times\C^{(n,2)}$ by
\begin{eqnarray*}
  (\psi|U_L)(\tau,z) &:=& \psi(\tau, z {^t L}) .
\end{eqnarray*}
It is not difficult to check that if $\psi$ belongs to $J_{k,\mathcal{M}}^{(n)}$,
then $\psi|U_L$ belongs to $J_{k,\mathcal{M}\left[L\right]}^{(n)}$.

For function $\phi$ on $\H_n \times \C^{(n,1)}$ and for integer $a$
we define the function $\phi|U_a$ on $\H_n \times \C^{(n,1)}$ by
\begin{eqnarray*}
   (\phi|U_a)(\tau,z) &:=& \phi(\tau,a z).
\end{eqnarray*}
We have $\phi|U_a \in J_{k-\frac12,ma^2}^{(n)}$ if $\phi \in J_{k-\frac12,m}^{(n)}$.

\begin{prop}\label{prop:iota_U}
For any $\psi \in J_{k,\mathcal{M}}^{(n)*}$ and for any $L = \smat{a}{ }{b}{1} \in \Z^{(2,2)}$ we obtain
\begin{eqnarray*}
 \iota_{\mathcal{M}[L]}(\psi|U_L)
 &=&
 \iota_{\mathcal{M}} (\psi)|U_a.
\end{eqnarray*}
In particular, for any prime $p$ we have
$\iota_{\mathcal{M}[\smat{p}{ }{ }{1}]} \left( \psi\left| U_{\smat{p}{ }{ }{1}}\right) \right. = 
 \iota_{\mathcal{M}} (\psi)|U_p$.
\end{prop}
\begin{proof}
We put $m = \det(2 \mathcal{M})$.
Let $\psi(\tau,z') = \displaystyle{
  \!\!\!\!\!\!\!\!\!\!
  \sum_{\begin{smallmatrix}
          N \in Sym_n^*,\  R\in \Z^{(n,2)} \\
          4 N - R \mathcal{M}^{-1} {^t R} \geq 0 
  \end{smallmatrix}}
  \!\!\!\!\!\!\!\!\!\!
  A(N,R) e(N\tau + R { ^t z'})
 }$ be a Fourier expansion of $\psi$.
Let
\begin{eqnarray*}
 \iota_{\mathcal{M}}(\psi)(\tau,z) 
 &=&
 \!\!\!\!\!\!\!\!\!\!
 \sum_{\begin{smallmatrix}
            M \in Sym_n^*,\ S\in \Z^{(n,1)} 
         \\ 4 M m - S {^t S} \geq 0 
       \end{smallmatrix}}
 \!\!\!\!\!\!\!\!\!\!
 C(M,S)\, e(M\tau + S{^t z}),
\\
 \iota_{\mathcal{M}[L]}(\psi|U_L)(\tau,z)
 &=&
 \!\!\!\!\!\!\!\!\!\!
 \sum_{\begin{smallmatrix}
            M \in Sym_n^*,\ S\in \Z^{(n,1)} 
         \\ 4 M m a^2 - S {^t S} \geq 0 
       \end{smallmatrix}}
 \!\!\!\!\!\!\!\!\!\!
 C_1(M,S)\, e(M\tau + S{^t z}) 
\end{eqnarray*}
and
\begin{eqnarray*}
 (\iota_{\mathcal{M}} (\psi)|U_a)(\tau,z)
 &=&
 \!\!\!\!\!\!\!\!\!\!
 \sum_{\begin{smallmatrix}
            M \in Sym_n^*,\ S\in \Z^{(n,1)} 
         \\ 4 M m a^2 - S {^t S} \geq 0 
       \end{smallmatrix}}
 \!\!\!\!\!\!\!\!\!\!
 C_2(M,S)\, e(M\tau + S{^t z})
\end{eqnarray*}
be Fourier expansions.
It is enough to show
$C_1(M,S) = C_2(M,S)$.

We have $C_2(M,S) = C(M,a^{-1}S)$.
Moreover, we obtain
$C_1(M,S) = A(N,R L^{-1})$ with $N \in \mbox{Sym}_n^*$ and $R \in \Z^{(n,2)}$
which satisfy
\begin{eqnarray*}
 \begin{pmatrix}
  M & \frac12 S \\
 \frac12 {^t S} & m a^2
 \end{pmatrix}
 &=&
 4 \begin{pmatrix}
  N & \frac12 R\\
 \frac12 {^t R} & \mathcal{M}[L]
 \end{pmatrix}
 \left[
 \begin{pmatrix}
  1_n & \begin{matrix} 0 \\ \vdots \\ 0 \end{matrix}\\
  \begin{matrix} 0 \cdots 0 \end{matrix} & 1 \\
  -\frac12 {^t (R \left(\begin{smallmatrix}0\\ 1 \end{smallmatrix}\right))} & -\frac12 r a - b
 \end{pmatrix}
 \right] .
\end{eqnarray*}
For the above matrices $N$, $R$, $M$ and $S$ we have
\begin{eqnarray*}
 \begin{pmatrix}
  M & \frac{1}{2} a^{-1} S \\
 \frac{1}{2}a^{-1}{^t S} & m
 \end{pmatrix}
 &=&
 4 \begin{pmatrix}
  N & \frac12 R\\
 \frac12 {^t R} & \mathcal{M}[L]
 \end{pmatrix}
 \left[
 \begin{pmatrix}
  1_n & \begin{matrix} 0 \\ \vdots \\ 0 \end{matrix} \\
  \begin{matrix} 0 \cdots 0 \end{matrix} & 1 \\
  -\frac12 {^t (R \left(\begin{smallmatrix}0\\ 1 \end{smallmatrix}\right))} & -\frac12 r a - b
  \end{pmatrix}
 \begin{pmatrix}
  1_n & \begin{matrix} 0 \\ \vdots \\ 0 \end{matrix} \\
  \begin{matrix} 0  \cdots  0 \end{matrix} & a^{-1}
 \end{pmatrix}
 \right]
\\
 &=&
 4 \begin{pmatrix}
  N & \frac12 R  L^{-1}\\
 \frac12 {^t (R L^{-1})} & \mathcal{M}
 \end{pmatrix}
 \left[
 \begin{pmatrix}
  1_{n} & \begin{matrix} 0 \\ \vdots \\ 0 \end{matrix} \\
   0 \cdots 0 & 1 \\
  -\frac12 {^t (R \left(\begin{smallmatrix}0\\ 1 \end{smallmatrix}\right))} & -\frac12 r 
  \end{pmatrix}
 \right]\\
 &=&
 4 \begin{pmatrix}
  N & \frac12 R  L^{-1}\\
 \frac12 {^t (R L^{-1})} & \mathcal{M}
 \end{pmatrix}
 \left[
 \begin{pmatrix}
  1_{n} & \begin{matrix} 0 \\ \vdots \\ 0 \end{matrix} \\
  0 \cdots 0 & 1 \\
  -\frac12 {^t (R L^{-1}\left(\begin{smallmatrix}0\\ 1 \end{smallmatrix}\right))} & -\frac12 r 
  \end{pmatrix}
 \right] .
\end{eqnarray*}
Thus $C_2(M,S) = C(M,a^{-1}S) = A(N,R L^{-1}) = C_1(M,S)$.
\end{proof}

\begin{prop} \label{prop:iota_hecke}
For odd prime $p$ and for $0\leq \alpha \leq n$,
let $\tilde{V}_{\alpha,n-\alpha}(p^2)$ and $V_{\alpha,n-\alpha}(p^2)$ be index-shift maps
defined in \S\ref{ss:hecke_operators}.
Then, for any $\psi \in J_{k,\mathcal{M}}^{(n)*}$ we have
\begin{eqnarray}\label{id:iota_hecke}
 \iota_{\mathcal{M}}(\psi)| \tilde{V}_{\alpha,n-\alpha}(p^2)
 &=&
 p^{k(2n+1) - n (n+\frac72) + \frac12 \alpha }\
 \iota_{\mathcal{M}[\smat{p}{ }{ }{1}]}(\psi| V_{\alpha,n-\alpha}(p^2)).
\end{eqnarray}
\end{prop}
\begin{proof}
The proof is similar to the case of
Jacobi forms of index $1$ (cf. ~\cite[Theorem~2]{Ib}).
However,
we remark that
the maps $\tilde{V}_{\alpha,n-\alpha}(p^2)$ and $V_{\alpha,n-\alpha}(p^2)$
in the present article change the indices of Jacobi forms.

To prove this proposition,
we compare the Fourier coefficients of the both sides of (\ref{id:iota_hecke}).
Let 
\begin{eqnarray*}
 \psi(\tau,z') 
 &=& 
 \sum_{N,R}A_1(N,R) e(N\tau + R{^t z'}),
\\
 (\psi|V_{\alpha,n-\alpha}(p^2))(\tau,z')
 &=&
 \sum_{N,R}A_2(N,R) e(N\tau + R{^t z'}),
\\
 (\iota_{\mathcal{M}}(\psi))(\tau,z)
 &=&
 \sum_{M,S}C_1(M,S) e(M\tau + S{^t z})
\end{eqnarray*}
and
\begin{eqnarray*}
 (\iota_{\mathcal{M}}(\psi)|\tilde{V}_{\alpha,n-\alpha}(p^2))(\tau,z)
 &=&
 \sum_{M,S}C_2(M,S) e(M\tau + S{^t z})
\end{eqnarray*}
be Fourier expansions,
where $\tau \in \H_n$, $z' \in \C^{(n,2)}$ and $z \in \C^{(n,1)}$.
For the sake of simplicity
we put $U = \smat{p^2}{ }{ }{p}$.
Then
\begin{eqnarray*}
 \begin{aligned}
 &
 \psi| V_{\alpha,n-\alpha}(p^2) \\
 &=
 \sum_{\smat{p^2 {^t D}^{-1}}{B}{0_n}{D}}
 \sum_{\lambda_2,\mu_2 \in (\Z/p\Z)^{(n,1)}} 
\\
 & \qquad
 \times \psi|_{k,\mathcal{M}}
  \left(\smat{p^2 {^t D}^{-1}}{B}{0_n}{D} \times \smat{U}{ }{ }{p^2 U^{-1}}, [((0,\lambda_2),(0,\mu_2)),0_2]\right)
\\
 &=
 \sum_{\smat{p^2 {^t D}^{-1}}{B}{0_n}{D}}
 \sum_{\lambda_2,\mu_2 \in (\Z/p\Z)^{(n,1)}} 
 \sum_{N,R}
  A(N,R)
\\
 &\qquad
 \times e(N\tau + {R ^t z})|_{k,\mathcal{M}}
  \left(\smat{p^2 {^t D}^{-1}}{B}{0_n}{D} \times \smat{U}{ }{ }{p^2 U^{-1}}, [((0,\lambda_2),(0,\mu_2)),0_2]\right),
 \end{aligned}
\end{eqnarray*}
where
$\smat{p^2 {^t D}^{-1}}{B}{0_n}{D}$ runs over a set of all representatives
of
\[
 \Gamma_n \backslash \Gamma_n \mbox{diag}(1_{\alpha},p 1_{n-\alpha},p^2 1_{\alpha}, p 1_{n-\alpha})\Gamma_n,
\]
and where the slash operator $|_{k,\mathcal{M}}$ is defined in \S\ref{ss:factors_automorphy}.

We put $\lambda = (0,\lambda_2)$, $\mu = (0, \mu_2)$ $\in \Z^{(n,2)}$, then
we obtain
\begin{eqnarray*}
 \begin{aligned}
 &
 e(N\tau + R{^t z})
 |_{k,\mathcal{M}} \left(\smat{p^2 {^t D}^{-1}}{B}{0_n}{D}  \times \smat{U}{ }{ }{p^2 U^{-1}},
                   [(\lambda,\mu),0_2] \right) 
\\
 &=
 p^{-k}\det(D)^{-k} e(\hat{N} \tau + \hat{R} {^t z} + NBD^{-1} + R U {^t \mu} D^{-1}),
 \end{aligned}
\end{eqnarray*}
where
\begin{eqnarray*}
 \hat{N} &=& p^2 D^{-1}N{^t D}^{-1} + D^{-1}R U {^t \lambda} +  \frac{1}{p^2}\lambda U \mathcal{M} U {^t \lambda}
\end{eqnarray*}
and
\begin{eqnarray*}
 \hat{R} &=& D^{-1} R U + \frac{2}{p^2} \lambda U \mathcal{M} U .
\end{eqnarray*}
Thus
\begin{eqnarray*}
 N
 &=& 
 \frac{1}{p^2} D\left(\left(\hat{N} - \frac{1}{4}\hat{R}_2 {^t \hat{R}_2}\right) 
   + \frac{1}{4}(\hat{R}_2 - 2 \lambda_2) {^t(\hat{R}_2 - 2 \lambda_2)}\right){^t D}
\end{eqnarray*}
and
\begin{eqnarray*}
 R &=& D\left(\hat{R} - \frac{2}{p^2} \lambda U \mathcal{M} U \right) U^{-1},
\end{eqnarray*}
where $\hat{R}_2 = \hat{R}\left(\begin{smallmatrix}0 \\ 1 \end{smallmatrix}\right)$.
Hence, for any $\hat{N} \in \mbox{Sym}_n^*$ and for any $\hat{R} \in \Z^{(n,2)}$,
we have
\begin{eqnarray*}
 \begin{aligned}
&
 A_2(\hat{N},\hat{R})
\\
 &=
 p^{-k} \!\!\!\!\!\!\! \sum_{\smat{p^2 {^t D}^{-1}}{B}{0_n}{D}}
   \!\!\!\!\!\!\!\!
   \det(D)^{-k}
   \!\!\!\!\!\!\!\!
   \sum_{\lambda_2 \in (\Z/p\Z)^{(n,1)}}\sum_{\mu_2 \in (\Z/p\Z)^{(n,1)}}
   \!\!\!\!\!\!\!
    A_1(N,R)\, e(NBD^{-1} + RU{^t(0, \mu_2)}D^{-1})
\\
 &=
 p^{-k+n} \sum_{\smat{p^2 {^t D}^{-1}}{B}{0_n}{D}} 
   \det(D)^{-k} \sum_{\lambda_2 \in (\Z/p\Z)^{(n,1)}}
   A_1(N,R)\, e(NBD^{-1}),
 \end{aligned}
\end{eqnarray*}
where $N$ and $R$ are the same symbols as above, which are determined by $\hat{N}$, $\hat{R}$ and $\lambda_2$,
and where
$\smat{p^2 {^t D}^{-1}}{B}{0_n}{D}$ runs over a complete set of representatives 
of
\[
  \Gamma_n \backslash \Gamma_n
  \mbox{diag}(1_{\alpha},p 1_{n-\alpha},p^2 1_{\alpha}, p 1_{n-\alpha}) \Gamma_n.
\]
On the RHS of the above first identity
the matrix $D^{-1} R U$ belongs to $\Z^{(n,2)}$, since $\hat{R} \in \Z^{(n,2)}$.
We remark that $A_1(N,R) = 0$ unless $N \in \mbox{Sym}_n^*$ and $R \in \Z^{(n,2)}$.

Due to the definition of $\iota_{\mathcal{M}}$,
for $N \in \mbox{Sym}_n^*$ and $R \in \Z^{(n,2)}$ 
we have the identity 
\begin{eqnarray*}
 A_1(N,R) 
 &=&
 C_1(4N-R\left(\begin{smallmatrix}0\\1\end{smallmatrix}\right)^t(R\left(\begin{smallmatrix}0\\1\end{smallmatrix}\right)),
    4 R\left(\begin{smallmatrix}1\\0\end{smallmatrix}\right) - 2 r R\left(\begin{smallmatrix}0\\1\end{smallmatrix}\right)).
\end{eqnarray*}
Here
\begin{eqnarray*}
 4N - R\left(\begin{smallmatrix}0\\1\end{smallmatrix}\right) {^t(R\left(\begin{smallmatrix}0\\1\end{smallmatrix}\right)) }
 &=&
 \frac{1}{p^2} D \left( 4\hat{N} - \hat{R}_2 {^t \hat{R}_2} \right) {^t D}
\end{eqnarray*}
and
\begin{eqnarray*}
 4 R\left(\begin{smallmatrix}1\\0\end{smallmatrix}\right) - 2 r R\left(\begin{smallmatrix}0\\1\end{smallmatrix}\right)
 &=&
 \frac{1}{p^2}D(4 \hat{R}\left(\begin{smallmatrix}1\\0\end{smallmatrix}\right) - 2 r p \hat{R}_2).
\end{eqnarray*}
Hence we have
\begin{eqnarray}\label{id:A2}
&&
 A_2(\hat{N},\hat{R})
\\
 &=&
 p^{-k+n} \!\!\! \sum_{\smat{p^2 {^t D}^{-1}}{B}{0_n}{D}} \!\!\! \det(D)^{-k} 
 C_1\!\left(\frac{1}{p^2} D \left( 4\hat{N} - \hat{R}_2 {^t \hat{R}_2} \right) {^t D},
     \frac{1}{p^2}D(4 \hat{R}\left(\begin{smallmatrix}1\\0\end{smallmatrix}\right) - 2 r p \hat{R}_2)\right)
\notag \\
 && \times
  e\!\left(\frac{1}{p^2} \left(\hat{N} - \frac{1}{4}\hat{R}_2 {^t \hat{R}_2}\right){^t D}B\right)
 \sum_{\lambda_2}
  e\!\left(\frac{1}{4 p^2} 
    (\hat{R}_2 - 2 \lambda_2) {^t(\hat{R}_2 - 2 \lambda_2)}{^t D}B\right), \notag
\end{eqnarray}
where $\lambda_2$ runs over a complete set of representatives of $(\Z/p\Z)^{(n,1)}$ such that
\begin{eqnarray*}
 D \left(\hat{R} - \frac{2}{p^2}(0, \lambda_2) U \mathcal{M} U\right) U^{-1} \in \Z^{(n,2)}.
\end{eqnarray*}

Let $\mathfrak{S}_{\alpha}$ be a complete set of representative of
$\Gamma_n \backslash
  \Gamma_n \left(\begin{smallmatrix}
            1_{\alpha}&&&\\
            &p 1_{n-\alpha}&&\\
            &&p^2 1_{\alpha}&\\
            &&&p 1_{n-\alpha}
           \end{smallmatrix}\right) \Gamma_n$.
Now we quote a complete set of representatives $\mathfrak{S}_{\alpha}$
from \cite{Zhu:euler}.
We put
\[
 \delta_{i,j} := \mbox{diag}(1_i,p 1_{j-i},p^2 1_{n-j})
\]
for $0 \leq i \leq j \leq n$.
We set
\begin{eqnarray*}
 \mathfrak{S}_{\alpha}
 &:=&
 \left.
 \left\{
  \begin{pmatrix}p^2{\delta_{i,j}}^{-1}&b_0\\0_n & \delta_{i,j} \end{pmatrix}
  \begin{pmatrix}{^t u}^{-1} & 0_n \\ 0_n & u \end{pmatrix}
  \, \right| \,
  i,j , b_0, u
 \right\},
\end{eqnarray*}
where $i$ and $j$ run over all non-negative integers such that $j-i-n+\alpha \geq 0$,
and where $u$ runs over a complete set of representatives of
$(\delta_{i,j}^{-1}\mbox{GL}_n(\Z) \delta_{i,j} \cap \mbox{GL}_n(\Z)) \backslash \mbox{GL}_n(\Z)$,
and $b_0$ runs over all matrices in the set
\begin{eqnarray*}
 \mathfrak{T}
 &:=&
 \left. \left\{
  \begin{pmatrix}
   0_i & 0 & 0 \\
   0 & a_1 & p b_1 \\
   0 & ^t b_1 & b_2 
  \end{pmatrix}
  \right|
  \begin{matrix}
  b_1 \in (\Z/p\Z)^{(j-i,n-j)},
  b_2 = {^t b_2} \in (\Z/p^2 \Z)^{(n-j,n-j)},\\
  a_1 = {^t a_1} \in (\Z/p\Z)^{(j-i,j-i)},
  \mbox{rank}_p(a_1) = j-i-n+\alpha
  \end{matrix}
    \right\} .
\end{eqnarray*}

For a matrix
 $g = \begin{pmatrix}p^2 {^t D}^{-1}&B \\ 0_n& D \end{pmatrix}
  =
  \begin{pmatrix}p^2{\delta_{i,j}}^{-1}&b_0\\0_n & \delta_{i,j} \end{pmatrix}
  \begin{pmatrix}{^t u}^{-1} & 0_n \\ 0_n & u \end{pmatrix}$
  $\in$ $\mathfrak{S}_{\alpha}$ 
with a matrix
$b_0 =  \begin{pmatrix}
   0_i & 0 & 0 \\
   0 & a_1 & p b_1 \\
   0 & ^t b_1 & b_2 
  \end{pmatrix} \in \mathfrak{T}$,
we define $\varepsilon(g) := \left(\frac{-4}{p}\right)^{rank_p(a_1)/2}\left(\frac{\det a_1'}{p}\right)$,
where $a_1' \in \mbox{GL}_{j-i-n+\alpha}(\Z/p\Z)$ is a matrix
such that
$a_1 \equiv \smat{a_1'}{0}{0}{0_{n-\alpha}}[v] \mod p$ with some $v \in \mbox{GL}_{j-i}(\Z)$.
Under the assumption
\begin{eqnarray*}
\frac{1}{p^2}D(4 \hat{R}\left(\begin{smallmatrix}1\\0\end{smallmatrix}\right) - 2 r p \hat{R}_2) \in \Z^{(n,2)},
\end{eqnarray*}
the condition $D (\hat{R} - p^{-2}(0, \lambda_2) U \mathcal{M} U) U^{-1} \in \Z^{(n,2)}$
is equivalent to the condition
\begin{eqnarray*}
 u(\hat{R}_2 - 2 \lambda_2) \in \smat{p 1_i}{0}{0}{1_{n-i}} \Z^{(n,1)}.
\end{eqnarray*}
Hence the last summation in~(\ref{id:A2}) is
\begin{eqnarray*}
 && \!\!\!\!\!\!\!\!\!\!
 \sum_{\lambda_2}
  e\!\left(\frac{1}{4 p^2} 
    (\hat{R}_2 - 2 \lambda_2) {^t(\hat{R}_2 - 2 \lambda_2)}{^t D}B\right)
\\
 &=&
 p^{n-j} \sum_{\lambda' \in (\Z/p\Z)^{(j-i,1)}} e\!\left(\frac{1}{p} {^t \lambda'} a_1 \lambda'\right)
\\
 &=&
 p^{n - i - rank_p(a_1)} \left(\left(\frac{-4}{p}\right) p\right)^{rank_p(a_1)/2} \left(\frac{\det a_1'}{p}\right)
\\
 &=&
 p^{n - i -\frac{rank_p(a_1)}{2}} \varepsilon(g)
\\
 &=&
 p^{n + (n-i-j-\alpha)/2} \varepsilon(g).
\end{eqnarray*}
Thus (\ref{id:A2}) is
\begin{eqnarray*}
 A_2(\hat{N},\hat{R})
 &=&
 p^{-k+2n} \sum_{g } p^{-k(2n-i-j) + (n-i-j-\alpha)/2} \varepsilon(g)\,
 e\! \left(p^{-2} \left(4 \hat{N} - \hat{R}_2 {^t \hat{R}_2}\right){^t D}B\right)
\\
 && \times
  C_1\!\left(p^{-2} D ( 4\hat{N} - \hat{R}_2 {^t \hat{R}_2} ) {^t D},\
     p^{-2}D(4 \hat{R}\left(\begin{smallmatrix}1\\0\end{smallmatrix}\right) - 2 r p \hat{R}_2)\right),
\end{eqnarray*}
where 
 $g = \begin{pmatrix}p^2 {^t D}^{-1}&B \\ 0_n& D \end{pmatrix} =
  \begin{pmatrix}p^2{\delta_{i,j}}^{-1}&b_0\\0_n & \delta_{i,j} \end{pmatrix}
  \begin{pmatrix}{^t u}^{-1} & 0_n \\ 0_n & u \end{pmatrix} $
  runs over all elements in the set $\mathfrak{S}_{\alpha}$.

Now we shall express $C_2(M,S)$ as a linear combination of Fourier coefficients $C_1(M,S)$ of $\iota_{M}(\psi)$.
For $Y = (\mbox{diag}(1_{\alpha},p 1_{n-\alpha},p^2 1_{\alpha}, p 1_{n-\alpha}),p^{\alpha/2}) \in \widetilde{\mbox{GSp}_n^+(\Z)}$
a complete set of representatives of
$\Gamma_0^{(n)}(4)^* \backslash \Gamma_0^{(n)}(4)^* Y \Gamma_0^{(n)}(4)^*$ is given by elements
\begin{eqnarray*}
 \widetilde{g}
 &=&
 (g,\varepsilon(g) p^{(n-i-j)/2}) \in \widetilde{\mbox{GSp}_n^+(\Z)},
\end{eqnarray*}
where $g$ runs over all elements in the set $\mathfrak{S}_{\alpha}$,
and $\varepsilon(g)$ is defined as above (cf. \cite[Lemma 3.2]{Zhu:euler}).
Hence
\begin{eqnarray*}
 \begin{aligned}
 &
 (\iota_{\mathcal{M}}(\psi)|\tilde{V}_{\alpha,n-\alpha}(p^2))(\tau,z)
\\
 &=
 p^{n(2k-1)/2 - n(n+1)}\sum_{M,S} \sum_{\widetilde{g}} p^{(-k+1/2)(n-i-j)} \varepsilon(g)\,
  C_1(M,S) \\
 & \qquad \qquad \qquad \qquad \qquad \qquad
  \times e(M(p^2 {^t D}^{-1}\tau + B) D^{-1} + p^2 S {^t z} D^{-1})
\\
 &=
 p^{n(2k-1)/2 - n(n+1)}\sum_{\hat{M},\hat{S}} \sum_{g \in \mathfrak{S}_{\alpha}} p^{(-k+1/2)(n-i-j)} \varepsilon(g)\,
  C_1(p^{-2} D \hat{M}{^t D}, p^{-2} D \hat{S})
\\
 & \qquad \qquad \qquad \qquad \qquad \qquad
  \times e(\hat{M} \tau + \hat{S} {^t z} + p^{-2} \hat{M} {^t D} B).
 \end{aligned}
\end{eqnarray*}
Thus 
\begin{eqnarray*}
 C_2(\hat{M},\hat{S})
 &=&
  \sum_{g} p^{-n(n+1)+(k-1/2)(i+j)} \varepsilon(g)\,
  C_1(p^{-2} D \hat{M}{^t D}, p^{-2} D \hat{S})\, e(p^{-2} \hat{M} {^t D} B).
\end{eqnarray*}

Now we put
$\hat{M} = 4 \hat{N} - \hat{R}_2 {^t \hat{R}_2 }$ and
$\hat{S} = 4 \hat{R} \left(\begin{smallmatrix}1\\0 \end{smallmatrix}\right) - 2 r p \hat{R}_2$,
then
\begin{eqnarray*}
 C_2(4 \hat{N} - \hat{R}_2 {^t \hat{R}_2 },4 \hat{R} \left(\begin{smallmatrix}1\\0 \end{smallmatrix}\right) - 2 r p \hat{R}_2)
 &=&
 p^{2nk+k - n^2 -\frac{7}{2}n  + \frac12 \alpha} A_2(\hat{N},\hat{R}).
\end{eqnarray*}
The proposition follows from this identity.
\end{proof}

\subsection{Index-shift maps at $p=2$}\label{ss:hecke_p2}
For $p=2$ we define the map
\begin{eqnarray*}
  \tilde{V}_{\alpha,n-\alpha}(4) \ :\
     J_{k-\frac12,m}^{(n)*} \rightarrow 
     \mbox{Hol}(\H_n \times \C^{(n,1)} \rightarrow \C)
\end{eqnarray*}
through an analogue of the identity~(\ref{id:iota_hecke}), it means that we define
\begin{eqnarray*}
 \phi|\tilde{V}_{\alpha,n-\alpha}(4)
 &:=&
 2^{k(2n+1) - n (n+\frac72) + \frac12 \alpha }\
 \iota_{\mathcal{M}[\smat{2}{ }{ }{1}]}(\psi| V_{\alpha,n-\alpha}(4))
\end{eqnarray*}
for any $\phi \in J_{k-\frac12,m}^{(n)*}$, and where $\psi \in J_{k,\mathcal{M}}^{(n)*}$
is the Jacobi form which satisfies $\iota_{\mathcal{M}}(\psi) = \phi$.
Here the map $V_{\alpha,n-\alpha}(4)$ is defined in \S\ref{ss:hecke_operators}
and the map  $\iota_{\mathcal{M}}$ is defined in \S\ref{ss:iota}.

\section{Action of index-shift maps on Jacobi-Eisenstein series}\label{s:action_of_ism}
In this section we fix a positive definite half-integral symmetric matrix
$\mathcal{M}  \in \mbox{Sym}_2^+$
and we assume that the right-lower part of $\mathcal{M}$ is $1$, 
it means $\mathcal{M} = \left( \begin{matrix} * & * \\ * & 1 \end{matrix} \right)$.

The purpose of this section is to show that the form
$E_{k,\mathcal{M}}^{(n)}|V_{\alpha,n-\alpha}(p^2)$ is a linear combination
of three forms 
$E_{k,\mathcal{M}\left[ \left(\begin{smallmatrix} p&0\\0&1 \end{smallmatrix} \right)\right]}^{(n)}$,
$E_{k,\mathcal{M}}^{(n)}|U_{\left(\begin{smallmatrix} p & 0 \\ 0 & 1\end{smallmatrix} \right)}$
and
$E_{k,\mathcal{M}\left[ X^{-1} \left( \begin{smallmatrix} p & 0 \\ 0 & 1\end{smallmatrix} \right)^{-1} \right]}^{(n)}|U_{\left( \begin{smallmatrix} p & 0 \\ 0 & 1\end{smallmatrix}\right) X \left( \begin{smallmatrix} p & 0 \\ 0 & 1\end{smallmatrix}\right)}$,
where $E_{k,\mathcal{M}}^{(n)}$ is the Jacobi-Eisenstein series of
index $\mathcal{M}$ (cf. \S\ref{s:fourier_matrix}),
and where $V_{\alpha,n-\alpha}(p^2)$
and $U_{\left( \begin{smallmatrix} p & 0 \\ 0 & 1\end{smallmatrix}\right)}$
are index-shift maps (cf. \S\ref{ss:hecke_operators} and \S\ref{ss:compati}).
Here $X = \smat{1}{0}{x}{1}$ is a certain matrix.

First
we will calculate a certain function $K_{i,j}^{\beta}$ (cf. Lemma~\ref{lemma:kij}) which 
appear in an expression of
$E_{k,\mathcal{M}}^{(n)}|V_{\alpha,n-\alpha}(p^2)$,
and after that, we will express $E_{k,\mathcal{M}}^{(n)}|V_{\alpha,n-\alpha}(p^2)$ as a summation of
certain functions $\tilde{K}_{i,j}^{\beta}$ (cf. Proposition~\ref{prop:tilde_kij}).

The calculation in this section is an analogue to the one given in \cite{Ya2}
for the case of index $\mathcal{M} = 1$.
However, we need to modify his calculation
for Jacobi-Eisenstein series $E_{k,1}^{(n)}$ of index $1$ to our case
for $E_{k,\mathcal{M}}^{(n)}$
with $\mathcal{M} = \smat{*}{*}{ *}{1} \in \mbox{Sym}_2^+$.
This calculation is not obvious, since we need to calculate the action of
the matrices of type $[((0,u_2),(0,v_2)),0_2]$.

\subsection{The function $K_{i,j}^{\beta}$}\label{ss:kij}
The purpose of this subsection is to introduce a function $K_{i,j}^{\beta}$ and to express
$E_{k,\mathcal{M}}^{(n)}|V_{\alpha,n-\alpha}(p^2)$ as a summation over $K_{i,j}^{\beta}$.
Moreover, we shall calculate $K_{i,j}^{\beta}$ explicitly (cf. Lemma~\ref{lemma:kij}).

We put $\delta_{i,j} := \mbox{diag}(1_i,p 1_{j-i},p^2 1_{n-j})$.
For $x = \mbox{diag}(0_i,x',0_{n-i-j})$ with $x' = {^t x'} \in \Z^{(j-i,j-i)}$,
we set $\delta_{i,j}(x) := \left(\begin{matrix}p^2 \delta_{i,j}^{-1} & x \\ 0 & \delta_{i,j} \end{matrix}\right)$
and 
$\Gamma(\delta_{i,j}(x)) := \Gamma_n \cap \delta_{i,j}(x)^{-1}\Gamma_{\infty}^{(n)} \delta_{i,j}(x)$.

For $x = \mbox{diag}(0_i,x',0_{n-i-j})$ and for $y = \mbox{diag}(0_i,y',0_{n-i-j})$
with $x'={^t x'}, y' = {^t y'} \in \Z^{(j-i,j-i)}$,
following~\cite{Ya2} we say that $x$ and $y$ are \textit{equivalent},
if there exists a matrix
$u \in \mbox{GL}_n(\Z) \cap \delta_{i,j}\mbox{GL}_n(\Z)\delta_{i,j}^{-1}$
which has a form
$u = 
 \left(\begin{smallmatrix}u_1& * & * \\
                          * &u_2& * \\ 
                          * & * &u_3 \end{smallmatrix}\right)$
satisfying
$x' \equiv u_2\, y'\, {^t u_2} \mod p$,
where $u_2 \in \Z^{(j-i,j-i)}$, $u_1 \in \Z^{(i,i)}$ and $u_3\in \Z^{(n-j,n-j)}$.

We denote by $[x]$ the equivalence class of $x$.
We quote the following lemma from \cite{Ya2}.
\begin{lemma}\label{lemma:hecke_decom}
The double coset $\Gamma_n \mbox{diag}(1_{\alpha},p 1_{n-\alpha}, p^2 1_{\alpha}, p 1_{n-\alpha})\Gamma_n$ is 
written as a disjoint union 
 \begin{eqnarray*}
  \Gamma_n \left( \begin{smallmatrix}1_{\alpha} &&&\\
                                &p 1_{n-\alpha}&&\\
                                &&p^2 1_{\alpha}&\\
                                &&&p 1_{n-\alpha}  \end{smallmatrix}\right)\Gamma_n
  &=&
  \bigcup_{\begin{smallmatrix} i,j \\ 0\leq i \leq j \leq n \end{smallmatrix}}
  \bigcup_{[x]}
   \Gamma_{\infty}^{(n)} \delta_{i,j}(x) \Gamma_n ,
 \end{eqnarray*}
where $[x]$ runs over all equivalence classes which satisfy $\mbox{rank}_p(x) = j - i - n + \alpha \geq 0$.
\end{lemma}
\begin{proof}
The reader is referred to~\cite[Corollary~2.2]{Ya2}.
\end{proof}

We put $U := \begin{pmatrix} p^2&0\\0&p \end{pmatrix}$.
By the definition of index-shift map $V_{\alpha,n-\alpha}(p^2)$
and of the Jacobi-Eisenstein series $E_{k,\mathcal{M}}^{(n)}$, we have
\begin{eqnarray*}
 \begin{aligned}
 &
 E_{k,\mathcal{M}}^{(n)}|V_{\alpha,n-\alpha}(p^2) \\
 &=
 \sum_{u,v \in \Z^{(n,1)}}
 \sum_{M' \in \Gamma_n \backslash \Gamma_n diag(1_{\alpha},p1_{n-\alpha},p^2 1_{\alpha}, p 1_{n-\alpha})\Gamma_n}
 \sum_{M \in \Gamma_{\infty}^{(n)}\backslash \Gamma_n}
 \sum_{\lambda \in \Z^{(n,2)}}
\\
 & \qquad
 \times 1 |_{k,\mathcal{M}} ([(\lambda,0),0_2],MM'\times \smat{U}{0}{0}{p^2 U^{-1}})
    |_{k,\mathcal{M}[\smat{p}{0}{0}{1}]}[((0, u),(0, v)),0_2]
\\
 &=
 \sum_{u,v \in \Z^{(n,1)}}
 \sum_{M \in \Gamma_{\infty}^{(n)} \backslash \Gamma_n
              diag(1_{\alpha},p1_{n-\alpha},p^2 1_{\alpha}, p 1_{n-\alpha})
             \Gamma_n}
 \sum_{\lambda \in \Z^{(n,2)}}
\\
 & \qquad
  \times 1 |_{k,\mathcal{M}} ([(\lambda,0),0_2],M\times \smat{U}{0}{0}{p^2 U^{-1}})
    |_{k,\mathcal{M}[\smat{p}{0}{0}{1}]}[((0, u),(0, v)),0_2] .
 \end{aligned}
\end{eqnarray*}
Hence, due to Lemma~\ref{lemma:hecke_decom}, we have
\begin{eqnarray*}
 \begin{aligned}
 &
 E_{k,\mathcal{M}}^{(n)}|V_{\alpha,n-\alpha}(p^2)
\\
 &=
 \sum_{u,v \in \Z^{(n,1)}}
 \sum_{\begin{smallmatrix} i,j \\ 0\leq i \leq j \leq n \end{smallmatrix}}
 \sum_{\begin{smallmatrix} [x] \\ rank_p(x) = j-i-n+\alpha\end{smallmatrix}}
 \sum_{M \in \Gamma_{\infty}^{(n)}\backslash \delta_{i,j}(x)\Gamma_n}
 \sum_{\lambda \in \Z^{(n,2)}}
\\
 & \qquad
 \times 1 |_{k,\mathcal{M}} ([(\lambda,0),0_2],M\times\smat{U}{0}{0}{p^2 U^{-1}})
    |_{k,\mathcal{M}[\smat{p}{0}{0}{1}]}[((0, u),(0, v)),0_2]
\\
 &=
 \sum_{u,v \in \Z^{(n,1)}}
 \sum_{\begin{smallmatrix} i,j \\ 0\leq i \leq j \leq n \end{smallmatrix}}
 \sum_{\begin{smallmatrix} [x] \\ rank_p(x) = j-i-n+\alpha \end{smallmatrix}}
 \sum_{M \in \Gamma(\delta_{i,j}(x))\backslash \Gamma_n}
 \sum_{\lambda \in \Z^{(n,2)}}
\\
 & \qquad
 \times 1 |_{k,\mathcal{M}} ([(\lambda,0),0_2],\delta_{i,j}(x)M\times\smat{U}{0}{0}{p^2 U^{-1}})
    |_{k,\mathcal{M}[\smat{p}{0}{0}{1}]}[((0, u),(0, v)),0_2] .
 \end{aligned}
\end{eqnarray*}
For $\beta \leq j-i $ we define a function
\begin{eqnarray*}
 \begin{aligned}
 &
 K_{i,j}^{\beta}(\tau,z)
\\
 &:=
 K_{i,j,\mathcal{M},p}^{\beta}(\tau,z)
\\
 &= \!\!
 \sum_{\begin{smallmatrix}[x] \\ rank_p(x) = \beta \end{smallmatrix}}
 \sum_{M \in \Gamma(\delta_{i,j}(x))\backslash \Gamma_n}\sum_{\lambda \in \Z^{(n,2)}}
 \left\{
  1|_{k,\mathcal{M}}(
   [(\lambda,0),0_2],
   \delta_{i,j}(x) M 
    \times\smat{U}{0}{0}{p^2 U^{-1}}
   )
 \right\}
  (\tau,z) .
 \end{aligned}
\end{eqnarray*}
Then we obtain
\begin{eqnarray*}
 E_{k,\mathcal{M}}^{(n)}|V_{\alpha,n-\alpha}(p^2)
 &=&
 \sum_{\begin{smallmatrix} i,j \\  0\leq i \leq j \leq n \end{smallmatrix}}
 \sum_{u,v \in \Z^{(n,1)}}
  K_{i,j}^{\alpha-i-n+j}
  |_{k,\mathcal{M}[\smat{p}{0}{0}{1}]}[((0, u),(0, v)),0_2] .
\end{eqnarray*}

We define
\begin{eqnarray*}
 L_{i,j}
 &:=&
 L_{i,j,\mathcal{M},p}
 =
 \left.
 \left\{ \left(\begin{smallmatrix}\lambda_1\\ \lambda_2\\ \lambda_3 \end{smallmatrix}\right)
   \ \right| \
 \begin{matrix}
  \lambda_1 \in (p\Z)^{(i,2)}\, , \, \lambda_2 \in \Z^{(j-i,2)}\, , \, \lambda_3 \in (p^{-1}\Z)^{(n-j,2)}
 \\
  2 \lambda_2 \mathcal{M} {^t \lambda_3} \in \Z^{(j-i,n-j)} \, , \, 
  \lambda_3 \mathcal{M} {^t \lambda_3} \in \Z^{(n-j,n-j)}
 \end{matrix}
 \right\}.
\end{eqnarray*}
Moreover, we define a subgroup $\Gamma(\delta_{i,j})$ of $\Gamma_{\infty}^{(n)}$ by
\begin{eqnarray*}
 \Gamma(\delta_{i,j})
 &:=&
 \left.
 \left\{
  \begin{pmatrix}
   A & B \\ 0_n & {^t A}^{-1}
  \end{pmatrix}
  \in
  \Gamma_{\infty}^{(n)}
 \, \right| \,
   A \in \delta_{i,j}\mbox{GL}_n(\Z) \delta_{i,j}^{-1}
 \right\} .
\end{eqnarray*}

\begin{lemma} \label{lemma:kij}
Let $K_{i,j}^{\beta}$ be as above.
We obtain
\begin{eqnarray*}
 K_{i,j}^{\beta}(\tau,z)
 &=&
 p^{-k(2n-i-j+1)+(n-j)(n-i+1)}
 \sum_{M \in \Gamma(\delta_{i,j})\backslash \Gamma_n}
\\
 && \times \sum_{\lambda \in L_{i,j}}
 1|_{k,\mathcal{M}}([(\lambda,0),0_2],M)(\tau,z\smat{p}{0}{0}{1}) \!\!\!
 \sum_{\begin{smallmatrix} x = ^t x \in (\Z/p\Z)^{(n,n)} \\
                           x = diag(0_i,x',0_{n-j}) \\
                           rank_p(x') = \beta \end{smallmatrix}} \!\!\!
 e\left(\frac{1}{p}\mathcal{M}{^t\lambda} x \lambda\right),
\end{eqnarray*}
where
$x$ runs over
a complete set of representatives of $(\Z/p\Z)^{(n,n)}$ such that $x = {^t x}$,
$\mbox{rank}_p(x) = \beta$ and
$x = \mbox{diag}(0_i,x',0_{n-j})$ with some $x' \in (\Z/p\Z)^{(j-i,j-i)}$.
\end{lemma}
\begin{proof}
We proceed as in~\cite[Proposition~3.2]{Ya2}.
The inside of the last summation of the definition of $K_{i,j}^{\beta}(\tau,z)$ is
\begin{eqnarray*}
 \begin{aligned}
 &
 \left(1|_{k,\mathcal{M}}(
   [(\lambda,0),0_2],
   \delta_{i,j}(x) M 
    \times\smat{U}{0}{0}{p^2 U^{-1}})\right)(\tau,z)
\\
 &=
 \det(p^2U^{-1})^{-k}\det(\delta_{i,j})^{-k}
\\
 & \qquad
 \times
 \left(e(\mathcal{M}(^t\lambda(p^2 \delta_{ij}^{-1}\tau+ x)\delta_{ij}^{-1}\lambda 
 + 2 ^t \lambda \delta_{ij}^{-1} z \begin{pmatrix} p^2 & \\ & p \end{pmatrix}))
 |_{k,\mathcal{M}[\left(\begin{smallmatrix} p&0\\0&1\end{smallmatrix} \right)]} M\right)(\tau,z)
\\
 &=
 p^{-k(2n-i-j+1)} 
\\
 & \qquad \times
 \left(\left((1|_{k,\mathcal{M}}([(p\delta_{i,j}^{-1}\lambda,0),0_2],\begin{pmatrix}1&p^{-1}x\\0&1\end{pmatrix}))
  (\tau,z\smat{p}{0}{0}{1})\right)|_{k,\mathcal{M}[\left(\begin{smallmatrix} p&0\\0&1\end{smallmatrix} \right)]} M\right)(\tau,z)
\\
 &=
 p^{-k(2n-i-j+1)}
 \left(1|_{k,\mathcal{M}}([(p\delta_{i,j}^{-1}\lambda,0),0_2],\begin{pmatrix}1&p^{-1}x\\0&1\end{pmatrix}M)\right)
  (\tau,z\smat{p}{0}{0}{1}) .
 \end{aligned}
\end{eqnarray*}
Here we used the identity $\delta_{i,j} x = \delta_{i,j} \mbox{diag}(0_i,x',0_{n-j}) = p x$.
Thus
\begin{eqnarray*}
 K_{i,j}^{\beta}(\tau,z)
 &=&
 p^{-k(2n-i-j+1)}
 \sum_{\begin{smallmatrix}[x] \\ rank_p(x) = \beta \end{smallmatrix}}
 \sum_{M \in \Gamma(\delta_{i,j}(x))\backslash \Gamma_n}
\\
 &&
 \times \sum_{\lambda \in \Z^n}
  1|_{k,\mathcal{M}}\left([(p\delta_{i,j}^{-1}\lambda,0),0_2],\begin{pmatrix}1&p^{-1}x\\0&1\end{pmatrix}M\right)
   (\tau,z\smat{p}{0}{0}{1}) .
\end{eqnarray*}
We put
\begin{eqnarray*}
 \mathcal{U}
 &:=&
 \left.
 \left\{\begin{pmatrix}1_n& s \\ 0_n & 1_n \end{pmatrix} \, \right| \, 
  s = {^t s} \in \Z^{(n,n)}
 \right\} .
\end{eqnarray*}
Then the set
\begin{eqnarray*}
 \mathcal{V}
 &:=&
 \left.
 \left\{
  \begin{pmatrix}1_n& s \\ 0_n & 1_n \end{pmatrix}
 \, \right| \,
  s = \left(\begin{smallmatrix}0&0&0\\0&0&s_2\\0&^t s_2&s_3\end{smallmatrix}\right),
  s_2 \in (\Z/p\Z)^{(j-i,n-j)},
  s_3 = {^t s_3} \in (\Z/p \Z)^{(n-j,n-j)}
 \right\}
\end{eqnarray*}
is a complete set of representatives of
$\Gamma(\delta_{i,j}(x))\backslash \Gamma(\delta_{i,j}(x))\mathcal{U}$.
Therefore
\begin{eqnarray*}
 \begin{aligned}
 &
 K_{i,j}^{\beta}(\tau,z)
\\
 &=
 p^{-k(2n-i-j+1)}
 \sum_{\begin{smallmatrix}[x] \\ rank_p(x) = \beta \end{smallmatrix}}
 \sum_{M \in (\Gamma(\delta_{i,j}(x))\mathcal{U})\backslash \Gamma_n}
 \sum_{\lambda \in \Z^{(n,2)}}
 \sum_{\smat{1_n}{s}{0}{1_n} \in \mathcal{V}}
\\
 & \qquad
 \times\quad 1|_{k,\mathcal{M}}([(p\delta_{i,j}^{-1}\lambda,0),0_2],\begin{pmatrix}1_n&p^{-1}x\\0&1_n\end{pmatrix}
                            \begin{pmatrix}1_n&s\\0&1_n\end{pmatrix}M)(\tau,z\smat{p}{0}{0}{1})
 \end{aligned}
\end{eqnarray*}
Hence
\begin{eqnarray*}
 K_{i,j}^{\beta}(\tau,z)
 &=&
 p^{-k(2n-i-j+1)}
 \sum_{\begin{smallmatrix}[x] \\ rank_p(x) = \beta \end{smallmatrix}}
 \sum_{M \in (\Gamma(\delta_{i,j}(x))\mathcal{U})\backslash \Gamma_n}
 \sum_{\lambda \in \Z^{(n,2)}} 
\\
 &&
 \times\quad
 1|_{k,\mathcal{M}}
  ([(p\delta_{i,j}^{-1}\lambda,0),0_2],\begin{pmatrix}1_n&p^{-1}x\\0&1_n\end{pmatrix}M)(\tau,z\smat{p}{0}{0}{1})
\\
 &&
 \times
 \sum_{\smat{1_n}{s}{0}{1_n} \in \mathcal{V}}
 e\left(p^2 \mathcal{M} {^t\lambda} \delta_{i,j}^{-1} s \delta_{i,j}^{-1} \lambda \right) .
\end{eqnarray*}
The last summation of the RHS of the above identity is
\begin{eqnarray*}
 \begin{aligned}
&
 \sum_{\smat{1_n}{s}{0}{1_n} \in \mathcal{V}} e\left(p^2 \mathcal{M} {^t\lambda} \delta_{i,j}^{-1} s \delta_{i,j}^{-1} \lambda \right)
\\
&=
 \begin{cases}
   p^{(n-j)(n-i+1)} & \mbox{ if } \lambda_3 \mathcal{M} {^t\lambda_3} \equiv 0 \!\! \mod p^2  
                                  \mbox{ and } 
                                  2 \lambda_3 \mathcal{M} {^t\lambda_2} \equiv 0 \!\! \mod p ,
                                  \\
   0 & \mbox{ otherwise, }
 \end{cases}
 \end{aligned}
\end{eqnarray*}
where 
$\lambda = \left(\begin{smallmatrix} \lambda_1 \\ \lambda_2 \\ \lambda_3 \end{smallmatrix} \right) \in \Z^{(n,2)}$
with $\lambda_1 \in \Z^{(i,2)}$, $\lambda_2 \in \Z^{(j-i,2)}$ and $\lambda_3 \in \Z^{(n-j,2)}$.

Thus
\begin{eqnarray*}
 K_{i,j}^{\beta}(\tau,z)
 &=&
 p^{-k(2n-i-j+1)+(n-j)(n-i+1)}\sum_{\begin{smallmatrix}[x] \\ rank_p(x) = \beta \end{smallmatrix}}
 \sum_{M \in (\Gamma(\delta_{i,j}(x))\mathcal{U})\backslash \Gamma_n}
\\
 && \times \sum_{\lambda \in L_{i,j}}
  1|_{k,\mathcal{M}}
  ([(\lambda,0),0_2],\smat{1_n}{p^{-1}x}{0}{1_n}M)(\tau,z\smat{p}{0}{0}{1}) .
\end{eqnarray*}
Now
$\Gamma(\delta_{i,j}(x))\mathcal{U}$ is a subgroup of $\Gamma(\delta_{i,j})$.
For any $\smat{A}{B}{0_n}{{^t A}^{-1}} \in \Gamma(\delta_{i,j})$ we have
\begin{eqnarray*}
 \begin{aligned}
 &
 1|_{k,\mathcal{M}}([(\lambda,0),0_2],\smat{1_n}{p^{-1}x}{0_n}{1_n} \smat{A}{B}{0_n}{^t A^{-1}} M)
 \\
 &=
 1|_{k,\mathcal{M}}([(\lambda,0),0_2],\smat{A}{B}{0_n}{{^t A}^{-1}}\smat{1_n}{p^{-1}A^{-1}x{{^t A}^{-1}}}{0_n}{1_n}M)
 \\
 &=
 1|_{k,\mathcal{M}}([({^t A}\lambda,{^t B}\lambda),0_2],\smat{1_n}{p^{-1}A^{-1}x{{^t A}^{-1}}}{0_n}{1_n}M)
 \\
 &=
 1|_{k,\mathcal{M}}([({^t A}\lambda,0),0_2],\smat{1_n}{p^{-1}A^{-1}x{{^t A}^{-1}}}{0_n}{1_n}M),
 \end{aligned}
\end{eqnarray*}
and ${^tA}L_{i,j} = L_{i,j}$.
Moreover, when $\smat{A}{B}{0_n}{{^t A}^{-1}}$ runs over all elements in a complete set of representatives
of $\Gamma(\delta_{i,j}(x))\mathcal{U} \backslash \Gamma(\delta_{i,j})$, then  $A^{-1}x {^t A}^{-1}$ runs over all elements in the equivalence class $[x]$ (cf.~\cite[proof of Proposition~3.2]{Ya2}).
Therefore we have
\begin{eqnarray*}
 \begin{aligned}
 &
 K_{i,j}^{\beta}(\tau,z) \\
 &=
 p^{-k(2n-i-j+1)+(n-j)(n-i+1)}
 \sum_{\begin{smallmatrix} x = ^t x \in (\Z/p\Z)^{(n,n)} \\ 
                           x = diag(0_i,x',0_{n-j}) \\
                           rank_p(x') = \beta 
       \end{smallmatrix}}
 \sum_{M \in \Gamma(\delta_{i,j})\backslash \Gamma_n}
\\
 & \qquad \times \sum_{\lambda \in L_{i,j}}
 1|_{k,\mathcal{M}}([(\lambda,0),0_2],\smat{1_n}{p^{-1}x}{0}{1_n}M)(\tau,z\smat{p}{0}{0}{1})
\\
 &=
 p^{-k(2n-i-j+1)+(n-j)(n-i+1)}
 \sum_{M \in \Gamma(\delta_{i,j})\backslash \Gamma_n}
\\
 & \qquad \times \sum_{\lambda \in L_{i,j}}
 1|_{k,\mathcal{M}}([(\lambda,0),0_2],M)(\tau,z\smat{p}{0}{0}{1})
 \sum_{\begin{smallmatrix} x = ^t x \in (\Z/p\Z)^{(n,n)} \\
                           x = diag(0_i,x',0_{n-j}) \\
                           rank_p(x') = \beta \end{smallmatrix}}
 e\left(\frac{1}{p}\mathcal{M}{^t\lambda} x \lambda\right).
 \end{aligned}
\end{eqnarray*}
\end{proof}

\subsection{The function $\tilde{K}_{i,j}^{\beta}$}\label{ss:tilde_kij}
The purpose of this subsection is to introduce a function $\tilde{K}_{i,j}^{\beta}$ and to express
$E_{k,\mathcal{M}}^{(n)}|V_{\alpha,n-\alpha}(p^2)$ as a summation of $\tilde{K}_{i,j}^{\beta}$.
Moreover, we shall show that $\tilde{K}_{i,j}^{\beta}$ is a summation of certain exponential functions
with generalized Gauss sums (cf. Proposition~\ref{prop:tilde_kij}).

We define
\begin{eqnarray*}
 L_{i,j}^*
 &:=&
 L_{i,j,\mathcal{M},p}^* \\
 &=&
 \left\{ \left. \left(\begin{smallmatrix} \lambda_1 \\ \lambda_2 \\ \lambda_3 \end{smallmatrix}\right)
  \in L_{i,j}
  \, \right| \,   2 \lambda_3 \mathcal{M}
   \left(\begin{smallmatrix}0\\1 \end{smallmatrix}\right) \in \Z^{(n-j,1)} \right\}
 \\ 
 &=&
 \left\{
  \left( \begin{matrix} \lambda_1\\ \lambda_2\\ \lambda_3 \end{matrix} \right)
  \in (p^{-1}\Z)^{(n,2)}
 \, \left| \,
  \begin{matrix}
  \lambda_1\smat{p}{0}{0}{1}^{-1} \in \Z^{(i,2)}, \
  \lambda_2 \in \Z^{(j-i,2)} \\
  \lambda_3 \in (p^{-1}\Z)^{(n-j,2)},\
  2 \lambda_2 \mathcal{M} {^t \lambda_3} \in \Z^{(j-i,n-j)} \\
  \lambda_3 \mathcal{M} {^t\lambda_3} \in \Z^{(n-j,n-j)}, \
  2 \lambda_3 \mathcal{M} \left(\begin{smallmatrix}0\\1 \end{smallmatrix}\right) \in \Z^{(n-j,1)}
  \end{matrix}
 \right\}
 \right.
\end{eqnarray*}
and define a generalized Gauss sum
\begin{eqnarray*}
 G_{\mathcal{M}}^{j-i,l}(\lambda_2) 
 &:=& 
 \sum_{\begin{smallmatrix} x' = ^t x' \in (\Z/p\Z)^{(j-i,j-i)} \\
                           rank_p(x') = j-i-l
       \end{smallmatrix}}
  e\left(\frac{1}{p}\mathcal{M} {^t \lambda_2 } x'  \lambda_2\right)
\end{eqnarray*}
for $\lambda_2 \in \Z^{(j-i,2)}$.
We define
\begin{eqnarray*}
 \tilde{K}_{i,j}^{\beta}(\tau,z)
 &:=&
 \tilde{K}_{i,j,\mathcal{M},p}^{\beta}(\tau,z)
 = \!\!\!\!\!\!\!\!
 \sum_{u,v \in (\Z/p\Z)^{(n,1)}}\left(K_{i,j}^{\beta}
  |_{k,\mathcal{M}[\smat{p}{0}{0}{1}]}
    [\left((0,u),(0,v)\right), 0_2 ] \right)(\tau,z).
\end{eqnarray*}

\begin{prop}\label{prop:tilde_kij}
Let the notation be as above. Then we obtain
\begin{eqnarray*} 
 \left(E_{k,\mathcal{M}}^{(n)}|V_{\alpha,n-\alpha}(p^2)\right)(\tau,z)
 &=&
 \sum_{\begin{smallmatrix} i,j \\ 0\leq i \leq j \leq n \\ j-i \geq n-\alpha \end{smallmatrix}}
  \tilde{K}_{i,j}^{\alpha-i-n+j}(\tau,z) ,
\end{eqnarray*}
where
\begin{eqnarray*}
 \tilde{K}_{i,j}^{\alpha-i-n+j}(\tau,z)
 &=&
 p^{-k(2n-i-j+1)+(n-j)(n-i+1)+2n-j }
\\
 &&
 \times
 \sum_{M \in \Gamma(\delta_{i,j})\backslash \Gamma_n}
 \sum_{ \lambda = \left(\begin{smallmatrix}\lambda_1\\ \lambda_2\\ \lambda_3 \end{smallmatrix}\right)\in L_{i,j}^*}
  \left\{1|_{k,\mathcal{M}}([(\lambda,0),0_2],M)\right\}(\tau,z\smat{p}{0}{0}{1})
\\
 &&
 \times
  \sum_{u_2 \in (\Z/p\Z)^{(j-i,1)}}G_{\mathcal{M}}^{j-i,n-\alpha}(\lambda_2 + (0,u_2)) .
\end{eqnarray*}
\end{prop}
\begin{proof}
From the definition of $\tilde{K}_{i,j}^{\beta}$ and Lemma~\ref{lemma:kij} we obtain
\begin{eqnarray}\label{eq:k_tilde_1}
 && \\
\notag
\begin{aligned}
 &
 \tilde{K}_{i,j}^{\alpha-i-n+j}(\tau,z)
\\
 &=
 p^{-k(2n-i-j+1)+(n-j)(n-i+1)}
 \sum_{M \in \Gamma(\delta_{i,j})\backslash \Gamma_n}
 \sum_{\lambda = \left(\begin{smallmatrix}\lambda_1\\ \lambda_2\\ \lambda_3 \end{smallmatrix}\right)\in L_{i,j}}
 G_{\mathcal{M}}^{j-i,n-\alpha}(\lambda_2)
\\
 &
 \times \!\!\!
 \sum_{u,v \in (\Z/p\Z)^{(n,1)}}
 \left(1|_{k,\mathcal{M}}([(\lambda,0),0_2],M)(\tau,z\smat{p}{0}{0}{1})\right)
  |_{k,\mathcal{M}[\smat{p}{0}{0}{1}]}[\left((0,u),(0,v)\right), 0_2],
\end{aligned}
\end{eqnarray}
where
$\lambda_1 \in \Z^{(i,2)}$, $\lambda_2 \in \Z^{(j-i,2)}$ and $\lambda_3 \in \Z^{(n-j,2)}$
satisfy $\left(\begin{smallmatrix}\lambda_1\\ \lambda_2\\ \lambda_3 \end{smallmatrix}\right) \in L_{i,j}$,
and where the $n \times 2$ matrix $\lambda = \left(\begin{smallmatrix}\lambda_1\\ \lambda_2\\ \lambda_3 \end{smallmatrix}\right)$
runs over the set $L_{i,j}$.

By a straightforward calculation we have
\begin{eqnarray*}
 \left(1|_{k,\mathcal{M}}([(\lambda,0),0_2],M)\right)\! (\tau,z\smat{p}{0}{0}{1})
 &=&
 \left(1|_{k,\mathcal{M}[\smat{p}{0}{0}{1}]}([(\lambda\smat{p}{0}{0}{1}^{-1},0),0_2],M)\right)\!(\tau,z).
\end{eqnarray*}
Thus the last summation of (\ref{eq:k_tilde_1}) is
\begin{eqnarray*}
 && \!\!\!\!\!
 \sum_{u,v \in (\Z/p\Z)^{(n,1)}}
 \!\!\!
 \left\{
  1|_{k,\mathcal{M}}([(\lambda,0),0_2],M)(\tau,z\smat{p}{0}{0}{1})
   |_{k,\mathcal{M}[\smat{p}{0}{0}{1}]}[\left((0,u),(0,v)\right), 0_2]
 \! \right\}\! (\tau,z)
\\
 &=&
 \sum_{u,v \in (\Z/p\Z)^{(n,1)}}
 \!\!\!
 \left\{
  1|_{k,\mathcal{M}[\smat{p}{0}{0}{1}]}([(\lambda\smat{p}{0}{0}{1}^{-1},0),0_2],M)
   |_{k,\mathcal{M}[\smat{p}{0}{0}{1}]}[\left((0,u),(0,v)\right), 0_2]
 \! \right\}\! (\tau,z)
\\
 &=&
 \sum_{u',v' \in (\Z/p\Z)^{(n,1)}}
 \left\{1|_{k,\mathcal{M}[\smat{p}{0}{0}{1}]}([(\lambda\smat{p}{0}{0}{1}^{-1}+(0,u'),(0,v')),0_2],M)\right\}
 (\tau,z)
\\
 &=&
 \sum_{u',v' \in (\Z/p\Z)^{(n,1)}}
 \left\{1|_{k,\mathcal{M}}([(\lambda+(0,u'),(0,v')),0_2],M)\right\}
 (\tau,z\smat{p}{0}{0}{1})
\\
 &=&
 \sum_{u' \in (\Z/p\Z)^{(n,1)}} \!\!\!
 \left\{1|_{k,\mathcal{M}}([(\lambda+(0,u'),0),0_2],M)\right\}
 (\tau,z\smat{p}{0}{0}{1}) \!\!\!\!\!
 \sum_{v' \in (\Z/p\Z)^{(n,1)}} \!\!\!\!\!\! e(2 \mathcal{M} {^t\lambda}(0, v')),
\end{eqnarray*}
where, in the second identity, we used
\begin{eqnarray*}
 (M,[((0,u),(0,v)),0_2])
 &=&
 ([((0,u'),(0,v')),0_2],M)
\end{eqnarray*}
with
$\left(\begin{smallmatrix} u' \\ v'\end{smallmatrix}\right)
  =
 \smat{D}{-C}{-B}{A}
 \left(\begin{smallmatrix} u \\ v \end{smallmatrix} \right)$
for $M = \smat{A}{B}{C}{D} \in \Gamma_n$.
For $\lambda = \left(\begin{smallmatrix}\lambda_1\\ \lambda_2 \\ \lambda_3 \end{smallmatrix}\right) \in L_{i,j}$
we now have
\begin{eqnarray*}
 \sum_{v' \in (\Z/p\Z)^{(n,1)}} e(2 \mathcal{M} {^t\lambda} (0,v'))
 &=&
 \begin{cases}
  p^n & \mbox{if } 2 \lambda_3 \mathcal{M}\left(\begin{smallmatrix}0 \\ 1 \end{smallmatrix}\right) \in \Z^{(n-j,1)},
 \\
  0   & \mbox{otherwise}.
 \end{cases}
\end{eqnarray*}
Therefore
\begin{eqnarray*}
 \begin{aligned}
 &
 \tilde{K}_{i,j}^{\alpha-i-n+j}(\tau,z)
\\
 &=
 p^{-k(2n-i-j+1)+(n-j)(n-i+1)+n}
 \sum_{M \in \Gamma(\delta_{i,j})\backslash \Gamma_n}
 \sum_{\begin{smallmatrix}
        \lambda = \left(\begin{smallmatrix}\lambda_1\\ \lambda_2\\ \lambda_3 \end{smallmatrix}\right)\in L_{i,j} \\
        2 \lambda_3 \mathcal{M} \left(\begin{smallmatrix} 0\\1 \end{smallmatrix}\right) \in \Z^{(n-j,1)}
       \end{smallmatrix}}
 G_{\mathcal{M}}^{j-i,n-\alpha}(\lambda_2)
\\
 & \qquad
 \times
 \sum_{u \in (\Z/p\Z)^{(n,1)}}
  \left\{1|_{k,\mathcal{M}}([(\lambda+(0,u),0),0_2],M)\right\}(\tau,z\smat{p}{0}{0}{1}).
 \end{aligned}
\end{eqnarray*}
Thus
\begin{eqnarray*}
 \begin{aligned}
 &
 \tilde{K}_{i,j}^{\alpha-i-n+j}(\tau,z)
\\
 &=
 p^{-k(2n-i-j+1)+(n-j)(n-i+1)+n }
 \sum_{M \in \Gamma(\delta_{i,j})\backslash \Gamma_n}
 \sum_{ \lambda = \left(\begin{smallmatrix}\lambda_1\\ \lambda_2\\ \lambda_3 \end{smallmatrix}\right)\in L_{i,j}^*}
\\
 & \qquad
 \times
  \left\{1|_{k,\mathcal{M}}([(\lambda,0),0_2],M)\right\}(\tau,z\smat{p}{0}{0}{1})
\\
 & \qquad
 \times
  p^{n-j}\sum_{u_2 \in (\Z/p\Z)^{(j-i,1)}}G_{\mathcal{M}}^{j-i,n-\alpha}(\lambda_2 + (0,u_2)),
 \end{aligned}
\end{eqnarray*}
where $L_{i,j}^*$ is defined as before.
\end{proof}

We put
\begin{eqnarray*}
  g_p(n,\alpha)
  &:=&
  \prod_{j=1}^{\alpha} \left\{ (p^{n-j+1}-1) (p^j -1)^{-1}\right\}
\end{eqnarray*}
It is not difficult to see $g_p(n,n-\alpha) = g_p(n,\alpha)$.

\begin{lemma}\label{lem:sum_gs}
For any $\lambda = (\lambda_1, \lambda_2) \in \Z^{(n,2)}$ and for any prime $p$, we have
\begin{eqnarray*}
 && \!\!\!\!\!\!\!\!\!\!
 \sum_{u_2 \in (\Z/p\Z)^{(n,1)}}G_{\mathcal{M}}^{n,\alpha}(\lambda + (0,u_2)) \\
 &=&
 \begin{cases}
     p^{\frac14 (n-\alpha-1)^2 + \frac12(n-\alpha-1)+\alpha+n} \left( \frac{-m}{p} \right)
     g_p(n-1,\alpha)
     \displaystyle{\prod_{\begin{smallmatrix} j = 1 \\ j \, : \, odd \end{smallmatrix}}^{n-\alpha-2}
     (p^j - 1) }
   &
     \begin{matrix}  \mbox{if } n-\alpha \equiv 1 \mod 2 \\
                                                \mbox{and } \lambda_1 \not \equiv 0 \mod p,
                       \end{matrix}
   \\
     0
   &
     \begin{matrix}  \mbox{if } n-\alpha \equiv 1 \mod 2 \\
                                                \mbox{and } \lambda_1 \equiv 0 \mod p,
                       \end{matrix}
   \\
     p^{\frac14(n-\alpha)^2 + \frac12 (n-\alpha) + \alpha} g_p(n,\alpha)
     \displaystyle{\prod_{\begin{smallmatrix} j = 1 \\ j \, : \, odd \end{smallmatrix}}^{n-\alpha-1}
     (p^j - 1) }
   &
     \mbox{if } n-\alpha \equiv 0 \mod 2.
 \end{cases}
\end{eqnarray*}
Here $m = \det(2\mathcal{M})$ and we regard the product
$\displaystyle{\prod_{\begin{smallmatrix} j = 1 \\ j \, : \, odd \end{smallmatrix}}^{c}(p^j - 1) }$
as $1$, if $c$ is less than $1$.
\end{lemma}
\begin{proof}
This calculation is similar to the calculation of
$\displaystyle{\sum_{\begin{smallmatrix} x = ^t x \in (\Z/p\Z)^{n} \\ rank_p x = n-\alpha \end{smallmatrix}} e\left(\frac{1}{p} m {^t \lambda_1} x  \lambda_1 \right)}$
for $\lambda_1 \in \Z^{(n,1)}$ and for $m \in \Z$ which is
in \cite[Lemma 3.1]{Ya2}.

If $p$ is an odd prime and if $\lambda_1 \not \equiv 0 \mod p$, then
\begin{eqnarray*}
 \sum_{u_2 \in (\Z/p\Z)^{(n,1)}} G_{\mathcal{M}}^{n,\alpha}(\lambda + (0,u_2))
 &=& \!\!\!
 \sum_{u_2 \in (\Z/p\Z)^{(n,1)}} 
  \sum_{\begin{smallmatrix} x' = ^t x' \in (\Z/p\Z)^{(n,n)} \\
                           rank_p(x') = n-\alpha
       \end{smallmatrix}}
  e\left(\frac{1}{p}\mathcal{M}\, {^t (\lambda_1,u_2)}  x'  (\lambda_1,u_2) \right) \\
\end{eqnarray*}
By diagonalizing the matrices $x'$ we have
\begin{eqnarray*}
 \sum_{u_2 \in (\Z/p\Z)^{(n,1)}} G_{\mathcal{M}}^{n,\alpha}(\lambda + (0,u_2))
  &=&
 \sum_{i=0,1}
 p^{n-1} \left| \mbox{GL}_{n-1}(\Z/p\Z) \right|  \left| O(x_i) \right|^{-1}\\
 &&
 \times
 \sum_{u_2 \in (\Z/p\Z)^{(n,1)}} 
 \sum_{\begin{smallmatrix} \eta \in (\Z/p\Z)^{(n,1)}\\ \eta \not \equiv 0 \mod p \end{smallmatrix}}
   e\left(\frac{1}{p}\mathcal{M}\, {^t (\eta,u_2)}  x_i  (\eta,u_2)\right),
\end{eqnarray*}
where $x_i = \left( \begin{smallmatrix} y_i & 0 \\ 0 & 0 \end{smallmatrix} \right) \in \Z^{(n,n)}$,
$y_0 = 1_{n-\alpha}$,
$y_1 = \left( \begin{smallmatrix} 1_{n-\alpha-1} & 0 \\ 0 & \gamma \end{smallmatrix} \right)
\in \Z^{(n-\alpha,n-\alpha)}$
and $\gamma$ is an integer such that $\left( \frac{\gamma}{p} \right) = -1$.
Here $O(x_i)$  is the orthogonal group of $x_i$:
\begin{eqnarray*}
  O(x_i) &:=& \left\{ g  \in \mbox{GL}_n(\Z/p\Z) \, | \, g x_i {^t g} = x_i \right\}.
\end{eqnarray*}
If we diagonalize the matrix $\mathcal{M}$ as
$\mathcal{M} \equiv {^t X} \left( \begin{smallmatrix} m & 0 \\ 0 & 1 \end{smallmatrix} \right) X \mod p$
with $X = \left( \begin{smallmatrix} 1 & 0 \\ x & 1 \end{smallmatrix} \right)$,
then
\begin{eqnarray*}
  \sum_{u_2 \in (\Z/p\Z)^{(n,1)}} G_{\mathcal{M}}^{n,\alpha}(\lambda + (0,u_2))
  &=&
   \sum_{i=0,1}
 p^{n-1} \left| \mbox{GL}_{n-1}(\Z/p\Z) \right|  \left| O(x_i) \right|^{-1}\\
 &&
 \times
 \sum_{u_2 \in (\Z/p\Z)^{(n,1)}} 
 \sum_{\begin{smallmatrix} \eta \in (\Z/p\Z)^{(n,1)}\\ \eta \not \equiv 0 \mod p \end{smallmatrix}}
   e\left(\frac{1}{p} \left(m \eta {^t \eta} + u_2 {^t u_2}\right) x_i\right).
\end{eqnarray*}
The rest of the calculation is an analogue to \cite[Lemma 3.1]{Ya2}.
For the case of $p=2$ or $\lambda_1 \equiv 0 \mod p$,
the calculation is similar.
If $p=2$, we need to calculate the case that
$\mathcal{M} = {^t X}\left( \begin{smallmatrix} m' & \frac12 \\ \frac12 & 1 \end{smallmatrix} \right) X$,
but it is not difficult.
We leave the detail to the reader.
\end{proof}

We set
\begin{eqnarray*}
 S_{\mathcal{M}}^{n,\alpha}(0) := \sum_{u_2 \in (\Z/p\Z)^{(n,1)}}G_{\mathcal{M}}^{n,\alpha}((0,u_2)) \\
\end{eqnarray*}
and
\begin{eqnarray*}
 S_{\mathcal{M}}^{n,\alpha}(1) :=
 \sum_{u_2 \in (\Z/p\Z)^{(n,1)}}G_{\mathcal{M}}^{n,\alpha}\left(\left( \left(\begin{smallmatrix} 1 \\ 0 \\ \vdots \\ 0 \end{smallmatrix} \right),u_2\right)\right).
\end{eqnarray*}
Due to Lemma~\ref{lem:sum_gs},
we have that $\displaystyle{\sum_{u_2 \in (\Z/p\Z)^{(n,1)}} G_{\mathcal{M}}^{n,\alpha}(\lambda + (0,u_2))}$
equals  $S_{\mathcal{M}}^{n,\alpha}(0)$ or $S_{\mathcal{M}}^{n,\alpha}(1)$,
according as $\lambda \in \Z^{(n,2)} \begin{pmatrix} p & 0 \\ 0 & 1 \end{pmatrix}$
or $\lambda \not \in \Z^{(n,2)} \begin{pmatrix} p & 0 \\ 0 & 1 \end{pmatrix}$.

\begin{prop}\label{prop:E_linear}
The form $ E_{k,\mathcal{M}}^{(n)}|V_{\alpha,n-\alpha}(p^2) $
is a linear combination of three forms
$E_{k,\mathcal{M}\left[ \left(\begin{smallmatrix} p&0\\0&1 \end{smallmatrix} \right)\right]}^{(n)}$,
$E_{k,\mathcal{M}}^{(n)}|U_{\left(\begin{smallmatrix} p & 0 \\ 0 & 1\end{smallmatrix} \right)}$
and
$E_{k,\mathcal{M}\left[ X^{-1} \left( \begin{smallmatrix} p & 0 \\ 0 & 1\end{smallmatrix} \right)^{-1} \right]}^{(n)}|U_{\left( \begin{smallmatrix} p & 0 \\ 0 & 1\end{smallmatrix}\right) X \left( \begin{smallmatrix} p & 0 \\ 0 & 1\end{smallmatrix}\right)}$.
Here the index-shift map $U_L$ is defined in \S\ref{ss:compati},
and
$X = \left( \begin{smallmatrix} 1 & 0 \\ x & 1  \end{smallmatrix}\right)$ is a matrix in $\Z^{(2,2)}$ such that
$\mathcal{M} = {^t X} \left( \begin{smallmatrix} m+1 & 1 \\ 1 & 1 \end{smallmatrix} \right) X$ 
if $p=2$ and $\frac{\det(2\mathcal{M})}{4} \equiv 3 \mod 4$,
or
$\mathcal{M} \equiv {^t X} \left( \begin{smallmatrix} m & 0 \\ 0 & 1 \end{smallmatrix} \right) X
  \mod p$ otherwise, 
and where $m = \det(2\mathcal{M})$.
\end{prop}
\begin{proof}
By virtue of Proposition~\ref{prop:tilde_kij} we only need to show
that the form $\tilde{K}_{i,j}^{\alpha-i-n+j}(\tau,z)$ is a linear combination of the above three forms.

Because of the conditions $\lambda_3 \mathcal{M} {^t \lambda_3} \in \Z^{(n-j,n-j)}$ and
$2 \lambda_3 \mathcal{M} \left(\begin{smallmatrix} 0\\1 \end{smallmatrix} \right) \in \Z^{(n-j,1)}$ in the definition of $L_{i,j}^*$, we obtain
\begin{eqnarray} \label{eq:Lijs1}
 L_{i,j}^* &=&
   \left\{ \left(\begin{smallmatrix} \lambda_1 \\ \lambda_2 \\ \lambda_3 \end{smallmatrix}\right)  \in \left( \frac{1}{p} \Z \right)^{(n,2)} \, \left| \,
   \begin{matrix}
                    \lambda_1 \in \Z^{(i,2)} \left( \begin{smallmatrix} p & 0 \\ 0 & 1 \end{smallmatrix} \right),
                    \lambda_2 \in \Z^{(j-i,2)}, \\
       \lambda_3 {^t X} \in \Z^{(n-j,2)} \left(\begin{smallmatrix} p^{-1} & 0 \\ 0 & 1 \end{smallmatrix} \right)
       \end{matrix} \right\} \right.
\end{eqnarray}
for the case $p | f$, and
\begin{eqnarray}\label{eq:Lijs2}
 L_{i,j}^* &=&
   \left\{ \left(\begin{smallmatrix} \lambda_1 \\ \lambda_2 \\ \lambda_3 \end{smallmatrix}\right)  \in \Z^{(n,2)} \, | \,
                    \lambda_1 \in \Z^{(i,2)} \left( \begin{smallmatrix} p & 0 \\ 0 & 1 \end{smallmatrix} \right),
                    \lambda_2 \in \Z^{(j-i,2)}, \lambda_3 \in \Z^{(n-j,2)} \right\}
\end{eqnarray}
for the case $p {\not |} f$.
Here $f$ is a natural number such that $D_0 f^2 = -\det(2\mathcal{M}) $ and $D_0$ is
a  fundamental discriminant, and where the matrix $X $ is stated in this proposition.

We now assume $p | f$.
If $p$ is an odd prime, then the matrix $X = \smat{1}{0}{x}{1} \in \Z^{(2,2)}$ satisfies
$\mathcal{M} \equiv {^t X} \smat{m}{0}{0}{1} X \mod p$
and $p^2 | m$.
If $p=2$, then
the matrix $X = \smat{1}{0}{x}{1} \in \Z^{(2,2)}$ satisfies
$\mathcal{M} = {^t X} \smat{m}{0}{0}{1} X$ with $4|m$,
or $\mathcal{M} = {^t X} \smat{m'}{1}{1}{1} X$ with $4|m'$.
We remark that
$\mathcal{M}\left[ X^{-1} \smat{p}{0}{0}{1}^{-1} \right]$ is a half-integral symmetric matrix.

We put
\begin{eqnarray*}
  L_0
  &:=&
  \left\{ \left( \begin{smallmatrix} \lambda_1 \\ \lambda_2 \\ \lambda_3 \end{smallmatrix} \right)
           \in L_{i,j}^* \, | \, \lambda_2 \in \Z^{(j-i,2)} \left( \begin{smallmatrix} p & 0 \\ 0 & 1 \end{smallmatrix} \right) \right\}, \\
  L_1
  &:=&
  \left\{ \left( \begin{smallmatrix} \lambda_1 \\ \lambda_2 \\ \lambda_3 \end{smallmatrix} \right)
           \in L_{i,j}^* \, | \, \lambda_2 \not \in \Z^{(j-i,2)} \left( \begin{smallmatrix} p & 0 \\ 0 & 1 \end{smallmatrix} \right) \right\}.
\end{eqnarray*}
and set
\begin{eqnarray*}
  L'_{i,j} &:=&
  \left. \left\{
    \left( \begin{smallmatrix} \lambda_1 \\ \lambda_2 \\ \lambda_3 \end{smallmatrix}\right)
    \, \right| \,
    \lambda_1 \in \Z^{(i,2)} \smat{p^2}{0}{0}{1}, \,
    \lambda_2 \in \Z^{(j-i,2)} \smat{p}{0}{0}{1}, \,
    \lambda_3 \in \Z^{(n-j,2)}
  \right\}.
\end{eqnarray*}

By using the identity
\begin{eqnarray*}
 \begin{aligned}
  &
  \left\{1|_{k,\mathcal{M}}([(\lambda,0),0_2],M)\right\}(\tau,z\smat{p}{0}{0}{1}) \\
  &=
  \left\{1 \left|_{k,\mathcal{M}\left[X^{-1} \right]}([(\lambda {^t X},0),0_2],M)\right\}(\tau,z\smat{p}{0}{0}{1} {^t X}) \right. \\
  &=
  \left\{1 \left|_{k,\mathcal{M}\left[X^{-1} \smat{p}{0}{0}{1}^{-1}\right]}([(\lambda {^t X} \smat{p}{0}{0}{1} ,0),0_2],M)\right\}(\tau,z\smat{p}{0}{0}{1} {^t X} \smat{p}{0}{0}{1} ) \right. ,
 \end{aligned}
\end{eqnarray*}
we have
\begin{eqnarray*}
 && \!\!\!\!\!\!\!\!\!\!
 \tilde{K}_{i,j}^{\alpha-i-n+j}(\tau,z) \\
 &=&
 p^{-k(2n-i-j+1)+(n-j)(n-i+1)+2n-j }
\\
 &&
 \times
 \sum_{M \in \Gamma(\delta_{i,j})\backslash \Gamma_n}
  \left\{  S_{\mathcal{M}}^{j-i,n-\alpha}(0)
 \sum_{\lambda \in L_0}
  \left\{1|_{k,\mathcal{M}}([(\lambda,0),0_2],M)\right\}(\tau,z\smat{p}{0}{0}{1})
  \right.
\\
 &&
  + 
  \left.
  S_{\mathcal{M}}^{j-i,n-\alpha}(1)
 \sum_{\lambda \in L_1}
  \left\{1|_{k,\mathcal{M}}([(\lambda,0),0_2],M)\right\}(\tau,z\smat{p}{0}{0}{1})
  \right\}\\
 &=&
 p^{-k(2n-i-j+1)+(n-j)(n-i+1)+2n-j }
 \sum_{M \in \Gamma(\delta_{i,j})\backslash \Gamma_n}
  \Bigg\{  \left(S_{\mathcal{M}}^{j-i,n-\alpha}(0) - S_{\mathcal{M}}^{j-i,n-\alpha}(1) \right) \\
  &&
  \times
 \sum_{\lambda \in L_0}
  \left\{1|_{k,\mathcal{M}\left[ X^{-1} \smat{p}{0}{0}{1}^{-1} \right]}([(\lambda {^t X} \smat{p}{0}{0}{1} ,0),0_2],M)\right\}(\tau,z\smat{p}{0}{0}{1} {^t X} \smat{p}{0}{0}{1})
\\
 &&
  + 
  S_{\mathcal{M}}^{j-i,n-\alpha}(1)\\
 &&
 \times
 \sum_{\lambda \in L_{i,j}^*}
  \left\{1|_{k,\mathcal{M}\left[ X^{-1} \smat{p}{0}{0}{1}^{-1} \right]}([(\lambda {^t X} \smat{p}{0}{0}{1},0),0_2],M)\right\}(\tau,z\smat{p}{0}{0}{1} {^t X} \smat{p}{0}{0}{1})
  \Bigg\} \\
 &=&
 p^{-k(2n-i-j+1)+(n-j)(n-i+1)+2n-j }
 \sum_{M \in \Gamma(\delta_{i,j})\backslash \Gamma_n}
  \Bigg\{  \left(S_{\mathcal{M}}^{j-i,n-\alpha}(0) - S_{\mathcal{M}}^{j-i,n-\alpha}(1) \right) \\
  &&
  \times
 \sum_{\lambda \in L'_{j,j}}
  \left\{1|_{k,\mathcal{M}\left[ X^{-1} \smat{p}{0}{0}{1}^{-1} \right]}([(\lambda,0),0_2],M)\right\}(\tau,z\smat{p}{0}{0}{1} {^t X} \smat{p}{0}{0}{1})
\\
 &&
  + 
  S_{\mathcal{M}}^{j-i,n-\alpha}(1)
 \sum_{\lambda \in L'_{i,j}}
  \left\{1|_{k,\mathcal{M}\left[ X^{-1} \smat{p}{0}{0}{1}^{-1} \right]}([(\lambda ,0),0_2],M)\right\}(\tau,z\smat{p}{0}{0}{1} {^t X} \smat{p}{0}{0}{1})
  \Bigg\} .
\end{eqnarray*}

We now calculate the sum
\begin{eqnarray*}
  \sum_{M \in \Gamma(\delta_{i,j})\backslash \Gamma_n}
  \sum_{\lambda \in L'_{j,j}}
  \left\{1|_{k,\mathcal{M}\left[ X^{-1} \smat{p}{0}{0}{1}^{-1} \right]}([(\lambda,0),0_2],M)\right\}(\tau,z\smat{p}{0}{0}{1} {^t X} \smat{p}{0}{0}{1}).
\end{eqnarray*}

We set
\begin{eqnarray*}
 H_{i,j} &:= \delta_{i,j} \, \mbox{GL}_n(\Z)\, \delta_{i,j}^{-1} \cap \mbox{GL}_n(\Z).
\end{eqnarray*}

If $\{A_l\}_l$ is a complete set of representatives of $H_{i,j}\backslash \mbox{GL}_n(\Z)$,
then one can say that
the set $\left\{ \begin{pmatrix} A_l & 0 \\ 0 & {^t A_l}^{-1} \end{pmatrix} \right\}_l$
is a complete set of representatives of $\Gamma(\delta_{i,j})\backslash \Gamma_{\infty}^{(n)}$.
Thus
\begin{eqnarray*}
 \begin{aligned}
  &
 \sum_{M \in \Gamma(\delta_{i,j})\backslash \Gamma_n}
 \sum_{\lambda \in L'_{j,j}}
  \left\{1|_{k,\mathcal{M}}([(\lambda,0),0_2],M)\right\}(\tau,z)\\
  &=
 \sum_{M \in \Gamma_{\infty}^{(n)}\backslash \Gamma_n}
 \sum_{A \in H_{i,j}\backslash GL_n(\Z)}
 \sum_{\lambda \in L'_{j,j}}
   \left\{1|_{k,\mathcal{M}}([(\lambda,0),0_2],\smat{A}{0}{0}{{^t A^{-1}}} M)\right\}(\tau,z)\\
 &=
 \sum_{M \in \Gamma_{\infty}^{(n)}\backslash \Gamma_n}
 \sum_{A \in H_{i,j}\backslash GL_n(\Z)}
 \sum_{\lambda \in L'_{j,j}}
   \left\{1|_{k,\mathcal{M}}([({^t A}\lambda,0),0_2],M)\right\}(\tau,z).
 \end{aligned}
\end{eqnarray*}

If $B(\lambda)$ is a function on $\lambda \in \Z^{(n,2)}$.
Then
\begin{eqnarray*}
 \begin{aligned}
  &
  \sum_{A \in H_{i,j}\backslash GL_n(\Z)}
  \sum_{\lambda \in L'_{j,j}}
  B({^t A} \lambda) \\
  &=
  \left[H_{j,j} : H_{i,j} \right]
  \sum_{A \in H_{j,j}\backslash GL_n(\Z)}
  \sum_{\lambda \in L'_{j,j}}
  B({^t A} \lambda) \\  
  &=
  \left[H_{j,j} : H_{i,j} \right]
  \left(
  a_0 \sum_{\lambda \in \Z^{(n,2)}} B(\lambda )
  +
  a_1 \sum_{\lambda \in \Z^{(n,2)}} B(\lambda \smat{p}{0}{0}{1} )
  +
  a_2 \sum_{\lambda \in \Z^{(n,2)}} B(\lambda \smat{p^2}{0}{0}{1} )
  \right)
 \end{aligned}
\end{eqnarray*}
with certain numbers $a_0$, $a_1$ and $a_2$
under the assumption that the summations converges absolutely.
The values $a_0$, $a_1$ and $a_2$ are independent of the choice of the function $B$.
For the exact values of $a_0$, of $a_1$ and of $a_2$ the reader is referred to \cite[Lemma 3.7]{MaassRe}.

Hence we have
\begin{eqnarray*}
  && \!\!\!\!\!\!\!\!\!\!
  \sum_{M \in \Gamma(\delta_{i,j})\backslash \Gamma_n}
  \sum_{\lambda \in L'_{j,j}}
  \left\{1|_{k,\mathcal{M}\left[ X^{-1} \smat{p}{0}{0}{1}^{-1} \right]}([(\lambda,0),0_2],M)\right\}(\tau,z\smat{p}{0}{0}{1} {^t X} \smat{p}{0}{0}{1})\\
  &=&
  \left[H_{j,j} : H_{i,j} \right] 
  \sum_{M \in \Gamma_{\infty}^{(n)}\backslash \Gamma_n}
  \\
  &&
  \times \Bigg(
    a_0 \sum_{\lambda \in \Z^{(n,2)} }\left\{1|_{k,\mathcal{M}\left[ X^{-1} \smat{p}{0}{0}{1}^{-1} \right]}([(\lambda,0),0_2],M)\right\}(\tau,z\smat{p}{0}{0}{1} {^t X} \smat{p}{0}{0}{1})\\
  &&
  \qquad + a_1 \sum_{\lambda \in \Z^{(n,2)} }\left\{1|_{k,\mathcal{M}\left[ X^{-1} \smat{p}{0}{0}{1}^{-1} \right]}([(\lambda \smat{p}{0}{0}{1} ,0),0_2],M)\right\}(\tau,z\smat{p}{0}{0}{1} {^t X} \smat{p}{0}{0}{1}) \\
  &&
  \qquad + a_2 \sum_{\lambda \in \Z^{(n,2)} }\left\{1|_{k,\mathcal{M}\left[ X^{-1} \smat{p}{0}{0}{1}^{-1} \right]}([(\lambda \smat{p^2}{0}{0}{1} ,0),0_2],M)\right\}(\tau,z\smat{p}{0}{0}{1} {^t X} \smat{p}{0}{0}{1})
  \Bigg)\\
  &=&
  \left[H_{j,j} : H_{i,j} \right]
  \sum_{M \in \Gamma_{\infty}^{(n)}\backslash \Gamma_n} \\
  &&
  \times \Bigg(
    a_0 \sum_{\lambda \in \Z^{(n,2)} }\left\{1|_{k,\mathcal{M}\left[ X^{-1} \smat{p}{0}{0}{1}^{-1} \right]}([(\lambda,0),0_2],M)\right\}(\tau,z\smat{p}{0}{0}{1} {^t X} \smat{p}{0}{0}{1})\\
  &&
  \qquad + a_1 \sum_{\lambda \in \Z^{(n,2)} }\left\{1|_{k,\mathcal{M}}([(\lambda,0),0_2],M)\right\}(\tau,z\smat{p}{0}{0}{1}) \\
  &&
  \qquad + a_2 \sum_{\lambda \in \Z^{(n,2)} }\left\{1|_{k,\mathcal{M}\left[\smat{p}{0}{0}{1} \right]}([(\lambda,0),0_2],M)\right\}(\tau,z)
  \Bigg)\\
  &=&
  \left[H_{j,j} : H_{i,j} \right] \Bigg(
   a_0 E_{k,\mathcal{M}\left[ X^{-1} \smat{p}{0}{0}{1}^{-1} \right]}^{(n)}(\tau,z\smat{p}{0}{0}{1} {^t X} \smat{p}{0}{0}{1})
  \\
  &&
  \qquad + a_1 E_{k,\mathcal{M}}^{(n)}(\tau,z\smat{p}{0}{0}{1})
  + a_2 E_{k,\mathcal{M}\left[\smat{p}{0}{0}{1} \right]}^{(n)}(\tau,z)
  \Bigg).
\end{eqnarray*}

Similarly, the summation
\begin{eqnarray*}
  \sum_{M \in \Gamma(\delta_{i,j})\backslash \Gamma_n}
  \sum_{\lambda \in L'_{i,j}}
  \left\{1|_{k,\mathcal{M}\left[ X^{-1} \smat{p}{0}{0}{1}^{-1} \right]}([(\lambda,0),0_2],M)\right\}(\tau,z\smat{p}{0}{0}{1} {^t X} \smat{p}{0}{0}{1})
\end{eqnarray*}
is a linear combination of
$E_{k,\mathcal{M}\left[ X^{-1} \smat{p}{0}{0}{1}^{-1} \right]}^{(n)}(\tau,z\smat{p}{0}{0}{1} {^t X} \smat{p}{0}{0}{1})$,
$E_{k,\mathcal{M}}^{(n)}(\tau,z\smat{p}{0}{0}{1})$ and
$E_{k,\mathcal{M}\left[\smat{p}{0}{0}{1} \right]}^{(n)}(\tau,z)$.

Therefore, if $p|f$, then the form $\tilde{K}_{i,j}^{\alpha-i-n+j}(\tau,z)$
is a linear combination of the above three forms.

The proof for the case $p {\not | } f$ is similar to the case $p |f$.
If $p {\not |} f$,
then
$\tilde{K}_{i,j}^{\alpha-i-n+j}(\tau,z)$
is a linear combination of two forms
$E_{k,\mathcal{M}}^{(n)}(\tau,z\smat{p}{0}{0}{1})$ and
$E_{k,\mathcal{M}\left[\smat{p}{0}{0}{1} \right]}^{(n)}(\tau,z)$.
We omit the detail of the calculation here.
\end{proof}

\section{Commutativity with the Siegel operators}\label{s:siegelop}
In~\cite{Kr} an explicit commutative relation between the generators of Hecke operators
for Siegel modular forms and Siegel $\Phi$-operator has been given.
In this section we shall give a similar relation in the frameworks of
Jacobi forms of certain matrix index
and of Jacobi forms of half-integral weight.

Let $\mathcal{M} = \begin{pmatrix} l & \frac{r}{2} \\ \frac{r}{2} & 1 \end{pmatrix}$
be a $2\times2$ matrix and put $m = \det(2\mathcal{M})$ as before.

For any Jacobi form $\phi \in J_{k,\mathcal{M}}^{(n)}$,
or $\phi \in J_{k-\frac12,m}^{(n)}$ we define the Siegel $\Phi$-operator
\begin{eqnarray*}
  \Phi(\phi)(\tau', z')
  &:=&
  \lim_{t \rightarrow + \infty}
    \phi\left(\begin{pmatrix}
               \tau' & 0 \\ 0 & \sqrt{-1} t \end{pmatrix},
                  \begin{pmatrix} z' \\ 0 \end{pmatrix} \right)
\end{eqnarray*}
for $(\tau',z') \in \H_{n-1} \times \C^{(n-1,2)}$,
or for $(\tau',z') \in \H_{n-1} \times \C^{(n-1,1)}$. 
This Siegel $\Phi$-operator is a map
from $J_{k, \mathcal{M}}^{(n)}$ to $J_{k, \mathcal{M}}^{(n-1)}$,
or from $J_{k-\frac12, m}^{(n)}$ to $J_{k-\frac12, m}^{(n-1)}$,
respectively.

\begin{prop}\label{prop:gesetz_int}
For any Jacobi form $\phi \in J_{k,\mathcal{M}}^{(n)}$ and for any prime $p$,
we have
\begin{eqnarray*}
  \Phi(\phi |V_{\alpha,n-\alpha}(p^2))
  &=&
  \Phi(\phi) | V_{\alpha,n-\alpha}(p^2)^*,
\end{eqnarray*}
where $V_{\alpha,n-\alpha}(p^2)^*$ is a map
$V_{\alpha,n-\alpha}(p^2)^*$ $:$ $J_{k,\mathcal{M}}^{(n-1)} \rightarrow
                     J_{k,\mathcal{M}\left[\smat{p}{0}{0}{1}\right]}^{(n-1)}$ given by
\begin{eqnarray*}
 V_{\alpha,n-\alpha}(p^2)^*
 &=&
 p^{\alpha+2-k} V_{\alpha,n-\alpha-1}(p^2)\\
 &&
 + p ( 1 + p^{2n+1-2k}) V_{\alpha-1,n-\alpha}(p^2)\\
 &&
 + (p^{2n -2\alpha +2}-1) p^{\alpha-k} V_{\alpha-2,n-\alpha+1}(p^2).
\end{eqnarray*}
\end{prop}
\begin{proof}
We shall first show that there exists a linear combination of index-shift map
$V_{\alpha,n-\alpha}(p^2)^*$
which satisfies
$  \Phi(\phi |V_{\alpha,n-\alpha}(p^2))
  =
  \Phi(\phi) | V_{\alpha,n-\alpha}(p^2)^*
$.

We set $U = \begin{pmatrix} p^2 & 0 \\ 0 & p \end{pmatrix}$.
Let
\begin{eqnarray*}
 \phi(\tau,z) 
 &=&
 \sum_{N,R}A_1(N,R) e(N\tau + R{^t z}),
\\
 (\phi|V_{\alpha,n-\alpha}(p^2))(\tau,z)
 &=&
 \sum_{\hat{N},\hat{R}}A_2(\hat{N},\hat{R}) e(\hat{N}\tau + \hat{R}{^t z})
\end{eqnarray*}
be the Fourier expansions.
Let
 $\left\{ \begin{pmatrix} p^2 {^t D_j}^{-1} & B_{(j,l)} \\ 0_n & D_j \end{pmatrix}\right\}_{(j,l)}$
be a complete set of representatives of
 $\Gamma_n \backslash 
   \Gamma_n
     \mbox{diag}(1_{\alpha},p 1_{n-\alpha},p^2 1_{\alpha}, p 1_{n-\alpha})
    \Gamma_n$.
Then
the Fourier coefficients $A_2(\hat{N},\hat{R})$ have been calculated in the proof of
Proposition~\ref{prop:iota_hecke}:
\begin{eqnarray}\label{eq:a2_fourier}
 A_2(\hat{N},\hat{R})
 &=&
 p^{-k+n} \sum_{j} 
   \det(D_j)^{-k} \sum_{\lambda_2 \in (\Z/p\Z)^{(n,1)}}
   A_1(N,R)\, \sum_{l} e(NB_{(j,l)} D_j^{-1}).
\end{eqnarray}
Here $N$ and $R$ are determined by
\begin{eqnarray}
 N
 &=& 
 \frac{1}{p^2} D_j \left(\left(\hat{N} - \frac{1}{4}\hat{R}_2 {^t \hat{R}_2}\right) 
   + \frac{1}{4}(\hat{R}_2 - 2 \lambda_2) {^t(\hat{R}_2 - 2 \lambda_2)}\right){^t D_j},
 \label{eq:n_a2_fourier}
 \\
 R &=& D_j \left(\hat{R} - \frac{2}{p^2} \lambda U \mathcal{M} U \right) U^{-1},
 \notag
\end{eqnarray}
where we put
 $\hat{R}_2 = \hat{R}\left(\begin{smallmatrix}0 \\ 1 \end{smallmatrix}\right)$
and $\lambda = (0\ \lambda_2) \in \Z^{(n,2)}$.

By the definition of $V_{\alpha,n-\alpha}(p^2)$ there exists $\{\gamma_i\}_i$ such that
$\phi|V_{\alpha,n-\alpha}(p^2) = \sum_i \phi|_{k,\mathcal{M}} \gamma_i$.
We can take $\gamma_i$ as a form
\begin{eqnarray*}
  \gamma_i
  &=&
  \gamma_{(j,l,\lambda_2,\mu_2)}
  \ =\
  \left(\begin{pmatrix} p^2 {^t D_j}^{-1} & B_{(j,l)} \\ 0_n & D_j \end{pmatrix} \times
     \begin{pmatrix} U & 0_2 \\ 0_2 & p^2 U^{-1}\end{pmatrix},
     [((0\ \lambda_2),(0\ \mu_2)), 0_2]
 \right),
\end{eqnarray*}
where $B_{(j,l)} = \smat{B_{(j,l)}^*}{ b_1 }{ ^t b_3 }{ b_2 }$,
 $D_j = \smat{D_j^*}{ \mathfrak{d} }{0}{d_j}$,
 $\lambda_2 = \left(\begin{smallmatrix} \lambda^* \\ \lambda_3 \end{smallmatrix} \right)$,
 $\mu_2 = \left(\begin{smallmatrix} \mu^* \\ \mu_3 \end{smallmatrix} \right)$
 with
 $\smat{p^2 {^t D_j^*}^{-1}}{B_{(j,l)}^*}{0_{n-1}}{D_j^*} \in \mbox{GSp}_{n-1}^+(\Z)$,
 $\lambda^*$, $\mu^*$ $\in \Z^{(n-1,1)}$, and $d_j$, $\lambda_3$, $\mu_3$ $\in \Z$.
We set
\begin{eqnarray*}
  \gamma_i^*
  &:=&
  \gamma_{(j,l,\lambda^*,\mu^*)}^*
  \ =\
  \left(\begin{pmatrix} p^2 {^t D_j^*}^{-1} & B_{(j,l)}^* \\ 0_{n-1} & D_j^*
  \end{pmatrix} \times
     \begin{pmatrix} U & 0_2 \\ 0_2 & p^2 U^{-1}\end{pmatrix},
     [((0\ \lambda^*),(0\ \mu^*)), 0_2] \right).
\end{eqnarray*}

By the definition of Siegel $\Phi$-operator we have
\begin{eqnarray*}
\Phi\!\left(\sum_i \phi|_{k,\mathcal{M}} \gamma_i\right)\!(\tau^*,z^*)
&=&
 \Phi(\phi|V_{\alpha,n-\alpha}(p^2))(\tau^*,z^*)
 \, =\,
 \sum_{\hat{N},\hat{R}}A_2(\hat{N},\hat{R}) e(\hat{N}^*\tau^* + \hat{R}^*{^t z^*}),
\end{eqnarray*}
where $\tau^* \in \H_{n-1}$, $z^* \in \C^{(n-1,2)}$,
 $\hat{N} = \smat{\hat{N}^*}{0}{0}{0} \in \mbox{Sym}_n^*$,
 $\hat{N}^* \in \mbox{Sym}_{n-1}^*$,
 $\hat{R} 
    = \left(\begin{smallmatrix} \hat{R}^* \\ 0 \end{smallmatrix}\right) \in \Z^{(n,2)}$
 and
 $\hat{R}^* \in \Z^{(n-1,2)}$.
 
 Hence we need to calculate $A_2(\hat{N},\hat{R})$ for
 $\hat{N} = \smat{\hat{N}^*}{0}{0}{0}$ and $\hat{R} 
    = \left(\begin{smallmatrix} \hat{R}^* \\ 0 \end{smallmatrix}\right) \in \Z^{(n,2)}$.
 From the identity (\ref{eq:a2_fourier}) we need to calculate 
 \begin{eqnarray}\label{eq:nbd}
 \sum_{j} 
   \det(D_j)^{-k} \sum_{\lambda_2 \in (\Z/p\Z)^{(n,1)}}
   A_1(N,R)\, \sum_{l} e(NB_{(j,l)} D_j^{-1}).
 \end{eqnarray}
 We remark that the value $A_1(N,R)$ depends on the choice of
 $\hat{N}$, $\hat{R}$, $D_j$
 and $\lambda_3$.
 Under the conditions $\hat{N} \in \mbox{Sym}_n^*$ and $\hat{R} \in \Z^{(n,2)}$
 and by the identity (\ref{eq:n_a2_fourier})
 we can assume $d \lambda_3 \in p \Z$, since $A_1(N,R) = 0$ 
 unless $N \in \mbox{Sym}_n^*$.
 It is known that the value $A_1(N,R)$ depends only on
 $4N - R \mathcal{M}^{-1} {^t R}$ and on $R \mod 2 \mathcal{M}$.
 We now have
 \begin{eqnarray*}
   4N - R \mathcal{M}^{-1} {^t R}
   &=&
   \frac{1}{p^2}D_j \left( 4 \hat{N} - p^2 \hat{R} U^{-1} \mathcal{M}^{-1} U^{-1} {^t \hat{R}}  \right) {^t D_j}.
 \end{eqnarray*}
  We set
  \begin{eqnarray*}
    R' &=& D_j \left(\hat{R} - \frac{2}{p} (0\ \lambda_2) \mathcal{M} U \right) U^{-1} + \frac{2}{p} \begin{pmatrix} 0 & 0 \\ 0 & d_j \lambda_3\end{pmatrix}\mathcal{M}
  \end{eqnarray*}
  and
  \begin{eqnarray*}
   N' &=& \frac{1}{4 p^2}D_j \left( 4 \hat{N} - p^2 \hat{R} U^{-1} \mathcal{M}^{-1} U^{-1} {^t \hat{R}}  \right) {^t D_j}
   + \frac14 R' \mathcal{M}^{-1} {^t R'}.
  \end{eqnarray*}
  We remark that the last row of $R'$ is zero, and the last row and the last column of $N'$ are also zero.
  Because $4N - R\mathcal{M}^{-1} {^t R} = 4N' - R'\mathcal{M}^{-1} {^t R'}$
  and because $R - R' \in 2 \Z^{(n-1,2)} \mathcal{M}$, we have $A_1(N,R) = A_1(N',R')$.
 We write $N' = \smat{N'^{*}}{ }{ }{0}$ with $N'^{*} \in \mbox{Sym}_{n-1}^*$.

We have
\begin{eqnarray*}
 R' &=&
D_j \left(\hat{R} - \frac{2}{p} \left(0\ \lambda_2' \right) \mathcal{M} U \right) U^{-1},
\end{eqnarray*}
where $\lambda_2' = \left(\begin{smallmatrix}\lambda^* - D_j^{* -1} \mathfrak{d} \lambda_3 \\ 0 \end{smallmatrix}\right) \in \Qq^{(n,1)}$.
Now we will show $\lambda^* - D_j^{* -1} \mathfrak{d} \lambda_3 \in \Z^{(n-1,1)}$
if $\sum_{l} e(NB_{(j,l)} D_j^{-1}) \neq 0$ in the sum (\ref{eq:nbd}).

  We remark $d_j = 1$, $p$ or $p^2$.
  Because $p^2 D_j^{-1} \in \Z^{(n,n)}$ we have $p^2 {D_j^*}^{-1} \mathfrak{d} d_j^{-1} \in \Z^{(n-1,1)}$.
  If $d_j = 1$, then we can take $\mathfrak{d} = 0 \in \Z^{(n-1,1)}$ as a representative.
If $d_j = p^2$, then ${D_j^*}^{-1} \mathfrak{d} \in \Z^{(n-1,1)}$.
We now assume $d_j = p$. Then $p {D_j^*}^{-1} \mathfrak{d} \in \Z^{(n-1,1)}$.
By using the identity ${^t B_{(j,l)}} D_j = {^t D_j} B_{(j,l)}$ we have
  \begin{eqnarray*}
   && \!\!\!\!\!\!\!\!\!\!
   e(NB_{(j,l)} D_j^{-1}) \\
   &=&
   e(N' B_{(j,l)} D_j^{-1} - \frac{d_j \lambda_3}{2p} R \left(\begin{smallmatrix} 0 & \cdots & 0 & 0 \\ 0 & \cdots & 0 & 1 \end{smallmatrix}\right) B_{(j,l)} D_j^{-1}
    - \frac{d_j \lambda_3}{2p} \left(\begin{smallmatrix} 0 & 0 \\ \vdots & \vdots \\ 0 & 0 \\ 0 & 1 \end{smallmatrix}\right) {^t R} B_{(j,l)} D_j^{-1} \\
   &&
   - \frac{d_j^2 \lambda_3^2}{p^2} \left(\begin{smallmatrix} 0_{n-1} &  \\  & 1 \end{smallmatrix} \right) B_{(j,l)} D_j^{-1})\\
  &=& e(N'^{*} B_{(j,l)}^{*} {D_j^{*}}^{-1})
  \, e\!\left( - \frac{d_j \lambda_3}{p^2} (\hat{R}_2^*-2\lambda^* -   
  {D_j^*}^{-1} \mathfrak{d_j} \lambda_3) {^t b_3}\right)
  \, e\!\left(\frac{d_j \lambda_3^2}{p^2} b_2\right).
  \end{eqnarray*}
Hence, if $d_j = p$,
then $\displaystyle{\sum_{b_2 \mod p} e\!\left(\frac{d_j \lambda_3^2}{p^2} b_2\right)}$
is zero unless $\lambda_3 \equiv 0 \mod p$.
Thus, for any $d_j \in \{1, p, p^2\}$, we conclude
$\sum_{l} e(NB_{(j,l)} D_j^{-1}) = 0$ in the sum (\ref{eq:nbd})
unless ${D_j^*}^{-1} \mathfrak{d} \lambda_3 \in \Z^{(n-1,1)}$.
Hence $\lambda^* - D_j^{* -1} \mathfrak{d} \lambda_3 \in \Z^{(n-1,1)}$ and $\lambda_2' \in \Z^{(n,1)}$,
if $\sum_{l} e(NB_{(j,l)} D_j^{-1}) \neq 0$.

Therefore there exists a set of complex numbers $\{ C_{\gamma_i}\}_i := \{C_{\gamma_i,k,\mathcal{M}}\}_i$ which satisfies
\begin{eqnarray*}
 \Phi\left(\sum_i \phi|_{k,\mathcal{M}} \gamma_i\right) =
 \sum_i C_{\gamma_i^*} \Phi(\phi)|_{k,\mathcal{M}} \gamma_i^*.
\end{eqnarray*}
By a well-known argument we have
$\sum_i C_{\gamma_i^*} \gamma_i^* \gamma = \sum_i C_{\gamma_i^*} \gamma_i^*$ for any $\gamma \in \Gamma_{n-1,2}^J$.
Hence there exists an index-shift map
$V_{\alpha,n-\alpha}(p^2)^*$
which satisfies the identity
$  \Phi(\phi |V_{\alpha,n-\alpha}(p^2))
  =
  \Phi(\phi) | V_{\alpha,n-\alpha}(p^2)^*
$.

For a fixed $\alpha$ $(0\leq \alpha \leq n)$ the index-shift map $V_{\alpha,n-\alpha}(p^2)^*$ is a linear combination
of $V_{\beta,n-1-\beta}(p^2)$ $(\beta=0,...,n-1)$.
We need to determine these coefficients of the linear combination.
This calculation is similar to the case of Siegel modular forms
\cite[page 325]{Kr}.
We leave the details to the reader.

\end{proof}

Now for integers $l$ $(2\leq l)$, $\beta$ $(0 \leq \beta \leq l-1)$
and $\alpha$ $(0\leq \alpha \leq l)$, we put
\begin{eqnarray*}
b_{\beta,\alpha} &:=&
b_{\beta,\alpha,l,p}(X)
=
\begin{cases}
  (p^{l +1 -\alpha} - p^{-l-1+\alpha})  p^{\frac12} & \mbox{if } \beta = \alpha-2 \\
  (X + X^{-1}) p & \mbox{if } \beta = \alpha-1 \\
  p^{-l+\alpha + \frac32}  & \mbox{if } \beta = \alpha \\
  0 & \mbox{otherwise}
\end{cases}
\end{eqnarray*}
and set a matrix
\begin{eqnarray*}
  B_{l,l+1}(X) &:=&
  (b_{\beta,\alpha})_{\begin{smallmatrix} \beta = 0,...,l-1 \\ \alpha = 0,...,l \end{smallmatrix}}
  =
  \begin{pmatrix}
    b_{0,0} & \cdots & b_{0,l} \\
    \vdots & \ddots & \vdots \\
    b_{l-1,0} & \cdots & b_{l-1,l}
  \end{pmatrix}
\end{eqnarray*}
with entries in $\C[X + X^{-1}]$.
For any $\phi \in J_{k,\mathcal{M}}^{(l)}$, due to
Proposition~\ref{prop:gesetz_int}, we obtain
\begin{equation}
\begin{split}
  &
 \Phi(\phi)|(V_{0,l}(p^2)^*,\cdots, V_{l,0}(p^2)^*) \\
 \label{eq:VB}
 &=
 p^{-k+l+\frac12} \, \left(\Phi(\phi)|(V_{0,l-1}(p^2), \cdots, V_{l-1,0}(p^2))\right)
 B_{l,l+1}(p^{k-l-\frac12}).
\end{split}
\end{equation}
Here $\Phi(\phi)|(V_{0,l}(p^2)^*,\cdots, V_{l,0}(p^2)^*)$ denotes the row vector
\begin{eqnarray*}
  \Phi(\phi)|(V_{0,l}(p^2)^*,\cdots, V_{l,0}(p^2)^*)
  &:=&
  \left(\Phi(\phi)|V_{0,l}(p^2)^*, ..., \Phi(\phi)|V_{l,0}(p^2)^* \right).
\end{eqnarray*}

Let $J_{k-\frac12,m}^{(n) * }$ be the subspace of $J_{k-\frac12,m}^{(n)}$
introduced in \S\ref{ss:CE_J*}.
\begin{cor}\label{cor:gesetz_half}
For any Jacobi form $\phi \in J_{k-\frac12,m}^{(n) * }$
and for any prime $p$, we have
\begin{eqnarray*}
  \Phi(\phi | \tilde{V}_{\alpha,n-\alpha}(p^2))
  &=&
  \Phi(\phi)| \tilde{V}_{\alpha,n-\alpha}(p^2)^*,
\end{eqnarray*}
where $\tilde{V}_{\alpha,n-\alpha}(p^2)^*$ is a map
$\tilde{V}_{\alpha,n-\alpha}(p^2)^*$ $:$ $J_{k-\frac12,m}^{(n-1) *} \rightarrow
                     J_{k-\frac12,mp^2}^{(n-1)}$ given by
\begin{eqnarray*}
  \widetilde{V}_{\alpha,n-\alpha}(p^2)^*
  &=&
  p^{k-n-\frac12} \Bigg\{
  p^{-n + \alpha} \widetilde{V}_{\alpha,n-\alpha-1}(p^2)\\
  &&
  + ( p^{-k+n+\frac12} + p^{k - n -\frac12}) \widetilde{V}_{\alpha-1,n-\alpha}(p^2)\\
  &&
  + (p^{n+1-\alpha} - p^{-n-1+\alpha}) \widetilde{V}_{\alpha-2,n-\alpha+1}(p^2) \Bigg\}.
\end{eqnarray*}
\end{cor}
\begin{proof}
By a straightforward calculation we get the fact that
$\iota_{\mathcal{M}}$ and $\Phi$ is commutative.
The rest of the proof of this corollary follows from Proposition~\ref{prop:gesetz_int}
and Proposition~\ref{prop:iota_hecke}.
\end{proof}

  Let $\tilde{\mathcal{H}}_{p^2}^{(m)}$ be the local Hecke ring
  and let $R_m$ be the subring of a polynomial ring both defined in \S\ref{ss:L_function_siegel_half}.
  The isomorphism
   $ \Psi_m \, : \, \tilde{\mathcal{H}}_{p^2}^{(m)} \cong R_m$
  has been obtained in~\cite{Zhu:hecke,Zhu:euler} (see ~\S\ref{ss:L_function_siegel_half}).
  \begin{prop}\label{prop:imageofpsim}
  Let $p$ be an odd prime. 
  For any $m \geq 2$,
  the image of generators $K_{\alpha}^{(m)}$ of $\tilde{\mathcal{H}}_{p^2}^{(m)}$ by $\Psi_m$ are
  expressed as a vector
  \begin{eqnarray} \label{eq:isom_p_half}
  \begin{aligned}
  &
  \left(\Psi_m(K_0^{(m)}), \Psi_m(K_1^{(m)}), \cdots, \Psi_m(K_m^{(m)}) \right) \\
  &=
  p^{-\frac32 (m -1) }\ z_0^2 z_1 \cdots z_m (p^{-1},\ z_1+z_1^{-1})
  \begin{pmatrix}
    1 & 0 \\
    0 & p^\frac12
  \end{pmatrix}^{-1} \\
  & \qquad \times
  \left\{\prod_{l=2}^m B_{l,l+1}(z_l) \right\}
  \ \mbox{diag}(1,p^{\frac12},...,p^{\frac{m}{2}}).
  \end{aligned}
  \end{eqnarray}
  Here $B_{l,l+1}(X)$ is the $l\times (l+1)$-matrix introduced in above,
  and where
  \[
    \prod_{l=2}^m B_{l,l+1}(z_l) = B_{2,3}(z_2)B_{3,4}(z_3)\cdots B_{m,m+1}(z_m)
  \]
  is a
  $2 \times (m+1)$ matrix
  with entries in $\C[z_2^{\pm},\cdots, z_m^{\pm}]$.
  We remark that
  \[
   \Psi_m(K_0^{(m)}) = p^{-\frac{m(m+1)}{2}} z_0^2 z_1 \cdots z_m.
  \]
  \end{prop}
  \begin{proof}
  Let $k$ be an even integer and let $F \in M_{k-\frac12}(\Gamma_0^{(m)}(4))$
  be a Siegel modular form such that $\Phi^S(F) \not \equiv 0$.
  Here $\Phi^S$ denotes the Siegel $\Phi$-operator for Siegel modular forms.
  Let $T \in \tilde{H}_{p^2}^{(m)}$ and let $f_T(z_0,...,z_m) := \Psi_m(T) \in R_m$.
  Then $f_T(z_0,...z_{m-1},p^{k-m-\frac12}) \in R_{m-1}$ and
  $\Psi_{m-1}^{-1}(f_T(z_0,...,z_{m-1},p^{k-m-\frac12})) \in \tilde{H}_{p^2}^{(m-1)}$.
  It is known by Oh-Koo-Kim~\cite[Theorem~5.1]{OKK} that
  \begin{eqnarray}\label{eq:okk}
    \Phi^{S}(F|T) &=&
    \Phi^{S}(F)| \Psi_{m-1}^{-1}(f_T(z_0,...,z_{m-1},p^{k-m-\frac12})).
  \end{eqnarray}
  Let $\phi \in J_{k-\frac12,a}^{(m)}$ be a Jacobi form with index $a \in \Z$ such that
  $\Phi(\phi) \not \equiv 0$. Here $\Phi$ is the Siegel $\Phi$-operator.
  If $k$ is large enough,
  then there exists such $\phi$.
  Due to Corollary~\ref{cor:gesetz_half} we have
  \begin{eqnarray*}
    \Phi(\phi|\tilde{V}_{\alpha,m-\alpha}(p^2))
    &=&
    \Phi(\phi)|\tilde{V}_{\alpha,m-\alpha}(p^2)^*.
  \end{eqnarray*}
  Let $\mathbb{W} : J_{k-\frac12,a}^{(m)} \rightarrow M_{k-\frac12}^+(\Gamma_0^{(m)}(4))$
  be the Witt operator which is defined by
  \begin{eqnarray*}
    \mathbb{W}(\phi)(\tau) := \phi(\tau,0)
  \end{eqnarray*}
  for any $\phi(\tau,z) \in J_{k-\frac12,a}^{(m)}$.
  By a straightforward calculation, for any $\phi \in J_{k-\frac12,a}^{(m)}$ we have
  \begin{eqnarray*}
    \mathbb{W}(\phi|\tilde{V}_{\alpha,m-\alpha}(p^2))
    &=&
    \mathbb{W}(\phi)|\tilde{T}_{\alpha,m-\alpha}(p^2)
  \end{eqnarray*}
  and
  \begin{eqnarray*}
    \mathbb{W}(\Phi(\phi))
    &=&
    \Phi^S(\mathbb{W}(\phi)).
  \end{eqnarray*}
  
  We set
  \begin{eqnarray*}
      \widetilde{T}_{\alpha,m-\alpha}(p^2)^*
  &=&
  p^{k-m-\frac12} \Bigg\{
  p^{-m + \alpha} \widetilde{T}_{\alpha,m-\alpha-1}(p^2)\\
  &&
  + ( p^{-k+m+\frac12} + p^{k - m -\frac12}) \widetilde{T}_{\alpha-1,m-\alpha}(p^2)\\
  &&
  + (p^{m+1-\alpha} - p^{-m-1+\alpha}) \widetilde{T}_{\alpha-2,m-\alpha+1}(p^2) \Bigg\}.
  \end{eqnarray*}  
  If we put $F = \mathbb{W}(\phi)$, then
  \begin{eqnarray*}
    \Phi^S(F|\tilde{T}_{\alpha,m-\alpha}(p^2))
    &=&
      \Phi^S(\mathbb{W}(\phi|\tilde{V}_{\alpha,m-\alpha}(p^2)))\\
    &=&
      \mathbb{W}(\Phi(\phi|\tilde{V}_{\alpha,m-\alpha}(p^2)))\\
    &=&
      \mathbb{W}(\Phi(\phi)|\tilde{V}_{\alpha,m-\alpha}(p^2)^*)\\
    &=&
      \Phi^S(F)|\tilde{T}_{\alpha,m-\alpha}(p^2)^*.
  \end{eqnarray*}
  Hence if we put $T = \tilde{T}_{\alpha,m-\alpha}(p^2)$ in (\ref{eq:okk})
  we have
  \begin{eqnarray*}
    f_T(z_0,...,z_{m-1},p^{k-m-\frac12})
   &=&
      p^{k-m-\frac12} \Bigg\{
  p^{-m + \alpha} \Psi_{m-1}(K_{\alpha}^{(m-1)})\\
  &&
  + ( p^{-k+m+\frac12} + p^{k - m -\frac12}) \Psi_{m-1}(K_{\alpha-1}^{(m-1)})\\
  &&
  + (p^{m+1-\alpha} - p^{-m-1+\alpha}) \Psi_{m-1}(K_{\alpha-2}^{(m-1)}) \Bigg\}.
  \end{eqnarray*}
  Since this identity is true for infinitely many $k$, we have
  \begin{eqnarray*}
       \Psi_m(K_{\alpha}^{(m)})
  &=&
    f_T(z_0,...,z_{m-1},z_m) \\
  &=&
      z_m \Bigg\{
  p^{-m + \alpha} \Psi_{m-1}(K_{\alpha}^{(m-1)})\\
  &&
  + ( z_m + z_m^{-1}) \Psi_{m-1}(K_{\alpha-1}^{(m-1)})\\
  &&
  + (p^{m+1-\alpha} - p^{-m-1+\alpha}) \Psi_{m-1}(K_{\alpha-2}^{(m-1)}) \Bigg\}.  
  \end{eqnarray*}
  Hence
  \begin{eqnarray*}
  \begin{aligned}
  &
  \left(\Psi_m(K_0^{(m)}), \Psi_m(K_1^{(m)}), \cdots, \Psi_m(K_m^{(m)}) \right) \\
  &=
  p^{-3/2 }\ z_m\
  \left(\Psi_{m-1}(K_0^{(m-1)}), \Psi_{m-1}(K_1^{(m-1)}), \cdots, \Psi_{m-1}(K_{m-1}^{(m-1)}) \right) \\
  & \quad \times
  \mbox{diag}(1,p^{\frac12},...,p^{\frac{m-1}{2}})^{-1}
  B_{m,m+1}(z_m)
  \ \mbox{diag}(1,p^{\frac12},...,p^{\frac{m}{2}}).
  \end{aligned}
  \end{eqnarray*}
  Moreover, we have
  \begin{eqnarray*}
   \Psi_1\!\left(K_{0}^{(1)}\right)
   &=& p^{-1} z_0^2 z_1,\\
   \Psi_1\!\left(K_{1}^{(1)}\right)
   &=& z_0^2 (1 + z_1^2)
  \end{eqnarray*}
  by the definition of $\Psi_1$.
  This proposition follows from these identities and the recursion with respect to $m$.
  \end{proof}

\section{Maass relation for generalized Cohen-Eisenstein series}\label{s:maass_relation_cohen_eisen}

We put a $2\times(n+1)$-matrix
\begin{eqnarray*}
  A_{2,n+1}^{p}(X)
  &:=&
  \prod_{l=2}^{n} B_{l,l+1}(p^{\frac{n+2}{2}-l} X) \\
  &=&
  B_{2,3}(p^{\frac{n+2}{2}-2} X) B_{3,4}(p^{\frac{n+2}{2}-3}X) \cdots B_{n,n+1}(p^{\frac{n+2}{2}-n}X),
\end{eqnarray*}
where $B_{l,l+1}(X)$ is the $l \times (l+1)$-matrix introduced in \S\ref{s:siegelop}.

\begin{lemma}\label{lem:axinvariant}
 All components of the matrix $A_{2,2n-1}^p(X)$ belong to $\C[X+X^{-1}]$.
\end{lemma}
\begin{proof}
  We assume $p$ is an odd prime.
  Let $R_m$ be the symbol introduced in \S\ref{ss:L_function_siegel_half}.
  Because $\Psi_{2n-2}(K_{\alpha}^{(2n-2)})$ belongs to $R_{2n-2}$
  and because of Proposition~\ref{prop:imageofpsim}
  we have relations
 $B_{l,l+1}(z_l) = B_{l,l+1}(z_l^{-1})$ $(l=2,...,2n-2)$ and
 \begin{eqnarray*}
   B_{2,3}(z_2) B_{3,4}(z_3) \cdots B_{2n-2,2n-1}(z_{2n-2})
   &=&
   B_{2,3}(z_{2n-2}) B_{3,4}(z_{2n-3}) \cdots B_{2n-2,2n-1}(z_2).
 \end{eqnarray*}
 Hence  
  \begin{eqnarray*}
    A_{2,2n-1}^p(X)
    &=&
    B_{2,3}(p^{n-2} X) B_{3,4}(p^{n-3}X) \cdots B_{2n-2,2n-1}(p^{-n+2}X) \\
    &=&
    B_{2,3}(p^{-n+2}X^{-1}) B_{3,4}(p^{-n+3}X^{-1}) \cdots B_{2n-2,2n-1}(p^{n-2}X^{-1}) \\
    &=&
    B_{2,3}(p^{n-2}X^{-1}) B_{3,4}(p^{n-3}X^{-1}) \cdots B_{2n-2,2n-1}(p^{-n+2}X^{-1}) \\
    &=& A_{2,2n-1}^p(X^{-1}).
  \end{eqnarray*}

  The relation $A_{2,2n-1}^p(X) = A_{2,2n-1}^p(X^{-1})$
  holds for infinitely many $p$. 
  Hence if we regard that the components of the matrix $A_{2,2n-1}^p(X)$ are
  Laurent-polynomials of variables $X$ and $p^{1/2}$,
  then we obtain $A_{2,2n-1}^p(X) = A_{2,2n-1}^p(X^{-1})$.
  Hence we have also $A_{2,2n-1}^p(X) = A_{2,2n-1}^p(X^{-1})$ for $p=2$.
\end{proof}

Let $\mathcal{M}$, $m$, $D_0$ and $f$ be the symbols
used in the previous sections,
it means that $\mathcal{M} = \begin{pmatrix} * & * \\ * & 1 \end{pmatrix}$
is a $2\times 2$ half-integral symmetric-matrix, $m = \det(2\mathcal{M})$, $D_0$ is 
the discriminant of $\Qq(\sqrt{-m})$
and $f$ is a non-negative integer which satisfies
$m = D_0 f^2$.

For any prime $p$ we set
\begin{eqnarray*}
  \begin{pmatrix}
    a_{0,m,p,k} \\
    a_{1,m,p,k} \\
    a_{2,m,p,k}
  \end{pmatrix}
  &:=&
  \begin{cases}
    \begin{pmatrix}
      p^{-3k+4} \\
      0 \\
      p^{-k+1}
    \end{pmatrix}
    &
    \mbox{if } p|f, \\
    \begin{pmatrix}
      0 \\
      p^{-3k+4} + p^{-2k+2} \left( \frac{-m}{p} \right) \\
      p^{-k+1} - p^{-2k+2} \left( \frac{-m}{p} \right)
    \end{pmatrix}
    &
    \mbox{if } p {\not |} f.
  \end{cases}
\end{eqnarray*}

\begin{lemma}\label{lem:Eisen_eins}
 For the Jacobi-Eisenstein series $E_{k,\mathcal{M}}^{(1)}$ of weight $k$
of index $\mathcal{M}$ of degree $1$,
 we have the identity
 \begin{eqnarray*}
   &&
   E_{k,\mathcal{M}}^{(1)} | (V_{0,1}(p^2), V_{1,0}(p^2)) \\
   &=&
   \left( E_{k,\mathcal{M}\left[ X^{-1} \smat{p}{0}{0}{1}^{-1} \right]}^{(1)}|
              U_{\smat{p}{0}{0}{1}  X \smat{p}{0}{0}{1} },
            E_{k,\mathcal{M}}^{(1)} |U_{\smat{p}{0}{0}{1}},
            E_{k,\mathcal{M}\left[\smat{p}{0}{0}{1} \right]}^{(1)}\right)
   \begin{pmatrix}
     0 & a_{0,m,p,k} \\
     p^{-2k+2} & a_{1,m,p,k} \\
     0 & a_{2,m,p,k}
   \end{pmatrix},
 \end{eqnarray*}
where $X = \smat{1}{0}{x}{1} \in \Z^{(2,2)}$ is a matrix such that
$\mathcal{M}[X^{-1} \smat{p}{0}{0}{1}^{-1} ] \in Sym_2^{+}$.
Here, if $p {\not|} f$, there does not exist such matrix $X$
and we regard $E_{k,\mathcal{M}\left[ X^{-1} \smat{p}{0}{0}{1}^{-1} \right]}^{(1)}$ as zero.
\end{lemma}
\begin{proof}
From Proposition~\ref{prop:tilde_kij} and due to~(\ref{eq:Lijs1}), (\ref{eq:Lijs2})
in the proof of Proposition~\ref{prop:E_linear},
we have
\begin{eqnarray*}
 E_{k,\mathcal{M}}^{(1)}|V_{0,1}(p^2)
 &=& \tilde{K}_{0,1}^{0} \\
 &=&
 p^{-2k+1 }
 \sum_{M \in \Gamma(\delta_{0,1})\backslash \Gamma_1}
 \sum_{ \lambda_2 \in L_{0,1}^*}
  \left\{1|_{k,\mathcal{M}}([(\lambda_2,0),0],M)\right\}(\tau,z\smat{p}{0}{0}{1})
\\
 &&
 \times
  \sum_{u_2 \in \Z/p\Z} G_{\mathcal{M}}^{1,1}(\lambda_2 + (0,u_2)) \\
  &=&
  p^{-2k+2}
   \sum_{M \in \Gamma_{\infty}^{(1)} \backslash \Gamma_1}
 \sum_{ \lambda_2 \in \Z^{(1,2)}}
  \left\{1|_{k,\mathcal{M}}([(\lambda_2,0),0],M)\right\}(\tau,z\smat{p}{0}{0}{1})\\
  &=&
  p^{-2k+2}
  E_{k,\mathcal{M}}^{(1)}|U_{\smat{p}{0}{0}{1}}.
\end{eqnarray*}
From Proposition \ref{prop:tilde_kij} we also have
\begin{eqnarray*}
 E_{k,\mathcal{M}}^{(1)}|V_{1,0}(p^2)
 &=&
 \tilde{K}_{1,1}^{0}
 +
 \tilde{K}_{0,1}^{1}
 +
 \tilde{K}_{0,0}^{0}.
\end{eqnarray*}
Here
\begin{eqnarray*}
 \tilde{K}_{1,1}^{0} &=&
 p^{-k+1}\sum_{M \in \Gamma(\delta_{1,1})\backslash \Gamma_1}
 \sum_{ \lambda_1 \in L_{1,1}^*}
 \left\{1|_{k,\mathcal{M}}([(\lambda_1,0),0],M)\right\}(\tau,z\smat{p}{0}{0}{1})\\
 &=&
 p^{-k+1}\sum_{M \in \Gamma_{\infty}^{(1)}\backslash \Gamma_1}
 \sum_{ \lambda_1 \in p\Z \times \Z}
 \left\{1|_{k,\mathcal{M}}([(\lambda_1,0),0],M)\right\}(\tau,z\smat{p}{0}{0}{1})\\
 &=&
 p^{-k+1}\sum_{M \in \Gamma_{\infty}^{(1)}\backslash \Gamma_1}
 \sum_{ \lambda \in \Z^{(1,2)}}
 \left\{1|_{k,\mathcal{M}}([(\lambda \smat{p}{ }{ }{1} ,0),0],M)\right\}(\tau,z\smat{p}{0}{0}{1})\\
 &=&
 p^{-k+1}\sum_{M \in \Gamma_{\infty}^{(1)}\backslash \Gamma_1}
 \sum_{ \lambda \in \Z^{(1,2)}}
 \left\{1|_{k,\mathcal{M}[\smat{p}{ }{ }{1} ]}([(\lambda,0),0],M)\right\}(\tau,z)\\
 &=&
 p^{-k+1} E_{k,\mathcal{M}[\smat{p}{ }{ }{1} ]}^{(1)}(\tau,z).
\end{eqnarray*}
Now we shall calculate $\tilde{K}_{0,1}^{1}$. First, due to Lemma~\ref{lem:sum_gs}
we have
\begin{eqnarray*}
 \sum_{u_2 \in \Z/p\Z} G_{\mathcal{M}}^{1,0}(\lambda + (0,u_2))
 &=&
 \begin{cases}
  0 & \mbox{ if } \lambda \in \Z^{(1,2)} \smat{p}{0}{0}{1} ,\\
  \left(\frac{ -m }{p}\right) p & \mbox{ if } \lambda \not\in \Z^{(1,2)} \smat{p}{0}{0}{1}
 \end{cases}
\end{eqnarray*}
for any $\lambda \in \Z^{(1,2)}$.
Thus
\begin{eqnarray*}
 \tilde{K}_{0,1}^{1}
 &=&
 p^{-2k+2}\sum_{M \in \Gamma(\delta_{0,1})\backslash \Gamma_1}
 \sum_{ \lambda_2 \in L_{0,1}^*}
 \left\{1|_{k,\mathcal{M}}([(\lambda_2,0),0],M)\right\}(\tau,z\smat{p}{0}{0}{1})\\
 && \times
 \sum_{u_2 \in \Z/p\Z}G_{\mathcal{M}}^{1,0}(\lambda_2 + (0,u_2)) \\
 &=&
 -p^{-2k+2} \left(\frac{-m}{p} \right)\sum_{M \in \Gamma_{\infty}^{(1)}\backslash \Gamma_1}
 \sum_{ \lambda  \in \Z^{(1,2)}  }
 \left\{1|_{k,\mathcal{M}}([(\lambda \smat{p}{0}{0}{1} ,0),0],M)\right\}(\tau,z\smat{p}{0}{0}{1}) \\
 &&
 +p^{-2k+2} \left(\frac{-m}{p} \right)\sum_{M \in \Gamma_{\infty}^{(1)}\backslash \Gamma_1}
 \sum_{ \lambda \in \Z^{(1,2)} }
 \left\{1|_{k,\mathcal{M}}([(\lambda,0),0],M)\right\}(\tau,z\smat{p}{0}{0}{1}) \\
 &=&
 -p^{-2k+2} \left( \frac{-m}{p} \right) E_{k,\mathcal{M} \left[\smat{p}{0}{0}{1} \right]}^{(1)}(\tau,z)
 +p^{-2k+2} \left( \frac{-m}{p} \right) E_{k,\mathcal{M}}(\tau,z\smat{p}{0}{0}{1})\\
 &=&
 -p^{-2k+2} \left( \frac{-m}{p} \right) E_{k,\mathcal{M} \left[\smat{p}{0}{0}{1} \right]}^{(1)}
 +p^{-2k+2} \left( \frac{-m}{p} \right) E_{k,\mathcal{M}}|U_{\smat{p}{0}{0}{1}}.
\end{eqnarray*}
We shall calculate $\tilde{K}_{0,0}^{0}$.
Due to (\ref{eq:Lijs1}) and due to (\ref{eq:Lijs2}) we have
\begin{eqnarray*}
 L_{0,0}^*
 &=&
 \begin{cases}
 \Z^{(1,2)} \smat{p}{0}{0}{1}^{-1} {^t X}^{-1}
 & \mbox{ if } p|f ,\\
 \Z^{(1,2)}
 & \mbox{ if } p {\not|} f .
 \end{cases}
\end{eqnarray*}
Thus if $p|f$, then
\begin{eqnarray*}
 \tilde{K}_{0,0}^{0}
 &=&
 p^{-3k+4} \sum_{M \in \Gamma(\delta_{0,0})\backslash \Gamma_1}
 \sum_{\lambda_3 \in \Z^{(1,2)} \smat{p}{0}{0}{1}^{-1} {^t X}^{-1}}
 \left\{1|_{k,\mathcal{M}}([(\lambda_3,0),0],M)\right\}(\tau,z\smat{p}{0}{0}{1}) \\
 &=&
 p^{-3k+4} \sum_{M \in \Gamma_{\infty}^{(1)}\backslash \Gamma_1}
 \sum_{\lambda_3 \in \Z^{(1,2)} \smat{p}{0}{0}{1}^{-1} {^t X}^{-1}} \\
 && \times
 \left\{1|_{k,\mathcal{M}[X^{-1}\smat{p}{ }{ }{1}^{-1}]}([(\lambda_3 {^t X} \smat{p}{ }{ }{1},0),0],M)\right\}(\tau,z\smat{p}{ }{ }{1} {^t X} \smat{p}{ }{ }{1}) \\
 &=&
 p^{-3k+4} E_{k,\mathcal{M}[X^{-1}\smat{p}{ }{ }{1}^{-1}]}^{(1)}(\tau,z \smat{p}{ }{ }{1} {^t X} \smat{p}{ }{ }{1}),
\end{eqnarray*}
and if $p {\not|} f$, then
\begin{eqnarray*}
  \tilde{K}_{0,0}^{0}
 &=&
 p^{-3k+4} \sum_{M \in \Gamma(\delta_{0,0})\backslash \Gamma_1}
 \sum_{\lambda_3  \in \Z^{(1,2)} }
 \left\{1|_{k,\mathcal{M}}([(\lambda_3,0),0],M)\right\}(\tau,z\smat{p}{0}{0}{1})\\
 &=&
 p^{-3k+4} E_{k,\mathcal{M}}^{(1)}(\tau,z \smat{p}{ }{ }{1}).
\end{eqnarray*}
Hence we obtain the formula for $\tilde{K}_{0,0}^{0}$.

Because $E_{k,\mathcal{M}}^{(1)}|V_{1,0}(p^2)
 =
 \tilde{K}_{1,1}^{0}
 +
 \tilde{K}_{0,1}^{1}
 +
 \tilde{K}_{0,0}^{0}$, we conclude the lemma.
\end{proof}

\begin{lemma}\label{lem:lin_indep}
  The three Jacobi forms
  $E_{k,\mathcal{M}\left[ X^{-1} \smat{p}{0}{0}{1}^{-1} \right]}^{(1)}|
     U_{\smat{p}{0}{0}{1}  X \smat{p}{0}{0}{1} }$,
   $E_{k,\mathcal{M}}^{(1)} |U_{\smat{p}{0}{0}{1}}$
   and
   $E_{k,\mathcal{M}\left[\smat{p}{0}{0}{1} \right]}^{(1)}$
  are linearly independent.
\end{lemma}
\begin{proof}
  We first assume that $\mathcal{M} \in \mbox{Sym}_g^+$ is a positive-definite
  half-integral symmetric matrix of size $g$.
  Let
  \begin{eqnarray*}
    E_{k,\mathcal{M}}^{(1)}(\tau,z)
            &=&
       \sum_{
          \begin{smallmatrix}
              n \in \Z , R \in \Z^{(1,g)} \\
              4n - R \mathcal{M}^{-1} {^t R} \geq 0
           \end{smallmatrix}}
       C_{k, \mathcal{M}}(n,R)\, e(n \tau + R {^tz})
  \end{eqnarray*}
  be the Fourier expansion of the Jacobi-Eisenstein series $E_{k,\mathcal{M}}^{(1)}$.
  For any pair
  $n \in \Z$ and $R \in \Z^{(1,g)}$ which satisfy $4n - R\mathcal{M}^{-1} {^t R} > 0$,
  we now show that $C_{k, \mathcal{M}}(n,R) \neq 0$.
  The Fourier coefficients of Jacobi-Eisenstein series of degree $1$ of integer index
  have been calculated in~\cite[pp.17--22]{EZ}.
  If $4n - R \mathcal{M}^{-1} {^t R} > 0$, by an argument similar to~\cite{EZ} we have
  \begin{eqnarray*}
    C_{k, \mathcal{M}}(n,R)
    &=&
    \frac{(-1)^{\frac{k}{2}}\pi^{k-\frac{g}{2}}}{2^{k-2} \Gamma(k-\frac{g}{2})}
    \frac{(4n-R\mathcal{M}^{-1} {^t R})^{k-\frac{g}{2}-1}}{\det(\mathcal{M})^{\frac12} }
    \zeta(k-g)^{-1}
    \sum_{a=1}^{\infty} \frac{N_a(Q)}{a^{k-1}},
  \end{eqnarray*}
  where $N_a(Q) := \left| \left\{ \lambda \in \left(\Z / a \Z \right)^{(1,g)} \ | \, \lambda \mathcal{M}^{-1} {^t \lambda} + R {^t \lambda} + n \equiv 0 \mod a  \right\} \right|$.
  Hence we conclude $C_{k, \mathcal{M}}(n,R) \neq 0$.

  We now assume $\mathcal{M} = \smat{l}{\frac{r}{2}}{\frac{r}{2}}{1} \in \mbox{Sym}_2^+$.
  The $(n,R)$-th Fourier coefficient of two Jacobi forms
  $E_{k,\mathcal{M}}^{(1)}|U_{\smat{p}{0}{0}{1}}$
  and $E_{k,\mathcal{M}\left[ \smat{p}{0}{0}{1} \right]}^{(1)}$ are
  $C_{k,\mathcal{M}}(n,R \smat{p}{0}{0}{1}^{-1})$
  and $C_{k,\mathcal{M}\left[  \smat{p}{0}{0}{1} \right]}(n,R)$, respectively.
  If $R {\not \in} \Z^{(1,2)} \smat{p}{0}{0}{1}$, then
  $C_{k,\mathcal{M}}(n,R \smat{p}{0}{0}{1}^{-1}) = 0$
  and
  $C_{k,\mathcal{M}\left[  \smat{p}{0}{0}{1} \right]}(n,R) \neq 0$.
  Hence $E_{k,\mathcal{M}}^{(1)}|U_{\smat{p}{0}{0}{1}}$
  and $E_{k,\mathcal{M}\left[ \smat{p}{0}{0}{1} \right]}^{(1)}$ are linearly independent.
  The proof for the linear independence of the three forms of the lemma
  is similar. We omitted the detail here.
\end{proof}

\begin{prop}\label{prop:Eisenstein_Voperator}
We obtain the identity
\begin{eqnarray*}
  &&  \!\!\!\!\!\!\!\!\!\!
  E_{k,\mathcal{M}}^{(n)} | \left( V_{0,n}(p^2), \cdots , V_{n,0}(p^2) \right)\\
  &=&
  \left(
    E_{k,\mathcal{M}\left[ X^{-1} \smat{p}{0}{0}{1}^{-1} \right]}^{(n)}|
              U_{\smat{p}{0}{0}{1}  X \smat{p}{0}{0}{1} },
            E_{k,\mathcal{M}}^{(n)} |U_{\smat{p}{0}{0}{1}},
            E_{k,\mathcal{M}\left[\smat{p}{0}{0}{1} \right]}^{(n)}
  \right) \\
  &&
  \times
   \begin{pmatrix}
     0 & a_{0,m,p,k} \\
     p^{-2k+2} & a_{1,m,p,k} \\
     0 & a_{2,m,p,k}
   \end{pmatrix}
    p^{(-k+\frac12 n + \frac32)(n-1)}
    A_{2,n+1}^{p}(p^{k-\frac{n+2}{2}-\frac12}),
\end{eqnarray*}
where the $2\times (n+1)$-matrix $A_{2,n+1}^{p}(p^{k-\frac{n+2}{2}-\frac12})$ is introduced
in the beginning of this section.
\end{prop}
\begin{proof}
Let $\Phi$ be the Siegel $\Phi$-operator introduced in \S\ref{s:siegelop}.
From the definition of Jacobi-Eisenstein series, we have
$\Phi(E_{k,\mathcal{M}}^{(l)}) = E_{k,\mathcal{M}}^{(l-1)}$.

From the identity~(\ref{eq:VB}) in \S\ref{s:siegelop} and from Lemma~\ref{lem:Eisen_eins}, we obtain
\begin{eqnarray}
  \notag
  &&  \!\!\!\!\!\!\!\!\!\!
  \Phi^{n-1}(E_{k,\mathcal{M}}^{(n)} | (V_{0,n}(p^2), ..., V_{n,0}(p^2)))\\
  \notag
  &=&
  \left( E_{k,\mathcal{M}}^{(1)} | (V_{0,1}(p^2), V_{1,0}(p^2)) \right) \left(\prod_{l=2}^{n}  p^{-k+l+\frac12} \right) B_{2,3}(p^{k-\frac52}) \cdots B_{n,n+1}(p^{k-n-\frac12}) \\
  \notag
  &=&
  p^{(-k+\frac12)(n-1)+\frac{n(n+1)}{2}-1} \left(E_{k,\mathcal{M}}^{(1)} | (V_{0,1}(p^2), V_{1,0}(p^2)) \right)
  A_{2,n+1}^{p}(p^{k-\frac{n+2}{2}-\frac12}) \\
  \label{eq:rel_BM}
  &=&
   \left(
    E_{k,\mathcal{M}\left[ X^{-1} \smat{p}{0}{0}{1}^{-1} \right]}^{(1)}|
              U_{\smat{p}{0}{0}{1}  X \smat{p}{0}{0}{1} },
            E_{k,\mathcal{M}}^{(1)} |U_{\smat{p}{0}{0}{1}},
            E_{k,\mathcal{M}\left[\smat{p}{0}{0}{1} \right]}^{(1)}
  \right) \\
  \notag
  &&
  \times
    p^{(-k+\frac12 n + \frac32)(n-1)} 
   \begin{pmatrix}
     0 & a_{0,m,p,k} \\
     p^{-2k+2} & a_{1,m,p,k} \\
     0 & a_{2,m,p,k}
   \end{pmatrix} 
   A_{2,n+1}^{p}(p^{k-\frac{n+2}{2}-\frac12}) .
\end{eqnarray}

From Proposition~\ref{prop:E_linear}
there exists a matrix $M \in \C^{(3,n+1)}$ which satisfies
\begin{eqnarray}
  \notag
  &&  \!\!\!\!\!\!\!\!\!\!
  E_{k,\mathcal{M}}^{(n)} | (V_{0,n}(p^2), ..., V_{n,0}(p^2)) \\
  \label{eq:EnV}
  &=&
     \left(
    E_{k,\mathcal{M}\left[ X^{-1} \smat{p}{0}{0}{1}^{-1} \right]}^{(n)}|
              U_{\smat{p}{0}{0}{1}  X \smat{p}{0}{0}{1} },
            E_{k,\mathcal{M}}^{(n)} |U_{\smat{p}{0}{0}{1}},
            E_{k,\mathcal{M}\left[\smat{p}{0}{0}{1} \right]}^{(n)}
  \right)
  M.
\end{eqnarray}
Thus
\begin{eqnarray*}
    &&  \!\!\!\!\!\!\!\!\!\!
  \Phi^{n-1}(E_{k,\mathcal{M}}^{(n)} | (V_{0,n}(p^2), ..., V_{n,0}(p^2)))\\
  &=&
   \left(
      E_{k,\mathcal{M}\left[ X^{-1} \smat{p}{0}{0}{1}^{-1} \right]}^{(1)}|
              U_{\smat{p}{0}{0}{1}  X \smat{p}{0}{0}{1} },
            E_{k,\mathcal{M}}^{(1)} |U_{\smat{p}{0}{0}{1}},
            E_{k,\mathcal{M}\left[\smat{p}{0}{0}{1} \right]}^{(1)}
  \right) 
  M.
\end{eqnarray*}
From Lemma~\ref{lem:lin_indep} the matrix $M$ is uniquely determined.
Therefore, by using the identity~(\ref{eq:rel_BM}), we have
\begin{eqnarray*}
  M &=&
    p^{(-k+\frac12 n + \frac32)(n-1)} 
     \begin{pmatrix}
     0 & a_{0,m,p,k} \\
     p^{-2k+2} & a_{1,m,p,k} \\
     0 & a_{2,m,p,k}
   \end{pmatrix} 
   A_{2,n+1}^{p}(p^{k-\frac{n+2}{2}-\frac12}) .
\end{eqnarray*}
Therefore we conclude that this Proposition follows from the identity~(\ref{eq:EnV}).
\end{proof}

We recall that the form $e_{k,\mathcal{M}}^{(n)}$ is the $\mathcal{M}$-th
Fourier-Jacobi coefficient of Siegel-Eisenstein series $E_k^{(n+2)}$
of weight $k$ of degree $n+2$.

\begin{prop}\label{prop:e_k_M_V}
We obtain the identity
\begin{eqnarray*}
 &&  \!\!\!\!\!\!\!\!\!\!
 e_{k,\mathcal{M}}^{(n)}|(V_{0,n}(p^2), ..., V_{n,0}(p^2))
 \\
 &=&
 p^{(-k+\frac12 n + \frac32)(n-1)}
 \left(
      e_{k,\mathcal{M}\left[ X^{-1} \smat{p}{0}{0}{1}^{-1} \right]}^{(n)}|
              U_{\smat{p}{0}{0}{1}  X \smat{p}{0}{0}{1} },\
      e_{k,\mathcal{M}}^{(n)} |U_{\smat{p}{0}{0}{1}},\
      e_{k,\mathcal{M}\left[\smat{p}{0}{0}{1} \right]}^{(n)}
 \right)\\
 &&
 \times
    \begin{pmatrix}
   0 & p^{-k+1} \\
   p^{-2k+2} & p^{-2k+2} \left( \frac{-m}{p} \right) \\
   0 & p^{-3k+4}
   \end{pmatrix}
   A_{2,n+1}^{p}(p^{k-\frac{n+2}{2}-\frac12}) .
\end{eqnarray*}
\end{prop}
\begin{proof}
  For any $\phi \in J_{k,\mathcal{M}}^{(n)}$ and for any $L = \smat{a}{0}{b}{1} \in \Z^{(2,2)}$, a straightforward calculation gives the identity
  \begin{eqnarray}\label{eq:U_L_V}
    \left(\phi|U_L\right) | V_{\alpha,n-\alpha}(p^2)
    &=&
    \left(\phi | V_{\alpha,n-\alpha}(p^2) \right) | U_{\smat{p}{0}{0}{1}^{-1} L \smat{p}{0}{0}{1}}.
  \end{eqnarray}
  
  We recall from Proposition~\ref{prop:fourier_jacobi} the identity
  \begin{eqnarray*}
       e_{k,\mathcal{M}}^{(n)}
     &=&
       \sum_{d|f} g_k\!\left(\frac{m}{d^2}\right)
         E_{k,\mathcal{M}[W_d^{-1}]}^{(n)}(\tau,z {^t W_d})\\
     &=&
       \sum_{d|f} g_k\!\left(\frac{m}{d^2}\right)
         E_{k,\mathcal{M}[W_d^{-1}]}^{(n)}|U_{W_d}
  \end{eqnarray*}
  where $W_d$ is a matrix such that
  $\mathcal{M}\left[W_d^{-1}\right] \in \mbox{Sym}_2^+$.
  We choose the set of matrices $\{W_d\}_d$ which satisfies
  $\mathcal{M}\left[W_d^{-1} \smat{p}{0}{0}{1}^{-1} \right] \in \mbox{Sym}_2^+$,
  if $d|\frac{f}{p}$.
  In particular, we choose $W_{pd} = \smat{p}{0}{0}{1} W_d$ for $d$ such that $pd | f$.
  By virtue of Lemma~\ref{lemma:eisen_A} and of the identity $W_{pd} = \smat{p}{0}{0}{1} W_d$,
  we have
  \begin{eqnarray}\label{eq:Wpd}
    E_{k,\mathcal{M}[W_d^{-1}]}^{(n)}|U_{W_d \smat{p}{0}{0}{1}}
    &=&
    E_{k,\mathcal{M}[\smat{p}{0}{0}{1} W_{pd}^{-1}]}^{(n)}|U_{W_{pd}}.
  \end{eqnarray}
  For the sake of simplicity we write
  \begin{eqnarray}\label{eq:E012}
  \begin{aligned}
    E_0(d) &= E_{k,\mathcal{M}\left[ W_{pd}^{-1} \right] }^{(n)}|U_{W_{pd} \smat{p}{0}{0}{1}},\\
    E_1(d) &= E_{k,\mathcal{M}\left[ W_d^{-1} \right] }^{(n)}|U_{W_d \smat{p}{0}{0}{1}},\\
    E_2(d) &=
    E_{k,\mathcal{M}\left[\smat{p}{0}{0}{1} W_d^{-1} \right] }^{(n)}|U_{W_d}.
  \end{aligned}
  \end{eqnarray}
  We remark $E_0(d) = E_1(pd)$ and $E_1(d) = E_2(pd)$
  due to the identity~(\ref{eq:Wpd}).

  From Proposition~\ref{prop:Eisenstein_Voperator}
  and due to identities (\ref{eq:U_L_V}) and (\ref{eq:E012}) we get
  \begin{eqnarray*}
    &&  \!\!\!\!\!\!\!\!\!\!
    \left(E_{k,\mathcal{M}\left[W_d^{-1}\right]}^{(n)}|U_{W_d}\right)|\left( V_{0,n}(p^2),\cdots,V_{n,0}(p^2)\right)\\
    &=&
    p^{(-k+\frac12 n + \frac32)(n-1)}
   \left(E_0(d),E_1(d),E_2(d) \right)
      \begin{pmatrix}
     0 & a_{0,\frac{m}{d^2},p,k} \\
     p^{-2k+2} & a_{1,\frac{m}{d^2},p,k} \\
     0 & a_{2,\frac{m}{d^2},p,k}
   \end{pmatrix}
    A_{2,n+1}^{p}(p^{k-\frac{n+2}{2}-\frac12}).
  \end{eqnarray*}
  
  Hence from Proposition~\ref{prop:fourier_jacobi} we have
  \begin{eqnarray*}
     &&  \!\!\!\!\!\!\!\!\!\!
     e_{k,\mathcal{M}}^{(n)}|(V_{0,n}(p^2), ..., V_{n,0}(p^2))\\
     &=&
     \sum_{d|f} g_k\!\left(\frac{m}{d^2}\right)
                \left(E_{k,\mathcal{M}\left[W_d^{-1}\right]}^{(n)}|U_{W_d}\right)|(V_{0,n}(p^2), ..., V_{n,0}(p^2)) \\
     &=&
        p^{(-k+\frac12 n + \frac32)(n-1)}
     \sum_{d|f} g_k\!\left(\frac{m}{d^2}\right) \\
    && \qquad \times
       \left(
       E_0(d), E_1(d), E_2(d)
       \right)
      \begin{pmatrix}
        0 & a_{0,\frac{m}{d^2},p,k} \\
        p^{-2k+2} & a_{1,\frac{m}{d^2},p,k} \\
        0 & a_{2,\frac{m}{d^2},p,k}
      \end{pmatrix} 
        A_{2,n+1}^{p}(p^{k-\frac{n+2}{2}-\frac12}).
  \end{eqnarray*}
  On the RHS of the above identity we obtain
  \begin{eqnarray*}
     &&  \!\!\!\!\!\!\!\!\!\!
     \sum_{d|f} g_k\!\left(\frac{m}{d^2}\right)
       \left(
	E_0(d), E_1(d), E_2(d)
       \right)
      \begin{pmatrix}
        0 \\
        p^{-2k+2}  \\
        0 
      \end{pmatrix} \\
      &=&
       p^{-2k+2}\, \sum_{d|f} g_k\!\left(\frac{m}{d^2}\right)
           E_1(d)\\
      &=&
       p^{-2k+2}\, \sum_{d|f} g_k\!\left(\frac{m}{d^2}\right)
            E_{k,\mathcal{M}\left[W_d^{-1} \right]}^{(n)} |U_{W_d \smat{p}{0}{0}{1}}\\
      &=&
       p^{-2k+2} e_{k,\mathcal{M}}^{(n)}|U_{\smat{p}{0}{0}{1}}.
  \end{eqnarray*}
  By using Lemma~\ref{lemma:gk} we now have
\begin{eqnarray*}
     &&  \!\!\!\!\!\!\!\!\!\!
     \sum_{d|f} g_k\!\left(\frac{m}{d^2}\right)
       \left(
            E_0(d),
            E_1(d),
            E_2(d)
       \right)
      \begin{pmatrix}
         a_{0,\frac{m}{d^2},p,k} \\
         a_{1,\frac{m}{d^2},p,k} \\
         a_{2,\frac{m}{d^2},p,k}
      \end{pmatrix}\\
    &=&
      \sum_{d|f} g_k\!\left(\frac{m}{d^2}\right) \Bigg\{
       \delta\!\left(p|\frac{f}{d}\right) p^{-3k+4}
       E_0(d)
      \\
      &&
      + \left(p^{-2k+2} \left(\frac{-m/d^2}{p}\right) + \delta\!\left(p {\not|} \frac{f}{d}\right) p^{-3k+4} \right)
        E_1(d) \\
      &&
      + \left(p^{-k+1} - p^{-2k+2} \left(\frac{-m/d^2}{p}\right)\right)
        E_2(d)
      \Bigg\}\\
    &=&
      p^{-3k+4} \sum_{\begin{smallmatrix} d|f \\ \frac{f}{d} \equiv 0 \ (p)\end{smallmatrix}}
                   g_k\!\left(\frac{m}{d^2}\right) E_0(d)
      + p^{-2k+2} \left(\frac{-D_0}{p}\right) 
      \sum_{\begin{smallmatrix} d|f \\ \frac{f}{d} {\not \equiv} 0 \ (p)\end{smallmatrix}}
                   g_k\!\left(\frac{m}{d^2}\right) E_1(d)\\
    &&
      + p^{-3k+4}  \sum_{\begin{smallmatrix} d|f \\ \frac{f}{d} {\not \equiv} 0 \ (p)\end{smallmatrix}}
                   g_k\!\left(\frac{m}{d^2}\right) E_1(d)\\
    &&
      + p^{-3k+4}  \sum_{d|f} g_k\!\left(\frac{m}{d^2}\right)
                   \left(p^{2k-3} - p^{k-2} \left(\frac{-m/d^2}{p}\right)\right)
                   E_2(d)\\
    &=&
      p^{-3k+4} \delta(p|f) \sum_{d|\frac{f}{p}} \left(p^{2k-3} - p^{k-2} \left(\frac{-m/(dp)^2}{p} \right) \right)
                g_k\!\left(\frac{m}{d^2p^2}\right) E_0(d)\\
    &&
      + p^{-2k+2} \left(\frac{D_0}{p}\right)  \sum_{\begin{smallmatrix} d|f \\ \frac{f}{d} {\not \equiv} 0 \ (p)\end{smallmatrix}}
                   g_k\!\left(\frac{m}{d^2}\right) E_1(d)
      + p^{-3k+4}  \sum_{\begin{smallmatrix} d|f \\ \frac{f}{d} {\not \equiv} 0 \ (p)\end{smallmatrix}}
                   g_k\!\left(\frac{mp^2}{(dp)^2}\right) E_2(pd)\\
    &&
      + p^{-3k+4}  \sum_{d|f} g_k\!\left(\frac{mp^2}{d^2}\right)
                   E_2(d)\\
    &=&
      p^{-k+1} \delta(p|f) \sum_{d|\frac{f}{p}} g_k\!\left(\frac{m}{d^2p^2}\right) E_0(d)
      - p^{-2k+2} \delta(p|f) \left( \frac{D_0}{p} \right)
                   \sum_{\begin{smallmatrix} d|f \\ \frac{f}{d} {\not \equiv} 0 \ (p)\end{smallmatrix}}
                   g_k\!\left(\frac{m}{d^2}\right) E_0\!\left(\frac{d}{p}\right) \\
    &&
      + p^{-2k+2}  \left( \frac{D_0}{p} \right)
                   \sum_{\begin{smallmatrix} d|f \\ \frac{f}{d} {\not \equiv} 0 \ (p)\end{smallmatrix}}
                   g_k\!\left(\frac{m}{d^2}\right) E_1(d)
      + p^{-3k+4}  \sum_{\begin{smallmatrix} d|fp \\ \frac{fp}{d} {\not \equiv} 0 \ (p)\end{smallmatrix}}
                   g_k\!\left(\frac{mp^2}{d^2}\right) E_2(d)\\
    &&
      + p^{-3k+4}  \sum_{\begin{smallmatrix} d|fp \\ \frac{fp}{d} \equiv 0 \ (p)\end{smallmatrix}}
                   g_k\!\left(\frac{mp^2}{d^2}\right) E_2(d)\\
    &=&
      p^{-k+1} \delta(p|f) \sum_{d|\frac{f}{p}} g_k\!\left(\frac{m}{d^2p^2}\right) E_0(d)
      + p^{-2k+2} \left( \frac{D_0 f^2}{p} \right)
                   \sum_{d|f}
                   g_k\!\left(\frac{m}{d^2}\right) E_1(d)\\
    &&
      + p^{-3k+4}  \sum_{d|fp}
                   g_k\!\left(\frac{mp^2}{d^2}\right) E_2(d).
\end{eqnarray*}
Hence
\begin{eqnarray*}
\begin{aligned}
     & \sum_{d|f} g_k\!\left(\frac{m}{d^2}\right)
       \left(
            E_0(d),
            E_1(d),
            E_2(d)
       \right)
      \begin{pmatrix}
         a_{0,\frac{m}{d^2},p,k} \\
         a_{1,\frac{m}{d^2},p,k} \\
         a_{2,\frac{m}{d^2},p,k}
      \end{pmatrix}\\
    &=
      \left(
      \sum_{d|\frac{f}{p}} g_k\!\left(\frac{m}{d^2p^2}\right) E_0(d),\
      \sum_{d|f} g_k\!\left(\frac{m}{d^2}\right) E_1(d),\
      \sum_{d|fp} g_k\!\left(\frac{mp^2}{d^2}\right) E_2(d)
      \right)\\
    &\qquad \times
     \begin{pmatrix} p^{-k+1} \\ p^{-2k+2} \left( \frac{-m}{p}\right) \\ p^{-3k+4} \end{pmatrix} \\
    &=
      \left(e_{k,\mathcal{M}\left[X^{-1} \smat{p}{0}{0}{1}^{-1} \right]}^{(n)}|U_{\smat{p}{0}{0}{1} X\smat{p}{0}{0}{1}},\
      e_{k,\mathcal{M}}^{(n)}|U_{\smat{p}{0}{0}{1}},\
      e_{k,\mathcal{M}\left[ \smat{p}{0}{0}{1} \right]}^{(n)} \right)
     \begin{pmatrix} p^{-k+1} \\ p^{-2k+2} \left( \frac{-m}{p}\right) \\ p^{-3k+4} \end{pmatrix}.
\end{aligned}
\end{eqnarray*}
Therefore
\begin{eqnarray*}
     &&  \!\!\!\!\!\!\!\!\!\!
     e_{k,\mathcal{M}}^{(n)}|(V_{0,n}(p^2), ..., V_{n,0}(p^2))\\
     &=&
        p^{(-k+\frac12 n + \frac32)(n-1)}
     \sum_{d|f} g_k\!\left(\frac{m}{d^2}\right)
       \left(
            E_0(d),\
            E_1(d),\
            E_2(d)
       \right)
      \begin{pmatrix}
        0 & a_{0,\frac{m}{d^2},p,k} \\
        p^{-2k+2} & a_{1,\frac{m}{d^2},p,k} \\
        0 & a_{2,\frac{m}{d^2},p,k}
      \end{pmatrix}\\
     &&
     \times
        A_{2,n+1}^{p}(p^{k-\frac{n+2}{2}-\frac12})\\
     &=&
     p^{(-k+\frac12 n + \frac32)(n-1)}
      \left(e_{k,\mathcal{M}\left[X^{-1} \smat{p}{0}{0}{1}^{-1} \right]}^{(n)}|U_{\smat{p}{0}{0}{1} X\smat{p}{0}{0}{1}},\
      e_{k,\mathcal{M}}^{(n)}|U_{\smat{p}{0}{0}{1}},\
      e_{k,\mathcal{M}\left[ \smat{p}{0}{0}{1} \right]}^{(n)} \right)\\
     &&
     \times
     \begin{pmatrix}
       0 & p^{-k+1} \\
       p^{-2k+2} & p^{-2k+2} \left( \frac{-m}{p}\right) \\
       0 & p^{-3k+4} \end{pmatrix}
     A_{2,n+1}^{p}(p^{k-\frac{n+2}{2}-\frac12}).
\end{eqnarray*}
\end{proof}

Proposition~\ref{prop:e_k_M_V} is a generalized Maass relation
for matrix index of integral-weight.
The generalized Maass relation for integer index of half-integral weight is as follows.

\begin{theorem}\label{thm:maass_e_half}
 Let $e_{k-\frac12,m}^{(n)}$ be the $m$-th Fourier-Jacobi coefficient of generalized Cohen-Eisenstein series
 $H_{k-\frac12}^{(n+1)}$. $($See~$($\ref{eq:df_fourier_jacobi_cohen_eisenstein}$)$$)$.
 Then we obtain
  \begin{eqnarray*}
  &&  \!\!\!\!\!\!\!\!\!\!
  e^{(n)}_{k-\frac12,m}|(\tilde{V}_{0,n}(p^2), \tilde{V}_{1,n-1}(p^2),...,\tilde{V}_{n,0}(p^2))\\
  &=&
  p^{k(n-1)-\frac12(n^2+5n-5)}
  \left(e^{(n)}_{k-\frac12,\frac{m}{p^2}}|U_{p^2},\, e^{(n)}_{k-\frac12,m}|U_p,\, e^{(n)}_{k-\frac12,mp^2} \right)
  \begin{pmatrix}
          0 & p^{2k-3} \\
          p^{k-2} & p^{k-2} \left(\frac{-m}{p}\right) \\
          0 & 1 \end{pmatrix}\\
  &&
  \times
  A_{2,n+1}^{p}\!\left( p^{k-\frac{n+2}{2}-\frac12} \right)
  \mbox{diag}(1,p^{1/2},\cdots, p^{n/2}).
\end{eqnarray*}
Here $A_{2,n+1}^p\!\left( p^{k-\frac{n+2}{2}-\frac12} \right)$
is a $2 \times (n+1)$ matrix which is introduced
in the beginning of \S\ref{s:maass_relation_cohen_eisen}
and the both side of the above identity are vectors of forms.
\end{theorem}
\begin{proof}
 From Lemma~\ref{lemma:iota} and from the definitions of $e_{k,\mathcal{M}}^{(n)}$
 and $e_{k-\frac12,m}^{(n)}$, we have
 \[
  \iota_{\mathcal{M}}(e_{k,\mathcal{M}}^{(n)}) = e_{k-\frac12,m}^{(n)}.
 \]
 By using Proposition~\ref{prop:iota_hecke} we have
  \begin{eqnarray*}
    e_{k-\frac12,m}^{(n)}|\tilde{V}_{\alpha,n-\alpha}(p^2)
    &=&
    \iota_{\mathcal{M}}(e_{k,\mathcal{M}}^{(n)})|\tilde{V}_{\alpha,n-\alpha}(p^2)\\
    &=&
    p^{k(2n+1)-n(n+\frac72)+\frac12\alpha}
    \iota_{\mathcal{M\left[\smat{p}{0}{0}{1}\right]}}(e_{k,\mathcal{M}}^{(n)}|V_{\alpha,n-\alpha}(p^2))
  \end{eqnarray*}
 From Proposition~\ref{prop:iota_U} we also have identities
  \begin{eqnarray*}
      e_{k-\frac12,\frac{m}{p^2}}^{(n)}|U_{p^2}
    &=&
    \iota_{\mathcal{M}\left[\smat{p}{0}{0}{1}\right]}
      \left(
       e_{k,\mathcal{M}\left[X^{-1} \smat{p}{0}{0}{1}^{-1} \right]}^{(n)}|U_{\smat{p}{0}{0}{1} X\smat{p}{0}{0}{1}}
      \right),
    \\
      e_{k-\frac12,m}^{(n)}|U_p
    &=&
    \iota_{\mathcal{M}\left[\smat{p}{0}{0}{1}\right]}
      \left(
       e_{k,\mathcal{M}}^{(n)}|U_{\smat{p}{0}{0}{1}}
      \right),
  \end{eqnarray*}
  and
  \begin{eqnarray*}
       e_{k-\frac12,mp^2}^{(n)}
   &=&
   \iota_{\mathcal{M}\left[\smat{p}{0}{0}{1}\right]}\left(e_{k,\mathcal{M}\left[\smat{p}{0}{0}{1}\right]}^{(n)}\right).
  \end{eqnarray*}
  Because the map $\iota_{\mathcal{M}\left[\smat{p}{0}{0}{1}\right]}$ is a linear map,
  this theorem follows from
  Proposition~\ref{prop:e_k_M_V}
  and from the above identities.
\end{proof}

\section{Maass relation for Siegel cusp forms of half-integral weight
and lifts}\label{s:maass_relation_siegel_cusp}

In this section we shall prove Theorem~\ref{th:main}.

We denote by $S_k(\Gamma_n) \subset M_k(\Gamma_n)$,
$S_{k-\frac12}^{+}(\Gamma_0^{(n)}(4)) \subset M_{k-\frac12}^{+}(\Gamma_0^{(n)}(4))$,
$J_{k,1}^{(n)\, cusp} \subset J_{k,1}^{(n)}$
and
$J_{k-\frac12, m}^{(n)*\, cusp} \subset J_{k-\frac12, m}^{(n)*}$
the spaces of the cusp forms, respectively
(cf \S\ref{ss:fj_expansion_half}, \S\ref{ss:CE_J*},
\S\ref{ss:jacobi_forms_of_matrix_index} and \S\ref{ss:def_jacobi_half_weight}).

Let $k$ be an even integer
and $f \in S_{2(k-n)}(\Gamma_1)$ be an eigenform for all Hecke operators.
Let
\begin{eqnarray*}
  h(\tau) 
  &=&
  \sum_{\begin{smallmatrix} N \in \Z \\ N \equiv 0,3 \!\! \mod 4, \ N > 0
                                                 \end{smallmatrix}} 
                  c(N) \, e(N\tau)
   \in S_{k-n+\frac12}^{+}(\Gamma_0^{(1)}(4))
\end{eqnarray*}
 be a Hecke eigenform
which corresponds to $f$ by the Shimura correspondence.
We assume that the Fourier coefficient of $f$ at $e^{2\pi i z}$ is $1$.

Let
\begin{eqnarray*}
  I_{2n}(h)(\tau) &=& \sum_{T \in Sym_{2n}^+} A(T)\, e(T\tau) \in S_k(\Gamma_{2n})
\end{eqnarray*}
 be the Ikeda lift of $h$.
 For $T \in \mbox{Sym}_{2n}^+$ the $T$-th Fourier coefficient $A(T)$ of $I_{2n}(h)$ is
\begin{eqnarray*}
  A(T) &=&
  c(|D_T|)\, f_T^{k-n-\frac12}\, \prod_{\begin{smallmatrix} q\, : \, prime \\ q|f_T \end{smallmatrix}} \tilde{F}_q(T,\alpha_q),
\end{eqnarray*}
 where $D_T$ is the fundamental discriminant and $f_T$ is the natural number
 which satisfy $\det(2 T) = |D_T|\, f_T^2$,
 and where $\left\{\alpha_q^{\pm}\right\}$ is the set of Satake parameters of $f$
 in the sense of Ikeda~\cite{Ik},
 it means that
 $(\alpha_q + \alpha_q^{-1})q^{k-n-1/2}$ is the $q$-th
 Fourier coefficient of $f$.
 Here $\tilde{F}_q(T,X) \in \C[X+X^{-1}]$ is a certain Laurent polynomial.
 For the detail of the definition of $\tilde{F}_q(T,X)$ the reader is referred to~\cite[page 642]{Ik}.

 Let
 \begin{eqnarray*}
   I_{2n}(h) \left(\begin{pmatrix} \tau & z \\ {^t z} & \omega \end{pmatrix}\right)
   &=&
   \sum_{a=1}^{\infty} \psi_a(\tau,z) \, e(a \omega)
 \end{eqnarray*}
 be the Fourier-Jacobi expansion of $I_{2n}(h)$, where $\tau \in \H_{2n-1}$, $\omega \in \H_1$
 and $z \in \C^{(2n-1,1)}$.
 Note that $\psi_a \in J_{k,a}^{(2n-1)\, cusp}$ is a Jacobi cusp form
 of weight $k$ of index $a$ of
 degree $2n-1$.

By the Eichler-Zagier-Ibukiyama correspondence (see \S\ref{ss:fj_expansion_half}) there exists
a Siegel cusp form $F \in S_{k-\frac12}^{+}(\Gamma_0^{(2n-1)}(4))$
which corresponds to $\psi_1 \in J_{k,1}^{(2n-1)\, cusp}$.

For $g \in S_{k-1/2}^+(\Gamma_0^{(1)}(4))$
we put
  \begin{eqnarray*}
      \mathcal{F}_{h,g}(\tau)
      &:=&
      \frac16
      \int_{\Gamma_1\backslash \H_1}
        F\!\left(
              \begin{pmatrix}
                 \tau & 0\\
                 0 & \omega
              \end{pmatrix}
            \right)
        \overline{g(\omega)}\,
        \mbox{Im}(\omega)^{k-\frac52} \, d\omega
  \end{eqnarray*}
  for $\tau \in \H_{2n-2}$.
It is not difficult to show that the form $\mathcal{F}_{h,g}$
belongs to $S_{k-\frac12}^{+}(\Gamma_0^{(2n-2)}(4))$.
The above construction of $\mathcal{F}_{h,g}$ was suggested by T.Ikeda to the author.

To show properties of $\mathcal{F}_{h,g}$ we consider the Fourier-Jacobi
expansion of $F$.
Let
\begin{eqnarray*}
  F\left( \begin{pmatrix} \tau & z \\ {^t z} & \omega \end{pmatrix}\right)
  &=&
  \sum_{\begin{smallmatrix} m \in \Z \\ m \equiv 0,3 \mod 4\end{smallmatrix}}
  \phi_m(\tau,z)\, e(m \omega)
\end{eqnarray*}
be the Fourier-Jacobi expansion of $F$,
where $\tau \in \H_{2n-2}$, $\omega \in \H_1$ and $z \in \C^{(2n-2,1)}$.
Note that $\phi_m \in J_{k-\frac12,m}^{(2n-2)*\, cusp}$ is a Jacobi cusp form
of weight $k-\frac12$ of index $m$ and of degree $2n-2$.

Let
\begin{eqnarray*}
  \phi_m(\tau,z)
  &=&
  \sum_{\begin{smallmatrix}
                        M \in Sym_{2n-2}^+,\ S\in \Z^{(2n-2,1)} \\
                        4Mm - S{^t S} > 0 \end{smallmatrix}}
  C_m(M,S) \, e(M \tau + S {^t z} )
\end{eqnarray*}
be the Fourier expansion of $\phi_m$,
where $\tau \in \H_{2n-2}$ and $z \in \C^{(2n-2,1)}$.
We have the diagram

\[
  \xymatrix{
	&
      I_{2n}(h) \in S_{k}(\Gamma_{2n})  \ar[d]^{\mbox{1st F-J}}
    &
   \\
	& \qquad
      \psi_1 \in J_{k,1}^{(2n-1)\, cusp}
                                       \ar[r]_{\qquad \mbox{E-Z-I \quad \quad}}
    &
      F \ar[d]^{\mbox{F-J}}  \in S_{k-\frac12}^{+}(\Gamma_0^{(2n-1)}(4))\\
	&
      \quad
    &
      \displaystyle{
       \left\{\phi_m\right\}_{m}
      }
      \in
          \!\!\!\!\!
	  \displaystyle{
          \bigotimes_{m \equiv 0,3\!\!\!\!\! \mod 4}
          \!\!\!\!\!
          J_{k-\frac12,m}^{(2n-2)* \, cusp}}
   \\	
       h \in S_{k-n+\frac12}^+(\Gamma_0^{(1)}(4))   \ar[r] \ar[ruuu]^{\mbox{Ikeda lift}}
    &
       f \in S_{2(k-n)}(\Gamma_1) .  \ar[l]
    &
	}
\]

\begin{lemma}\label{lem:coe_cm}
  The $(M,S)$-th Fourier coefficient $C_m(M,S)$ of $\phi_m$ is
  \begin{eqnarray*}
    C_m(M,S)
    &=&
    c(|D_T|)\, f_T^{k-n-\frac12}\, \prod_{q | f_T} \tilde{F}_q(T,\alpha_q),
  \end{eqnarray*}
  where $T \in \mbox{Sym}_{2n}^+$ is the matrix which satisfies
  \begin{eqnarray*}
    T &=&
    \begin{pmatrix}
      N & \frac12 R \\
      \frac12 {^t R} & 1
    \end{pmatrix}
  \end{eqnarray*}
  and $N \in \mbox{Sym}_{2n-1}^+$ and $R \in \Z^{(2n-1,1)}$ are the matrices
  which satisfy
  \begin{eqnarray*}
    4N - R {^t R} &=&
    \begin{pmatrix}
      M & \frac12 S \\
      \frac12 {^t S} & m
    \end{pmatrix}.
  \end{eqnarray*}
\end{lemma}
\begin{proof}
  The Fourier expansion of $\psi_1$ is
  \begin{eqnarray*}
    \psi_1(\tau,z) &=&
    \sum_{
     \begin{smallmatrix}
       N \in Sym_{2n-1}^+,\, R \in \Z^{(2n-1,1)} \\
       4N - R {^t R} > 0
     \end{smallmatrix}}
     A\left(\begin{pmatrix} N & \frac12 R \\ \frac12 {^t R} & 1 \end{pmatrix} \right)
     \,
     e(N\tau + R {^t z}).
  \end{eqnarray*}
  And the Fourier expansion of $F$ is
  \begin{eqnarray*}
    F(\tau)
    &=&
    \sum_{4N - R {^t R} > 0}
     A\left(\begin{pmatrix} N & \frac12 R \\ \frac12 {^t R} & 1 \end{pmatrix} \right)
     \,
     e((4N - R {^t R}) \tau).     
  \end{eqnarray*}
  Since $\phi_m$ is the $m$-th Fourier-Jacobi coefficient of $F$,
  the $(M,S)$-th Fourier coefficient $C_m(M,S)$ of $\phi_m$ is $A(T)$, where
  $T$ is in the statement of this lemma.
\end{proof}

The following theorem is a generalization of the Maass relation
for Siegel cusp forms of half-integral weight.
\begin{theorem}\label{th:maass_relation_half_cusp}
 Let $\phi_m$ be the $m$-th Fourier-Jacobi coefficient of $F$ as above.
 Then we obtain
  \begin{eqnarray*}
  &&  \!\!\!\!\!\!\!\!\!\!
  \phi_m|(\tilde{V}_{0,2n-2}(p^2), \tilde{V}_{1,2n-3}(p^2),...,\tilde{V}_{2n-2,0}(p^2))\\
  &=&
  p^{k(2n-3)-2 n^2 - n + \frac{11}{2}}
  \left(\phi_{\frac{m}{p^2}}|U_{p^2},\, \phi_m|U_p,\, \phi_{mp^2} \right)
  \begin{pmatrix}
          0 & p^{2k-3} \\
          p^{k-2} & p^{k-2} \left(\frac{-m}{p}\right) \\
          0 & 1 \end{pmatrix}\\
  &&
  \times
  A_{2,2n-1}^{p}\!\left(\alpha_p\right)
  \,
  diag(1,p^{\frac12},p,...,p^{n-1})
\end{eqnarray*}
for any prime $p$, where the $2\times (n+1)$-matrix $A_{2,2n-1}^{p}\!\left(\alpha_p\right)$
is introduced in the beginning of \S\ref{s:maass_relation_cohen_eisen}.
\end{theorem}
 \begin{proof}
 Let
 \begin{eqnarray*}
   \left(\phi_m|\tilde{V}_{\alpha,2n-2-\alpha}(p^2)\right)\!(\tau,z)
   &=&
   \sum_{\begin{smallmatrix}M \in Sym_{2n-2}^+,\ S\in \Z^{(2n-2,1)} \\ 4Mmp^2 - S{^t S} > 0 \end{smallmatrix}}
    C_m(\alpha;M,S)\ e(M\tau + S {^t z}).
 \end{eqnarray*}
 be the Fourier expansion of $\phi_m|\tilde{V}_{\alpha,2n-2-\alpha}(p^2)$.
 We first calculate the Fourier coefficients $C_m(\alpha;M,S)$.
 There exist matrices $N \in \Z^{(2n-1,2n-1)}$ and $R \in \Z^{(2n-1,1)}$ which satisfy
 $  4N - R {^t R}
  =
    \begin{pmatrix}
      M & \frac12 S \\
      \frac12 {^t S} & mp^2
    \end{pmatrix}$.
 We put
 $
  T
  =
  \begin{pmatrix}
    N & \frac12 R \\
    \frac12 {^t R} & 1
  \end{pmatrix}
 $.
 Due to Proposition~\ref{prop:iota_hecke} and
 due to the definition of $\tilde{V}_{\alpha,2n-2-\alpha}(4)$
 in \S\ref{ss:hecke_p2}, we can take $N$ and $R$ which satisfy 
 \begin{eqnarray*}
   T
   &=&
   \begin{pmatrix}
     N' & \frac12 R' \\
     \frac12 {^t R'} & \mathcal{M}\left[\smat{p}{0}{0}{1}\right]
   \end{pmatrix}
 \end{eqnarray*}
 with matrices $N' \in \Z^{(2n-2,2n-2)}$ and $R' \in \Z^{(2n-2,2)}$.

 We assume that $p$ is an odd prime.
 Let
 $\left\{\left(
     \begin{pmatrix}
       p^2 {^t D_i}^{-1} & B_i \\
       0 & D_i \end{pmatrix},
      \gamma_i\ p^{-n+1} \left(\det D_i\right)^{\frac12} \right) \right\}_i$
 be a complete set of the representatives of
 $\Gamma_0^{(n)}(4)^* \backslash \Gamma_0^{(n)}(4)^* Y \Gamma_0^{(n)}(4)^*$,
 where $Y$ is
 $Y = (\mbox{diag}(1_{\alpha},p 1_{2n-2-\alpha},
                p^2 1_{\alpha}, p 1_{2n-2-\alpha}),p^{\alpha/2})$
 and $\gamma_i$ is a root of unity
 (see~\cite[Prop.7.1]{Zhu:hecke} or ~\cite[Lemma 3.2]{Zhu:euler}
 for the detail of these representatives).
 Then by a straightforward calculation and from Lemma~\ref{lem:coe_cm} we obtain
 \begin{eqnarray} \label{eq:cm_alpha}
   C_m(\alpha;M,S)
   &=&
   p^{k(2n-3)+2n^2-\frac12-4n(n-1)} 
   c(|D_T|) f_T^{k-n-\frac12}\\
   \notag
   &&
   \times
   \sum_i \gamma_i \left(\det D_i\right)^{-n} e\!\left(\frac{1}{p^2} N {^t D_i} B_i\right)
   \prod_{q|f_{T\left[Q_i\right]}} \tilde{F}_q\left(T\left[Q_i\right],\alpha_q \right),
 \end{eqnarray}
 where
 $D_T$ is the fundamental discriminant and $f_T > 0$ is the natural number
 which satisfy 
 $\det(2T)= |D_T| {f_T}^2$,
 and where $Q_i = \mbox{diag}(p^{-1} {^t D_i}, p^{-1},1) \in \Qq^{(2n,2n)}$.
 The number $c(|D_T|)$ is the $|D_T|$-th Fourier coefficient of $h$.
 
 By virtue of the definition of $\tilde{V}_{\alpha,2n-2-\alpha}(4)$
 the identity~(\ref{eq:cm_alpha}) also holds for $p=2$.
 
 For any prime $p$ the $(M,S)$-th Fourier coefficients of
 $\phi_{\frac{m}{p^2}}|U_{p^2}$,
 $\phi_m|U_p$
 and $\phi_{mp^2}$ are
 $C_\frac{m}{p^2}(M,p^{-2} S)$,
 $C_m(M,p^{-1} S)$
 and
 $C_{mp^2}(M,S)$, respectively.
 These are
 \begin{eqnarray*}
     C_\frac{m}{p^2}(M,p^{-2} S)
   &=&
     p^{-2(k-n-\frac12)} c(|D_T|) f_T^{k-n-\frac12}
     \prod_{q|f_Tp^{-2}} \tilde{F}_q\left( T_0, \alpha_q \right),
   \\
     C_m(M,p^{-1} S)
   &=&
     p^{-(k-n-\frac12)} c(|D_T|) f_T^{k-n-\frac12}
     \prod_{q|f_Tp^{-1}} \tilde{F}_q\left( T_1, \alpha_q \right)
 \end{eqnarray*}
 and
 \begin{eqnarray*}
     C_{mp^2}(M,S)
   &=&
     c(|D_T|) f_T^{k-n-\frac12}
     \prod_{q|f_T} \tilde{F}_q\left( T, \alpha_q \right),
 \end{eqnarray*}
 respectively,
 where we put
 $T_0 = T \left[\left(\begin{smallmatrix}
                             1_{2n-2} & 0 & 0 \\ 
                             0 & p^{-2} & 0 \\
                             0 & 0 & 1
                                \end{smallmatrix}\right)\right]$
 and
 $T_1 = T \left[\left(\begin{smallmatrix}
                             1_{2n-2} & 0 & 0 \\ 
                             0 & p^{-1} & 0 \\
                             0 & 0 & 1
                                \end{smallmatrix}\right)\right]$.
 Note that if $p^{-1} S \in \Z^{(2n-2,1)}$, then $f_T$ is  divisible by $p$,
 and if $p^{-2} S \in \Z^{(2n-2,1)}$, then $f_T$ is divisible by $p^2$.

 Note that the Fourier coefficients of $e_{k-\frac12,m}^{(2n-2)}|\tilde{V}_{\alpha,2n-2-\alpha}(p^2)$,
 $e_{k-\frac12,\frac{m}{p^2}}^{(2n-2)}|U_{p^2}$,
 $e_{k-\frac12,m}^{(2n-2)}|U_p$ and $e_{k-\frac12,mp^2}^{(2n-2)}$ have the
 same form of the above expressions by substituting
 $\alpha_q = q^{k-n-\frac12}$ and
 by replacing $c(|D_T|)$ by $c\, h_{k-n+\frac12}(|D_T|)$,
 where $h_{k-n+\frac12}(|D_T|)$ is the $|D_T|$-th Fourier coefficient
 of the Cohen-Eisenstein series $\mathcal{H}_{k-n+\frac12}^{(1)}$ of weight $k-n+\frac12$,
and where
$c := c_{k,2n} = 2^n \zeta(1-k)^{-1} \prod_{i=1}^n \zeta(1+2i-2k)^{-1}$.
 Hence from Theorem~\ref{thm:maass_e_half} and from a well-known argument,
 we can conclude that the polynomial
  \begin{eqnarray*}
    p^{k(2n-3)+2n^2-\frac12-4n(n-1)} 
   \sum_i \gamma_i \left(\det D_i\right)^{-n} e\!\left(\frac{1}{p^2} N {^t D_i} B_i\right)
   \prod_{q|f_{T\left[Q_i\right]}} \tilde{F}_q\left(T\left[Q_i\right],X \right)
 \end{eqnarray*}
 of $X$ coincides the $(\alpha+1)$-th component of the vector
 \begin{eqnarray*}
   &&
      p^{k(2n-3)-2 n^2 - n + \frac{11}{2}}\\
   &&
   \times
   \left( 
       p^{-2(k-n-\frac12)} \prod_{q|f_Tp^{-2}} \tilde{F}_q\left( T_0, X \right),\
       p^{-(k-n-\frac12)} \prod_{q|f_Tp^{-1}} \tilde{F}_q\left( T_1, X \right),\
       \prod_{q|f_T} \tilde{F}_q\left( T, X \right)
   \right)
  \\
  &&
   \times
     \begin{pmatrix}
           0 & p^{2k-3} \\
           p^{k-2} & p^{k-2} \left(\frac{-m}{p}\right) \\
           0 & 1 \end{pmatrix}
    A_{2,2n-1}^{p}\!\left(X\right)\,
   \mbox{diag}(1,p^{1/2},...,p^{(2n-2)/2}).
 \end{eqnarray*}

 Therefore $C_m(\alpha;M,S)$ coincides the $(\alpha+1)$-th component of the vector
 \begin{eqnarray*}
  &&
     p^{k(2n-3)-2 n^2 - n + \frac{11}{2}}
    \left(C_{\frac{m}{p^2}}(M,p^{-2}S),\, C_m(M,p^{-1}S),\, C_{mp^2}(M,S) \right)\\
  &&
    \times
    \begin{pmatrix}
            0 & p^{2k-3} \\
            p^{k-2} & p^{k-2} \left(\frac{-m}{p}\right) \\
            0 & 1 \end{pmatrix}
    A_{2,2n-1}^{p}\!\left(\alpha_p\right)\,
    \mbox{diag}(1,p^{1/2},...,p^{(2n-2)/2}).
 \end{eqnarray*}
 Thus we conclude this theorem.
 \end{proof}

Let $\tilde{T}_{\alpha,2n-2-\alpha}(p^2)$ be the Hecke operator
introduced in \S\ref{ss:hecke_op_siegel_half}
and let $L(s,\mathcal{F})$ be the $L$-function
for a Hecke eigenform $\mathcal{F} \in S_{k-\frac12}^+(\Gamma_0^{(m)}(4))$
introduced in \S\ref{ss:L_function_siegel_half}.
\begin{theorem}\label{th:main}
  Let $k$ be an even integer and $n$ be a non-negative integer.
  Let $h \in S_{k-n+\frac12}^+(\Gamma_0^{(1)}(4))$
  and $g \in S_{k-\frac12}^+(\Gamma_0^{(1)}(4))$
  be eigenforms for all Hecke operators.
  Then there exists a $\mathcal{F}_{h,g} \in S_{k-\frac12}^+(\Gamma_0^{(2n-2)})$.
  Under the assumption that $\mathcal{F}_{h,g}$ is not identically zero, then $\mathcal{F}_{h,g}$
  is an eigenform with the $L$-function which satisfies
  \begin{eqnarray*}
    L(s,\mathcal{F}_{h,g})
    &=&
    L(s,g) \prod_{i=1}^{2n-3} L(s-i,h).
  \end{eqnarray*}
\end{theorem}
  \begin{proof}
  The construction of $\mathcal{F}_{h,g}$ is stated in the above:
  \begin{eqnarray*}
      \mathcal{F}_{h,g}(\tau)
      &=&
      \frac16
      \int_{\Gamma_1\backslash \H_1}
        F\!\left(
              \begin{pmatrix}
                 \tau & 0\\
                 0 & \omega
              \end{pmatrix}
            \right)
        \overline{g(\omega)}\,
        \mbox{Im}(\omega)^{k-\frac52} \, d\omega,
  \end{eqnarray*}
  where $F \in S_{k-\frac12}^{+}(\Gamma_0^{2n-1}(4))$ is constructed from $h$.
  By the definition of $\tilde{V}_{\alpha,2n-2-\alpha}(p^2)$ and due to
  Theorem~\ref{th:maass_relation_half_cusp} we have
  \begin{eqnarray*}
    &&  \!\!\!\!\!\!\!\!\!\!
      \phi_m(\tau,0)|\left(\tilde{T}_{0,2n-2}(p^2),...,\tilde{T}_{2n-2,0}(p^2)\right)\\
    &=&
      \left(\phi_m|\left(\tilde{V}_{0,2n-2}(p^2),...,\tilde{V}_{2n-2,0}(p^2) \right)\right)(\tau,0)\\
    &=&
      p^{k(2n-3)-2 n^2 - n + \frac{11}{2}}
      \left(\left(\phi_{\frac{m}{p^2}}|U_{p^2}\right)(\tau,0),\
             \left(\phi_m|U_p\right)(\tau,0),\
             \phi_{mp^2}(\tau,0) \right)\\
    &&
      \times
      \begin{pmatrix}
              0 & p^{2k-3} \\
              p^{k-2} & p^{k-2} \left(\frac{-m}{p}\right) \\
              0 & 1 \end{pmatrix}
      A_{2,2n-1}^{p}\!\left(\alpha_p\right)
        \mbox{diag}(1,p^{1/2}, \cdots , p^{(2n-2)/2})
    \\
    &=&
        p^{k(2n-3)-2 n^2 - n + \frac{11}{2}}
    \left(\phi_{\frac{m}{p^2}}(\tau,0),\
         \phi_m(\tau,0),\
         \phi_{mp^2}(\tau,0) \right)\\
    &&
    \times
    \begin{pmatrix}
          0 & p^{2k-3} \\
          p^{k-2} & p^{k-2} \left(\frac{-m}{p}\right) \\
          0 & 1 \end{pmatrix}
    A_{2,2n-1}^{p}\!\left(\alpha_p\right)
     \mbox{diag}(1,p^{1/2}, \cdots , p^{(2n-2)/2}).
  \end{eqnarray*}
  We remark
  \begin{eqnarray*}
    &&  \!\!\!\!\!\!\!\!\!\!
      \sum_{\begin{smallmatrix} m \\ m \equiv 0, 3 \!\!\! \mod 4 \end{smallmatrix}}
      \left(
        p^{2k-3} \phi_{\frac{m}{p^2}}(\tau,0) +
        p^{k-2}\left(\frac{-m}{p}\right) \phi_m(\tau,0) +
        \phi_{mp^2}(\tau,0)
      \right)
      e(m\omega) \\
    &=&
      F\left(\begin{pmatrix} \tau & 0 \\ 0 & \omega \end{pmatrix}\right)
      |\tilde{T}_{1,0}(p^2).
  \end{eqnarray*}  
  Thus
  \begin{eqnarray*}
    &&  \!\!\!\!\!\!\!\!\!\!
      F\left(\begin{pmatrix} \tau&0\\0&\omega \end{pmatrix}\right)
      \bigg| \left(\tilde{T}_{0,2n-2}(p^2),...,\tilde{T}_{2n-2,0}(p^2) \right)\\
    &=&
        \sum_{\begin{smallmatrix} m \\ m \equiv 0, 3 \!\!\! \mod 4 \end{smallmatrix}}
        \left\{
          \phi_m(\tau,0)\bigg| \left(\tilde{T}_{0,2n-2}(p^2),...,\tilde{T}_{2n-2,0}(p^2) \right)
        \right\} e(m\omega)\\
    &=&
         p^{k(2n-3)-2 n^2 - n + \frac{11}{2}}
        \sum_{\begin{smallmatrix} m \\ m \equiv 0, 3 \!\!\! \mod 4 \end{smallmatrix}}
        \Bigg\{
            \left(\phi_{\frac{m}{p^2}}(\tau,0),\
               \phi_m(\tau,0),\
               \phi_{mp^2}(\tau,0) \right)\\
    &&
      \times
      \begin{pmatrix}
            0 & p^{2k-3} \\
            p^{k-2} & p^{k-2} \left(\frac{-m}{p}\right) \\
            0 & 1 \end{pmatrix}
      e(m\omega)
      \Bigg\}\
      A_{2,2n-1}^{p}\!\left(\alpha_p\right)
              \mbox{diag}(1,p^{1/2}, \cdots , p^{(2n-2)/2})\\
    &=&
         p^{k(2n-3)-2 n^2 - n + \frac{11}{2}}
      \left(
        F\left(\begin{pmatrix} \tau & 0 \\ 0 & \omega \end{pmatrix}\right)
        \bigg|_{\omega}
        \left(\tilde{T}_{0,1}(p^2),\ \tilde{T}_{1,0}(p^2)\right)
      \right)\\
    &&
      \times
      A_{2,2n-1}^{p}\!\left(\alpha_p\right)
              \mbox{diag}(1,p^{1/2}, \cdots , p^{(2n-2)/2}).
  \end{eqnarray*}

  Hence
  \begin{eqnarray*}
    &&  \!\!\!\!\!\!\!\!\!\!
      \mathcal{F}_{h,g}|\left(\tilde{T}_{0,2n-2}(p^2),...,\tilde{T}_{2n-2,0}(p^2) \right)\\
    &=&
      \int_{\Gamma_1 \backslash \H_1}
        \left(
          F\left(
              \begin{pmatrix}
                \tau & 0 \\
                0 & \omega
              \end{pmatrix}
            \right)
          \bigg|_{\tau}\left(\tilde{T}_{0,2n-2}(p^2),...,\tilde{T}_{2n-2,0}(p^2) \right) 
        \right)
        \overline{g(\omega)}\,
        \mbox{Im}(\omega)^{k-\frac52}\, d\omega\\
    &=&
      p^{k(2n-3)-2 n^2 - n + \frac{11}{2}}\\
    &&
      \times
      \int_{\Gamma_1 \backslash \H_1}
        \left(
          F\left(
              \begin{pmatrix}
                \tau & 0 \\
                0 & \omega
              \end{pmatrix}
            \right)
          \bigg|_{\omega}
          \left(\tilde{T}_{0,1}(p^2),\ \tilde{T}_{1,0}(p^2)\right)
        \right)
        \overline{g(\omega)} \, \mbox{Im}(\omega)^{k-\frac52}\, d\omega\\
    &&
      \times
      A_{2,2n-1}^{p}\!\left(\alpha_p\right)
      \mbox{diag}(1,p^{1/2}, \cdots, p^{(2n-2)/2}).
 \end{eqnarray*}
   Let $b(p)$ be the eigenvalue of $g$ with respect to $\tilde{T}_{1,0}(p^2)$.
 We remark that $b(p)$ is a real number.
 We have
 \begin{eqnarray} \label{eq:eigenvalue_Ffg}
 \begin{aligned}
   &
      \mathcal{F}_{h,g}|\left(\tilde{T}_{0,2n-2}(p^2),...,\tilde{T}_{2n-2,0}(p^2) \right)\\
   &=
      p^{k(2n-3)-2 n^2 - n + \frac{11}{2}}
      \mathcal{F}_{h,g}(\tau)\\
   & \qquad \times
      \left\{
        (p^{k-2}, b(p))
        A_{2,2n-1}^{p}\!\left(\alpha_p\right)
        \mbox{diag}(1,p^{1/2}, \cdots , p^{(2n-2)/2})
      \right\}.
    \end{aligned}
  \end{eqnarray}
  Therefore $\mathcal{F}_{h,g}$ is an eigenform for any
  $\tilde{T}_{\alpha,2n-2-\alpha}(p^2)$.

  Let $\{\beta_p^{\pm}\}$ be the set of complex numbers which satisfy
  \begin{eqnarray*}
    1 - b(p) z + p^{2k-3} z^2
    &=&
    (1 - \beta_p        p^{k-3/2} z)
    (1 - \beta_p^{-1} p^{k-3/2} z).
  \end{eqnarray*}
  Let
  $\{\mu_{0,p}^2,\mu_{1,p}^{\pm},...\mu_{2n-2,p}^{\pm}\}$ be the $p$-parameters
  of $\mathcal{F}_{h,g}$ (see \S\ref{ss:L_function_siegel_half}
  for the definition of $p$-parameters).
  We remark $\mu_{0,p}^2 \mu_{1,p} \cdots \mu_{2n-2,p} = p^{2(n-1)(k-n)}$.

  We now assume that $p$ is an odd prime.

  Let $\Psi_{2n-2}(K_{\alpha}^{(2n-2)}) \in R_{2n-2}$ be the Laurent polynomial
  of $\{z_i\}_{i = 0,...,2n-2}$ introduced in
  \S\ref{ss:L_function_siegel_half}.
  The explicit formula of $\Psi_{2n-2}(K_{\alpha}^{(2n-2)})$ was obtained in
  Proposition~\ref{prop:imageofpsim}.
  The eigenvalue of $\mathcal{F}_{h,g}$ for
  $\tilde{T}_{\alpha,2n-2-\alpha}(p^2)$ $(\alpha = 0,...,2n-2)$
  is obtained by substituting $z_i = \mu_i$ into
  $\Psi_{2n-2}(K_{\alpha}^{(2n-2)})$.
  We remark that the eigenvalue of $\mathcal{F}_{h,g}$ for $\tilde{T}_{0,2n-2}(p^2)$
  is $p^{(n-1)(2k-4n+1)}$.

  From the identities~(\ref{eq:eigenvalue_Ffg}) and~(\ref{eq:isom_p_half}),
  we obtain
  \begin{eqnarray}
  \begin{aligned}\label{eq:mu_and_alpha}
    & p^{2n^2-6n+5} (p^{-1/2}, \mu_{1,p} + \mu_{1,p}^{-1})
    \prod_{l=2}^{2n-2} B_{l,l+1}(\mu_{l,p}) \\
    &=
    p^{2n^2-6n+5}(p^{-1/2}, \beta_p + \beta_p^{-1})
    \prod_{l=2}^{2n-2} B_{l,l+1}(p^{n-l} \alpha_p).
  \end{aligned}
  \end{eqnarray}
  Here the components of the vectors in the above identity (\ref{eq:mu_and_alpha})
  are eigenvalues of $\mathcal{F}_{h,g}$ for
  \[
  \tilde{T}_{0,2n-2}(p^2)^{-1} \tilde{T}_{\alpha,2n-2-\alpha}(p^2) \qquad (\alpha = 0,...,2n-2).
  \]

  If we substitute $z_1 = \beta_p$ and $z_i = p^{n-i} \alpha_p$ $(i = 2,...,2n-2)$
  into the Laurent polynomial $(\Psi_{2n-2}(K_0^{(2n-2)}))^{-1} \Psi_{2n-2}(K_{\alpha}^{(2n-2)})$,
  then due to (\ref{eq:mu_and_alpha}) this value is
  the eigenvalue of $\mathcal{F}_{h,g}$ for
  $\tilde{T}_{0,2n-2}(p^2)^{-1} \tilde{T}_{\alpha,2n-2-\alpha}(p^2)$.
  Because $R_{2n-2}$ is generated by $\Psi_{2n-2}(K_{\alpha}^{(2n-2)})$
  $(\alpha = 0,...,2n-2)$ and $\Psi_{2n-2}(K_0^{(2n-2)})^{-1}$
  and  because of the fact that the $p$-parameters are uniquely determined up to the action of
  the Weyl group $W_2$, 
  we therefore can take the $p$-parameters
  $\left\{\mu_{1,p}^{\pm}, \cdots, \mu_{2n-2,p}^{\pm}\right\}$ of $\mathcal{F}_{h,g}$ as
  \[
  \left\{ \beta_p^{\pm},
   p^{n-2}\alpha_p^{\pm}, p^{n-3}\alpha_p^{\pm}, \cdots, p^{-n+2}\alpha_p^{\pm} \right\}.
  \]
  Hence the Euler $p$-factor $Q_{\mathcal{F}_{h,g},p}(z)$
  of $\mathcal{F}_{h,g}$
  for odd prime $p$ is
  \begin{eqnarray}\label{eq:euler_p_factor}
  \begin{aligned}
    Q_{\mathcal{F}_{h,g},p}(z)
    &=
    \prod_{i=1}^{2n-2}
    \left\{ \left( 1 - \mu_{i,p} z \right)
    \left( 1 - \mu_{i,p}^{-1} z \right)
    \right\} \\
    &=
    \left(1 - \beta_p z \right)
    \left(1 - \beta_p^{-1} z \right)
    \prod_{i=1}^{2n-3}
    \left\{
    \left(1 - \alpha_p p^{-n+i} z \right)
    \left(1 - \alpha_p^{-1} p^{-n+i}  z\right)      
    \right\}.
  \end{aligned}
  \end{eqnarray}

  We now consider the case $p=2$.
  The identity~(\ref{eq:eigenvalue_Ffg}) is also valid for $p=2$.
   Because $\tilde{\gamma}_{j,2}$ is defined in the same formula
   as in the case of odd primes, we also obtain  the identity (\ref{eq:euler_p_factor})
  for $p=2$.

  Thus we conclude
  \begin{eqnarray*}
   L(s,\mathcal{F}_{h,g})
  &=&
   \prod_p     \prod_{i=1}^{2n-2}
    \left\{ \left( 1 - \mu_{i,p} p^{-s+k-\frac32} \right)
    \left( 1 - \mu_{i,p}^{-1} p^{-s+k-\frac32} \right)
    \right\}^{-1} \\
  &=&
  L(s,g) \prod_{i=1}^{2n-3} L(s-i,h).
  \end{eqnarray*}
  \end{proof}

\section{Examples of non-vanishing}\label{s:examples_non_vanishing}

\begin{lemma}
 The form $\mathcal{F}_{h,g}$ in Theorem~\ref{th:main} is not identically zero,
 if $(n,k)$ $=$ $(2,12)$, $(2,14)$, $(2,16)$, $(2,18)$, $(3,12)$, $(3,14)$, $(3,16)$,
$(3,18)$,
$(3,20)$, $(4,10)$, $(4,12)$, $(4,14)$, $(4,16)$, $(4,18)$, $(4,20)$, $(5,14)$, $(5,16)$,
$(5,18)$, $(5,20)$, $(6,12)$, $(6,14)$, $(6,16)$, $(6,18)$ or $(6,20)$.
\end{lemma}
\begin{proof}

Let $h \in S_{k-n+\frac12}^+(\Gamma_0^{(1)}(4))$,
$F \in S_{k-\frac12}^+(\Gamma_0^{(2n-1)}(4))$
and
$\mathcal{F}_{h,g} \in S_{k-\frac12}^+(\Gamma_0^{(2n-2)}(4))$ be the same symbols in \S\ref{s:maass_relation_siegel_cusp}.
We have
\begin{eqnarray}\label{eq:Fhgg}
  F\left(\begin{pmatrix} \tau & 0 \\ 0 & \omega \end{pmatrix} \right)
  &=&
  \sum_{g} \frac{1}{\langle g, g \rangle} \mathcal{F}_{h,g}(\tau) g(\omega).
\end{eqnarray}
Here in the summation $g$ runs over a basis of $S_{k-\frac12}^+(\Gamma_0^{(1)}(4))$
which consists of Hecke eigenforms.

On the other hand, we have
\begin{eqnarray}\label{eq:Fhgg2}
  F\left(\begin{pmatrix} \tau & 0 \\ 0 & \omega \end{pmatrix} \right)
  &=&
  \sum_{M \in Sym_{2n-2}^+,\, m \in Sym_1^+} K(M,m) e(N\tau) e(m\omega),
\end{eqnarray}
where
\begin{eqnarray*}
  K(M,m)
  &=&
  \sum_{\begin{smallmatrix}S \in \Z^{(2n-2,1)} \\ 4Mm - S {^t S} > 0 \end{smallmatrix}} C_m(M,S)
\end{eqnarray*}
and where $C_m(M,S)$ is the $\begin{pmatrix} M & S \\ {^t S} & m \end{pmatrix}$-th
Fourier coefficient of $F$.
By using a computer algebraic system and Katsurada's formula for Siegel series \cite{Katsu},
we can compute the explicit values
of Fourier coefficients $C_m(M,S)$.
Hence we can also compute some Fourier coefficients $K(M,m)$.

By virtue of the identities (\ref{eq:Fhgg}) and (\ref{eq:Fhgg2}), we obtain
\begin{eqnarray*}
  K(M,m)
  &=&
  \sum_{g} \frac{1}{\langle g,g \rangle} A(M; \mathcal{F}_{h,g}) A(m;g),
\end{eqnarray*}
where $A(M;\mathcal{F}_{h,g})$ is the $M$-th Fourier coefficient of $\mathcal{F}_{h,g}$
and where $A(m;g)$ is the $m$-th Fourier coefficient of $g$.
Here Fourier coefficients $A(m;g)$ are calculated through the structure theorem
of Kohnen plus space~\cite{Ko}.
Therefore we can calculate some Fourier coefficients $A(M;\mathcal{F}_{h,g})$.

For example, if $(n,k) = (2,10)$, then $k-1/2 = 19/2$ and $k-n+1/2= 17/2$.
We have $\dim S_{19/2}^+(\Gamma_0^{(1)}(4)) = \dim S_{17/2}^+(\Gamma_0^{(1)}(4)) = 1$.
Let $g \in S_{19/2}^+(\Gamma_0^{(1)}(4))$ and $h \in S_{17/2}^+(\Gamma_0^{(1)}(4))$
be Hecke eigenforms such that the Fourier coefficients satisfy $A(3;g) = A(1;h) = 1$.
We remark that all Fourier coefficients of $g$ and $h$ are real numbers.
Let $K(M,m)$ be the number defined in (\ref{eq:Fhgg2}),
where $F \in S_{19/2}^+(\Gamma_0^{(3)}(4))$ is the Siegel modular form constructed from $h$.
Because $\dim S_{19/2}^+(\Gamma_0^{(1)}(4)) = 1$, we need to check $K(M,m) \neq 0$
for a pair $(M,m) \in Sym_{2n-2}^+ \times Sym_1^+$.
We take $M = \begin{pmatrix}  3&1\\1&3 \end{pmatrix}$ and $m = 3$, then
\begin{eqnarray*}
  K(M,m) &=&
  C_3\!\left(M,\begin{pmatrix} 2 \\ 2 \end{pmatrix} \right)
+   C_3\!\left(M,\begin{pmatrix} 2 \\ -2 \end{pmatrix} \right)
+   C_3\!\left(M,\begin{pmatrix} -2 \\ 2 \end{pmatrix} \right)
+   C_3\!\left(M,\begin{pmatrix} -2 \\ -2 \end{pmatrix} \right)\\
&=&
  -336 -168 -168 -336 \\
 & \neq &
 0.
\end{eqnarray*}
Therefore $\mathcal{F}_{h,g} \not \equiv 0$ for $(n,k) = (2,10)$.

Similarly, by using a computer algebraic system,
we can also check $\mathcal{F}_{h,g} \not \equiv 0$ for any $h$ and $g$
for other $(n,k)$ in the lemma.
\end{proof}

\ \\

\begin{flushleft}
S.  Hayashida\\
Department of Mathematics, Joetsu University of Education,\\
1 Yamayashikimachi, Joetsu, Niigata 943-8512, JAPAN\\
e-mail hayasida@juen.ac.jp
\end{flushleft}


\begin{thebibliography}{99}
 \bibitem[Ar 98]{Ar3}
   T.~Arakawa: K\"ocher-Maass Dirichlet Series Corresponding to Jacobi forms and Cohen Eisenstein Series,
     {\it Comment.\ Math.\ Univ.\ St.\ Paul.}\ {\bf 47} No.1 (1998), 93--122.
 \bibitem[Bo 83]{Bo}
   S.~B\"ocherer: 
     \"Uber die Fourier-Jacobi-Entwicklung Siegelscher Eisensteinreihen, 
     {\it Math.\ Z.}\ {\bf 183} (1983), 21--46
 \bibitem[Co 75]{Co}
   H.~Cohen: Sums involving the values at negative integers of $L$-functions of
    quadratic characters, {\it Math.\ Ann.}\ {\bf 217} (1975), 171--185.
 \bibitem[E-Z 85]{EZ}
  M.~Eichler and D.~Zagier: Theory of Jacobi Forms, Progress in
	 Math.\ {\bf 55}, Birkh\"auser, Boston-Basel-Stuttgart, (1985).
\bibitem[H-I 05]{HI}
    S.~Hayashida and T.~Ibukiyama: 
    Siegel modular forms of half integral weights and a lifting conjecture,
    {\it Journal of Kyoto Univ,} {\bf 45} No.3 (2005), 489--530.
\bibitem[H 13]{MaassRe}
    S.~Hayashida:
    On generalized Maass relations and their application to Miyawaki-Ikeda lifts,
    {\it Comment. Math. Univ. St. Pauli}, {\bf 62} No.1 (2013), 59--90.
\bibitem[Ib 92]{Ib}
  T.~Ibukiyama: On Jacobi forms and Siegel modular forms of half
	 integral weights, {\it Comment.\ Math.\ Univ.\ St.\ Paul.}\ {\bf 41} 
	 No.2 (1992), 109--124.
\bibitem[Ik 01]{Ik}
  T.~Ikeda: On the lifting of elliptic cusp forms to Siegel cusp forms of degree 2n, 
         {\it Ann.\ of Math.\ (2)} {\bf 154} no.3, (2001), 641--681.
 \bibitem[Ik 06]{Ik2}
   T.~Ikeda: Pullback of the lifting of elliptic cusp forms and Miyawaki's conjecture,
         {\it Duke Math. J.} {\bf 131}, No.3, (2006), 469--497.
\bibitem[Ka 99]{Katsu}
   H.~Katsurada: An explicit formula for Siegel series.
         {\it Amer. J. Math.} {\bf 121}, No. 2, (1999), 415--452. 
\bibitem[Ko 80]{Ko}
  W.~Kohnen: Modular forms of half integral weight on $\Gamma_{0}(4)$, 
	 {\it Math,\ Ann.}\ {\bf 248} (1980), 249--266.
\bibitem[Ko 02]{Ko2}
  W.~Kohnen: Lifting modular forms of half-integral weight to Siegel 
  	modular forms of even genus, 
  	{\it Math, Ann.}, {\bf 322} (2002), 787--809. 
\bibitem[KK 05]{KK}
  W.~Kohnen and H.~Kojima:
        A Maass space in higher genus,
        {\it Compos. Math.}, {\bf 141}, No. 2, (2005),  313--322. 
\bibitem[Kr 86]{Kr}
   A.~Krieg: Das Vertauschungsgesetz zwischen Hecke-Operatoren und dem
     Siegelschen $\phi$-Opera-tor,
     {\it Arch. Math.} {\bf 46}, No. 4, (1986), 323--329.
\bibitem[OKK 89]{OKK}
    Y.-Y.~Oh, J.-K. Koo and M.-H.~Kim:
    Hecke operators and the Siegel operator,
    {\it J. Korean Math. Soc.} {\bf 26}, No. 2, (1989),  323--334.
\bibitem[Sh 73]{Shi}
  G.~Shimura: On modular forms of half-integral weight,
  {\it Ann. of Math.}\ {\bf 97}, (1973), 440--481.
\bibitem[Ta 86]{Tani}
  Y.~Tanigawa: Modular descent of Siegel modular forms of half integral weight
                and an analogy of the Maass relation,
  {\it Nagoya Math. J.}\ {\bf 102} (1986), 51--77.
 \bibitem[Yn 10]{Yamana}
  S.~Yamana: Maass relations in higher genus,
  {\it Math. Z.}, {\bf 265}, no. 2, (2010), 263--276.
 \bibitem[Yk 86]{Ya}
  T.~Yamazaki: Jacobi forms and a Maass relation for Eisenstein series,  
                        {\it J.\ Fac.\ Sci.\ Univ.\ Tokyo Sect.\ IA, Math.}\ {\bf 33} (1986), 295--310.
 \bibitem[Yk 89]{Ya2}
  T.~Yamazaki: Jacobi forms and a Maass relation for Eisenstein series II,
                        {\it J.\ Fac.\ Sci.\ Univ.\ Tokyo Sect.\ IA, Math.}\ {\bf 36} (1989), 373--386.
 \bibitem[Zh 83]{Zhu:hecke}
  V.~G.~Zhuravlev: Hecke rings for a covering of the symplectic group,
                         {\it Mat.\ Sb.}\ {\bf 121 (163)} (1983), 381--402.
 \bibitem[Zh 84]{Zhu:euler}
  V.~G.~Zhuravlev: Euler expansions of theta transforms of Siegel modular forms
  of half-integral weight and their analytic properties,
                         {\it Mat.\ Sb.}\ {\bf 123 (165)} (1984), 174--194.   
 \bibitem[Zi 89]{Zi} 
  C.~Ziegler: Jacobi forms of higher degree, {\it Abh.\ Math.\ Sem.\ Univ.\  
  Hamburg.}\ {\bf 59} (1989), 191--224.  
\end{thebibliography}
\end{document}